\documentclass[12pt]{ijmart}
\usepackage{amsmath}
\usepackage{latexsym}
\usepackage{amssymb}

\usepackage{verbatim}
\makeatletter
\newcommand{\labitem}[2]{%
\def\@itemlabel{(#1)}
\item
\def\@currentlabel{#1}\label{#2}}
\makeatother
\newtheorem{thm}{Theorem}[section]
\newtheorem{la}[thm]{Lemma}
\newtheorem{Defn}[thm]{Definition}
\newtheorem{Remark}[thm]{Remark}
\newtheorem{Conj}[thm]{Conjecture}
\newtheorem{prop}[thm]{Proposition}
\newtheorem{cor}[thm]{Corollary}
\newtheorem{Example}[thm]{Example}
\newtheorem{Number}[thm]{\!\!}
\newenvironment{defn}{\begin{Defn}\rm}{\end{Defn}}

\newenvironment{rem}{\begin{Remark}\rm}{\end{Remark}}

\newenvironment{numba}{\begin{Number}\rm}{\end{Number}}
\newcommand{\wb}{\overline}

\newcommand{\at}{\symbol{'100}}
\newcommand{\wt}{\widetilde}
\newcommand{\impl}{\Rightarrow}
\newcommand{\mto}{\mapsto}

\newcommand{\N}{{\mathbb N}}
\newcommand{\K}{{\mathbb K}}

\newcommand{\Q}{{\mathbb Q}}
\newcommand{\Z}{{\mathbb Z}}

\newcommand{\cE}{{\mathcal E}}
\newcommand{\cS}{{\mathcal S}}
\newcommand{\cN}{{\mathcal N}}
\newcommand{\cU}{{\mathcal U}}

\newcommand{\cK}{{\mathcal K}}
\newcommand{\cR}{{\mathcal R}}

\newcommand{\cF}{{\mathcal F}}

\newcommand{\sub}{\subseteq}

\DeclareMathOperator{\con}{con}
\DeclareMathOperator{\parb}{par}
\DeclareMathOperator{\bik}{bik}
\DeclareMathOperator{\nub}{nub}
\DeclareMathOperator{\lev}{lev}

\DeclareMathOperator{\COS}{COS}
\DeclareMathOperator{\TID}{TID}
\DeclareMathOperator{\tp}{top}
\begin{document}
\title[Endomorphisms of t.d.l.c. groups]{Contraction Groups and Passage to Subgroups and Quotients for Endomorphisms of Totally Disconnected Locally Compact Groups}
\author[T.P. Bywaters]{Timothy P. Bywaters}
\address{The University of Sydney, School of Mathematics and Statistics \\ NSW 2006, Australia \\ e-mail: {\tt  t.bywaters\at{}maths.usyd.edu.au}}
\author[H. Gl\"{o}ckner]{Helge Gl\"{o}ckner}
\address{Universit\"at Paderborn, Institut f\"{u}r Mathematik \\ Warburger Str.\ 100, 33098 Paderborn, Germany \\ e-mail: {\tt glockner\at{}math.upb.de}}
\author[S. Tornier]{Stephan Tornier}
\address{ETH Z{\"u}rich, Department of Mathematics \\ R{\"a}mistrasse 101, 8092 Z{\"u}rich, Switzerland \\ The University of Newcastle, School of Mathematical and Physical Sciences \\ University Drive, Callaghan, NSW 2308, Australia \\
e-mail: {\tt stephan.tornier@math.ethz.ch}}
\begin{abstract}
The concepts of the scale and tidy subgroups
for an automorphism of a totally disconnected locally compact group
were defined in seminal work by George A. Willis in the 1990s,
and recently generalized to the case of endomorphisms
(G.\,A. Willis, Math.\ Ann.\ {\bf 361} (2015), 403--442).
We show that central facts concerning the scale, tidy subgroups, quotients,
and contraction groups of automorphisms
extend to the case of endomorphisms.
In particular, we obtain results concerning the domain
of attraction around an invariant closed subgroup.\vspace{2mm}
\end{abstract}
%
\maketitle
\newpage
\noindent
{\bf\Large Statement of Results}\\[4mm]
Let $G$ be a Hausdorff topological group with neutral element~$e$
and $\alpha\colon G\to G$ be an endomorphism (which we always assume continuous).
The \emph{contraction group of $\alpha$}
is
\[
\con(\alpha):=\Big\{x\in G\colon\lim_{n\to\infty}\alpha^n(x)=e\Big\}.
\]
Given a closed subgroup $H\sub G$ with $\alpha(H)\sub H$, we say that a sequence $(x_n)_{n\in\N_0}$ in $G$
\emph{converges to~$e$ modulo~$H$} (and we write $x_n\to e$ modulo~$H$)
if, for each identity neighbourhood $V\sub G$,
there exists $N\in\N_0$ such that $x_n\in VH$ for all $n\geq N$
(cf.\ \cite{BaW}, \cite{Jaw}).
The \emph{domain of attraction of $H$ with respect to~$\alpha$}
is defined as the set $\con(\alpha,H)$ of all $x\in G$ such that $\alpha^n(x)\to e$
modulo~$H$.
The following result extends the partial results \cite[Theorem 3.8]{BaW}, \cite[Theorem~1]{Jaw}, \cite[Theorem 2.4]{DaS}
and \cite{Tid}.
\\[4mm]
{\bf Theorem~A.}
\emph{Let $\alpha\colon G\to G$ be an endomorphism of a totally disconnected, locally compact group~$G$
and $H\sub G$ be a closed subgroup.
If $\alpha(H)=H$ or $\alpha(H)\sub H$ and~$H$ is compact, then
$\con(\alpha,H)=\con(\alpha)H$.}\\[4mm]
Given $x\in G$, a sequence $(x_n)_{n\in\N_0}$ in~$G$
is called an \emph{$\alpha$-regressive trajectory} for~$x$ if $x_0=x$
and $\alpha(x_n)=x_{n-1}$ for all $n\in\N$.
Let
$\con^-(\alpha)$ be the set of all
$x\in G$ for which there exists an $\alpha$-regressive trajectory $(x_n)_{n\in\N_0}$
with
\[
\lim_{n\to\infty} x_n=e.\vspace{-.3mm}
\]
Then $\con^-(\alpha)$ is a subgroup of~$G$, called the \emph{anti-contraction group} of~$\alpha$.
If~$H$ is a closed subgroup of~$G$ with $\alpha(H)\sub H$,
then the \emph{anti-contraction set of $\alpha$ modulo~$H$}
is defined as the set $\con^-(\alpha,H)$ of all $x\in G$
for which there exists an $\alpha$-regressive trajectory $(x_n)_{n\in\N_0}$
with $x_n\to e$ modulo~$H$.
\\[4mm]
{\bf Theorem~B.}
\emph{Let $\alpha\colon G\to G$ be an endomorphism of a totally disconnected, locally compact group $G$ and let $H\!\sub\! G$ be a closed subgroup with $\alpha(H)\!\sub\! H$.~Then}
\[
\con^-(\alpha,H)=\con^-(\alpha)H.
\]
\emph{For each $\alpha$-regressive trajectory $(x_n)_{n}$ such that $x_n\to e$ modulo~$H$,
there is an $\alpha$-regressive trajectory $(y_n)_{n}$ such that $y_n\in x_n H$
for all $n$ and $\lim_{n\to\infty}y_n = e$.}\\[4mm]
Let $G$ be a totally disconnected, locally compact group and $\alpha\colon G\to G$ be
an endomorphism.
For a compact open subgroup $U\sub G$, the index
\[
[\alpha(U):\alpha(U)\cap U]\in\N
\]
is called the \emph{displacement index} of $U$.
If $U$ has minimal displacement index among all compact open subgroups,
then~$U$ is called \emph{minimizing}, and one defines the \emph{scale} of~$\alpha$
as this minimum displacement index, denoted $s(\alpha)$. Equivalently, $U$ has certain structural properties, summarized as being \emph{tidy}
Let $\parb^-(\alpha)$ be the \emph{anti-parabolic subgroup} of~$\alpha$
consisting of all $x\in G$ for which there exists an $\alpha$-regressive
trajectory $(x_n)_{n\in\N_0}$ such that $\{x_n\colon n\in \N_0\}$ is relatively compact.
\\
The article \cite{Fur} analyzes how the scale of an automorphism $\alpha$ behaves with respect to taking subgroups and quotients. We generalize these results to the case of endomorphisms. Given a group $G$, a normal subgroup $H$ of $G$ and an endomorphism $\alpha$ of $G$
with $\alpha(H)\subseteq H$, we let $\overline{\alpha}$ denote the associated endomorphism of $G/H$. \\[4mm]
{\bf Theorem~C.}
\emph{Let $\alpha\colon G\to G$ be an endomorphism of a totally disconnected, locally compact group $G$ and $H\sub G$ a closed subgroup with $\alpha(H)\subseteq H$.}
\begin{itemize}
\labitem{a}{itm:scale_subgroup} \emph{There is a compact open subgroup $U$ of $G$ which is tidy for $\alpha$ and such that $U\cap H$ is tidy for $\alpha|_{H}$. Furthermore}
\begin{displaymath}
s_{H}(\alpha|_{H})\le s_{G}(\alpha).
\end{displaymath}
\labitem{b}{itm:scale_divide} \emph{If, in addition, $H$ is normal in $G$, then}
\begin{displaymath}
  s_{H}(\alpha|_{H})s_{G/H}(\overline{\alpha}) \text{ \emph{divides} } s_{G}(\alpha).
 \end{displaymath}
\labitem{c}{itm:scale_equality} \emph{If, in addition to all the above, $H\subseteq\parb^{-}(\alpha)$
is a closed normal subgroup of~$G$ and $\alpha(H)=H$, then}
\begin{displaymath}
 s_{H}(\alpha|_{H})s_{G/H}(\overline{\alpha})=s_{G}(\alpha).
\end{displaymath}
\end{itemize}
\vspace{0.4cm}\noindent
Using Theorem~A, we show the following result which generalizes \cite[Theorem~3.32]{BaW} and the corresponding statement in \cite{Lie}.\\[2.3mm]
{\bf Theorem~D.}
\emph{Let $\alpha\colon G\to G$ be an endomorphism of a totally disconnected,
locally compact group~$G$.
Then $\alpha$ has small tidy subgroups
if and only if~$\con(\alpha)$ is closed in~$G$.}\\[4mm]
\\[2.3mm]
Using Theorem~B, we prove a generalization of \cite[Proposition~3.21\,(3)]{BaW}:\\[4mm]
{\bf Theorem~E.}
\emph{Let $\alpha\colon G\to G$ be an endomorphism of a totally disconnected,
locally compact group~$G$.
Then}
\begin{equation}\label{oncon}
s(\alpha)=s\big( \alpha|_{\wb{\con^-(\alpha)}}\big).\vspace{1mm}
\end{equation}
Further consequences of Theorems~A and~B are obtained.
Let $\parb(\alpha)\sub G$ be the \emph{parabolic subgroup}
associated to an endomorphism~$\alpha$, i.e., the set of all
$x\in G$ for which $\{\alpha^n(x)\colon n\in\N_0\}$ is relatively compact in~$G$.
Let
\[
\lev(\alpha):=\parb(\alpha)\cap \parb^-(\alpha)
\]
be the \emph{Levi subgroup}.
If $G$ is locally compact and totally disconnected, then $\parb(\alpha)$ and
$\lev(\alpha)$ are closed subgroups of~$G$ such that $\alpha(\parb(\alpha))\sub \parb(\alpha)$
and $\alpha(\lev(\alpha))=\lev(\alpha)$ (see \cite[Proposition~19]{End}).
Given a compact open subgroup $V\sub G$, we define
$V_-:=\bigcap_{n\in\N_0}\alpha^{-n}(V)$
and write $V_+$ for the set of all $x\in V$ for which there exists an
$\alpha$-regressive trajectory $(x_n)_{n\in\N_0}$ within~$V$.\\[2.3mm]
We study $\parb(\alpha)$, $\lev(\alpha)$ and their connections to tidy subgroups.
If $\alpha$ has small tidy subgroups, then the strongest conclusions
can be obtained, as compiled in the next theorem.
For $\alpha$ an automorphism, cf.\
already~\cite{Tid} for part~(a) and~\cite[Theorem 3.32]{BaW} for~(b).
The semidirect products mean semidirect products of topological groups.\\[4mm]
{\bf Theorem~F.} \emph{Let $\alpha$ be an endomorphism of a totally disconnected, locally
compact group~$G$. Then $\alpha$ has small tidy subgroups if and only if $\alpha|_{\lev(\alpha)}$
has small tidy subgroups. In this case, the following holds}:
\begin{itemize}
\item[(a)]
\emph{$\Omega := \con(\alpha)\lev(\alpha)\con^-(\alpha)$ is an open identity neighbourhood in~$G$
such that $\alpha(\Omega)\sub\Omega$, and the product map}
\begin{equation}\label{prodmap}
\pi\colon \con(\alpha)\times \lev(\alpha)\times \con^-(\alpha)\to\Omega,\;\; (x,y,z)\mto xyz
\end{equation}
\emph{is a homeomorphism.}
\item[(b)]
\emph{$\parb(\alpha)=\con(\alpha)\rtimes \lev(\alpha)$ and $\parb^-(\alpha)=\con^-(\alpha)\rtimes \lev(\alpha)$.}
\item[(c)]
\emph{$\alpha|_{\parb^-(\alpha)}$, $\alpha|_{\lev(\alpha)}$ and $\alpha|_{\con^-(\alpha)}$ are automorphisms.}
\item[(d)]
\emph{Every subgroup $V$ of $G$ which is tidy for $\alpha$ is a subset of~$\Omega$;
it satisfies}
\[
V_-=(\con(\alpha)\cap V)\rtimes (\lev(\alpha)\cap V)
\]
\emph{and $V_+=(\con^-(\alpha)\cap V)\rtimes (\lev(\alpha)\cap V)$.}
\item[(e)]
\emph{A compact open subgroup $V\sub G$ is tidy for $\alpha$ if and only if}
\begin{equation}\label{chart1}
V=(\con(\alpha)\cap V)(\lev(\alpha)\cap V)(\con^-(\alpha)\cap V),\vspace{-3mm}
\end{equation}
\begin{eqnarray}\alpha(\con(\alpha)\cap V)\, &\sub & \con(\alpha)\cap V,\label{chart2}\\
\alpha(\lev(\alpha)\cap V)\,\, &=& \lev(\alpha)\cap V,\;\;\;\mbox{and}\label{chart3}\\
\alpha(\con^-(\alpha)\cap V)&\supseteq &
\con^-(\alpha)\cap V.\label{chart4}
\end{eqnarray}
\item[(f)]
\emph{The compact open subgroups $W\sub \lev(\alpha)$ with $\alpha(W)=W$ form a basis of identity neigbourhoods
in~$\lev(\alpha)$.}\vspace{2mm}
\end{itemize}
If $G$ is a Lie group over a
totally disconnected local field
(as in \cite{Bou} and \cite{Ser})
and $\alpha\colon G\to G$ is an analytic endomorphism with small tidy subgroups,
then $\con(\alpha)$, $\lev(\alpha)$ and $\con^-(\alpha)$ are Lie subgroups of~$G$
(in the strong sense of submanifolds) and the product map in~(\ref{prodmap})
is an analytic diffeomorphism, see \cite{Lie}.
\\[2.3mm]
By (e) in Theorem~F, the automorphism $\alpha|_{\lev(\alpha)}$ is distal, see \cite{SaR}.
Information concerning contractive automorphisms of locally
compact groups can be found in \cite{Sie} and \cite{GaW};
contractive analytic automorphisms of Lie groups over a totally disconnected
local field~$\K$
are discussed in \cite{Wan} (for $\K=\Q_p$) and \cite{CON}.\\[2.3mm]
Independenly, studies of endomorphisms of totally disconnected groups have also been performed in \cite{Rei} and \cite{GBV}. Whereas the former deals with profinite groups, the latter has a view towards topological entropy.
According to \cite[Corollary~4.11]{GBV}
(combined with Proposition~\ref{betterC}),
the topological entropy $h_{\tp}(\alpha)$
of an endomorphism $\alpha$ of a totally disconnected,
locally compact group~$G$ is given by $h_{\tp}(\alpha)=
\ln s(\alpha)$,
if $\con(\alpha)$ is closed.
Our Theorem~C\,(c) therefore implies the following Addition Theorem
for topological entropy,
which is complementary to a similar, known result (\cite[Theorem 1.2]{GBV}):\\[2.3mm]
\emph{If $\con(\alpha)$ is closed in~$G$ and
$H$ is a closed normal subgroup of~$G$ such that
$H\sub \parb^-(G)$ and $\alpha(H)=H$, then
$h_{\tp}(\alpha)=h_{\tp}(\alpha|_H)+h_{\tp}(\wb{\alpha})$.}\\[2.3mm]  
{\bf Acknowledgements.} The research was carried out during the `Winter of Disconnectedness'
hosted by the University of Melbourne (Matrix Center, Creswick)
and the University of Newcastle (NSW). We owe thanks to both institutions for their hospitality
and are grateful to George Willis for stimulating questions, discussions, and support. The first and last authors would like to extend their gratitude to Colin Reid for helpful discussions, as well as their supervisors, Jacqui Ramagge and Marc Burger respectively, for their ongoing support and making attendance to the aforementioned workshop possible. This research was conducted while the first author was supported by the Australian Government Research Training Program Scholarship and, partly, while the last author was supported by the SNSF Doc.Mobility fellowship 172120.
Finally, we owe thanks to an anonymous referee whose comments helped us to improve the original manuscript.\\[2.3mm]
{\bf Conventions.} Write $\N:=\{1,2,\ldots\}$ and $\N_0:=\N\cup\{0\}$.
All topological groups and uniform spaces are assumed Hausdorff.
The automorphisms considered are continuous, with continuous inverse.
A map $q\colon X\to Y$ between topological spaces is called a \emph{quotient map} if it is surjective
and $Y$ carries the quotient topology (the final topology with respect to~$q$).
If $X$ is a set and $\alpha\colon X\to X$ a map, we
say that a subset $M\sub X$ is \emph{$\alpha$-invariant}
if $\alpha(M)\sub M$. If $\alpha(M)=M$, we say that~$M$ is \emph{$\alpha$-stable}.
This terminology is in line with \cite{BaW} but differs from \cite{Jaw}.
A sequence $(x_n)_{n\in\N_0}$ (or two-sided sequence $(x_n)_{n\in\Z}$)
in a topological space~$X$ is called \emph{bounded}
if $\{x_n\colon n\in\N_0\}$ (resp., $\{x_n\colon n\in\Z\}$)
is relatively compact in~$X$.
If $G$ is a totally disconnected, locally compact group,
then $\COS(G)$ denotes the set of all compact open subgroups of~$G$.
We write $|X|$ for the cardinality of a set~$X$. If $k$ is an ordinal number, we write $[0,k[$ for the set of ordinals $j<k$.
For each cardinal number~$c$, the set of ordinal numbers with cardinality~$c$
has a minimal element~$k$. Thus $[0,k[$ has cardinality~$c$,
while $[0,j[$ has cardinality less than $c$ for each $j\in[0,k[$.
We identify~$c$ with~$k$.
\\[2.3mm]
{\bf Structure of the Article.}
In a preparatory first section, we recall necessary concepts from~\cite{End}.
Auxiliary results concerning convergence
modulo~$H$ and generalizations of the concept for self-maps of uniform spaces
are compiled in Section~\ref{secuni}.
Section~\ref{sec1} is devoted to the proof of
Theorem~A for metrizable~$G$, by an adaptation of arguments from~\cite{BaW}.
In Section~\ref{sec4}, we prove Theorem~B
for metrizable~$G$.
Also this proof was inspired by \cite{BaW}, but major changes
were necessary to enable (unique) $\alpha$-orbits to be replaced with (not necessarily unique) $\alpha$-regressive
trajectories.
In both cases, the generalization to non-metrizable~$G$ uses
a method that was stimulated by Jaworski's arguments in~\cite{Jaw},
but deviates in detail as we use a different cardinal as the parameter
in a transfinite induction. In this way, we can do without some of the preparatory lemmas in~\cite{Jaw},
which do not have analogues for anti-contraction groups.
See Section~\ref{sec3} for the case of Theorem~A.
The relevant cardinal invariant is introduced in Section~\ref{sec2}.
As to Theorem~B, we have to line up regressive trajectories
by a transfinite induction over the steps of a projective system,
rather than
mere
elements of anti-contraction groups (Section~\ref{proofB}).
This also enables us to get around the problem
that although $\con(\alpha,K)\cap \con(\alpha,L)=\con(\alpha,K\cap L)$
for compact $K$ and~$L$, a corresponding formula cannot be expected for
anti-contraction groups of an endomorphism.
Section~\ref{sec:scale_sub_quot} analyzes how the scale of an endomorphism behaves with respect to subgroups and quotients and proves most of Theorem~C. 
In Sections~\ref{sec5}--\ref{sec-cell},
we give applications of Theorems~A and~B
concerning endomorphisms, their scale and the various subgroups going along with them.
Notably, we complete the proof of Theorem~C and obtain proofs for
Theorems~D, ~E, and~F.

\section{Basic concepts}\label{sec:basic_concepts}

In this section, we collect basic definitions and properties of the objects appearing in this article. For details, we refer the reader to \cite{End}. Let $G$ be a totally disconnected, locally compact group and $\alpha\colon G\to G$ be
an endomorphism.
Recall that by a theorem of van Dantzig (see \cite{vDa} or \cite[(7.7)]{HaR}),
$G$ admits an identity neighbourhood basis consisting of compact open subgroups.
For a compact open subgroup $U\sub G$, the index
\[
[\alpha(U):\alpha(U)\cap U]\in\N
\]
is called the \emph{displacement index} of $U$.
If $U$ has minimal displacement index among all compact open subgroups,
then~$U$ is called \emph{minimizing}, and one defines the \emph{scale} of~$\alpha$
as this minimum displacement index, denoted $s(\alpha)$
or $s_G(\alpha)$.
Equivalently, $U$ has certain structural properties,
summarized as being \emph{tidy}.
Write
\[
U_-:=\bigcap_{n\in\N_0}\alpha^{-n}(U)=\{x\in U\colon (\forall n\in\N_0)\;\alpha^n(x)\in U\},\vspace{-1mm}
\]
where $\alpha^{-n}(U)$ means the preimage $(\alpha^n)^{-1}(U)$.
Let $U_+$ be the set of all $x\in U$ for which there exists an
$\alpha$-regressive trajectory $(x_n)_{n\in\N_0}$ such that $x_n\in U$ for all~$n$.
Then
\[
U_+=\bigcap_{n\in\N_0}U_{n,\alpha}\quad\mbox{with}\vspace{-1mm}
\]
\begin{equation}\label{defUn}
U_{0,\alpha}:=U\;\;\mbox{and}\;\; U_{n+1,\alpha}:=U\cap\alpha(U_{n,\alpha})\;\;\mbox{for $n\in\N_0$;}
\end{equation}
moreover, $U_+$ and $U_-$ are compact subgroups of~$G$
such that
\[
\alpha(U_-)\sub U_-\quad\mbox{and}\quad \alpha(U_+)\supseteq U_+
\]
(see \cite{End}, or also Lemma~\ref{good-prepar}).
The sets
\[
U_{--}:=\bigcup_{n\in\N_0}\alpha^{-n}(U_-)\quad\mbox{and}\quad
U_{++}:=\bigcup_{n\in\N_0} \alpha^n(U_+)
\]
are unions of ascending sequences of subgoups, whence they are subgroups of~$G$.
If $U=U_+U_-$, then $U$ is called \emph{tidy above} for~$\alpha$.
If $U_{++}$ is closed in~$G$
and
\[
[\alpha^{n+1}(U_+):\alpha^n(U_+)]\in\N
\]
is independent of $n\in\N_0$,
then $U$ is called \emph{tidy below}.
If $U$ is both tidy above and tidy below for~$\alpha$, then $U$ is called \emph{tidy for $\alpha$} (see \cite{End}).\\[2.3mm]
We shall use the fact that $U$ is tidy if and only if~$U$ is tidy above
and $U_{--}$ is closed in~$G$ (see \cite[Proposition~9]{End}).\\[2.3mm]
By \cite[Lemma 5 and Theorem 2]{End}, a compact open subgroup
$U$ of~$G$ is minimizing for~$\alpha$ if and only if it is tidy for~$\alpha$, in which case
\[
s(\alpha)=[\alpha(U):\alpha(U)\cap U]=[\alpha(U_+) : U_+].
\]
If $U$ is a compact open subgroup of~$G$, then the compact open subgroup
\[
\bigcap_{j=0}^\ell \alpha^{-j}(U)
\]
is tidy above for large $\ell \in\N_0$ (see \cite{End}).
For $U\sub G$ a compact open subgroup, $U_+\cap \, U_-$ is a compact subgroup of~$G$ such that $\alpha(U_+\cap  U_-)=U_+\cap U_-$
(see \cite{End} or Lemma~\ref{good-prepar}).
Let $\parb^-(\alpha)$ be the \emph{anti-parabolic subgroup} of~$\alpha$
consisting of all $x\in G$ for which there is an $\alpha$-regressive
trajectory $(x_n)_{n\in\N_0}$ with $\{x_n\colon n\in \N_0\}$ being relatively compact.
Then $\parb^-(\alpha)$ is a closed subgroup of~$G$ and
\[
\alpha(\parb^-(\alpha))=\parb^-( \alpha)
\]
(see \cite[comment after Proposition~19]{End}).
The intersection $\nub(\alpha)$
of all tidy subgroups is called the \emph{nub} of $\alpha$.
The nub is compact, and \emph{$\alpha$-stable}
in the sense that $\alpha(\nub(\alpha))=\nub(\alpha)$
(see \cite[Proposition~20]{End}).
The nub coincides with the trivial group~$\{e\}$
if and only if \emph{$\alpha$ has small tidy subgroups}
in the sense that the subgroups which are tidy for~$\alpha$ form a basis
of identity neighbourhoods.
The \emph{bounded iterated kernel} $\bik(\alpha)$
is defined as the closure
\[
\bik(\alpha):=\wb{\{x\in \parb^-(\alpha):\alpha^n(x)=e\;\mbox{for some $n\in \N_0$}\}}
\]
(as opposed to \cite[Definition~12]{End} which is neither intended nor used);
it is an $\alpha$-stable compact normal subgroup of $\parb^-(\alpha)$.
The endomorphism~$\wb{\alpha}$ of $\parb^-(\alpha)/\bik(\alpha)$ induced by~$\alpha$
is an automorphism and $\bik(\alpha)\sub \nub(\alpha)$ \cite[\!Proposition~20]{End}.
Then $\nub(\alpha)/\!\bik(\alpha)=\nub(\wb{\alpha})$ (see Proposition~\ref{keynub}\,(b))
and
$\nub(\wb{\alpha})$ is the largest $\wb{\alpha}$-stable closed
subgroup of $\parb^-(\alpha)/\bik(\alpha)$ on which $\wb{\alpha}$
acts ergodically (see \cite[Proposition~4.4]{Nub}).\\[2.3mm]
For $\alpha$ an endomorphism of a totally disconnected locally compact group,
the following sets were already defined above.
\begin{defn}\label{defKpm}
If $X$ is a Hausdorff topological space,
$K\sub X$ a compact subset and
$\alpha\colon X\to X$
a continuous map,
we define
\[
K_-:=\bigcap_{n\in\N_0} \alpha^{-n}(K)\qquad
\mbox{and}\qquad
K_+:=\bigcap_{n\in\N_0}K_{n,\alpha},
\]
where $K_{0,\alpha}:=K$ and $K_{n,\alpha}:=K\cap \alpha(K_{n-1,\alpha})$ for $n\in\N$.
\end{defn}
The following properties are useful in general. Compare \cite{End} for proofs.
\begin{la}\label{good-prepar}
For $X$, $\alpha$, $K$, $K_{n,\alpha}$, $K_+$, $K_-$ as in Definition~\emph{\ref{defKpm}},
we have:
\begin{itemize}
\item[\rm(a)]
$K_-$ is a compact subset of~$K$.
It is the set of all $x\in K$ such that $\alpha^n(x)\in K$ for all $n\in\N_0$.
Moreover, $\alpha(K_-)\sub K_-$.
\item[\rm(b)]
$K_{n,\alpha}$ is a compact subset of~$K$, for each $n\in\N_0$.
It is the set of all $x\in K$ for which there exist $x_0,x_1,\ldots,x_n\in K$
such that $x_0=x$ and $\alpha(x_j)=x_{j-1}$ for all $j\in\{1,2,\ldots, n\}$.
\item[\rm(c)]
$K_+$ is a compact subset of~$K$. It is the set of all $x\in K$
for which there exists an $\alpha$-regressive trajectory $(x_n)_{n\in\N_0}$
with $x_0=x$ and $x_n\in K$ for all $n\in\N_0$. Moreover, $K_+\sub \alpha(K_+)$.
\item[\rm(d)]
$K_+\cap K_-$ is a compact subset of~$K$. It is the set of all $x\in K$
for which there exists a family $(x_n)_{n\in\Z}$ of elements $x_n\in K$
such that $x_0=x$ and $\alpha(x_n)=x_{n+1}$ for all $n\in\Z$.
Moreover, $\alpha(K_+\cap K_-)=K_+\cap K_-$.
\end{itemize}
If $X$ is a topological group, $\alpha$ an endomorphism
and $K$ a compact subgroup of~$X$, then also $K_+$, $K_-$ and each $K_{n,\alpha}$
are compact subgroups of~$X$.
\end{la}
\section{Dynamics on uniform spaces}\label{secuni}
In this section, we collect generalities about dynamical systems
on uniform spaces. As a reference for the
theory of uniform spaces, we recommend \cite{Isb}.
A simple topological fact will be used (whose proof is left to the reader).
\begin{la}\label{la:filter_basis}
Let $X$ be a topological space and $\cN$ a filter basis in~$X$.
If each $K\in\cN$ is compact and $U\sub X$ is an open subset with $\bigcap \cN \sub U$,
then there exists some $K\in\cN$ such that $K\sub U$.
\end{la}
\begin{numba}
Let $X$ be a set, $\alpha\colon X\to X$ be a map
and $x\in X$.
As usual, the sequence $(\alpha^n(x))_{n\in\N_0}$ is called
the \emph{$\alpha$-orbit} of~$x$.
A sequence $(x_n)_{n\in\N_0}$ is called
an \emph{$\alpha$-regressive trajectory}
for~$x$ if $x_0=x$ and $\alpha(x_n)=x_{n-1}$ for all $n\in\N$. Note the slightly different usage in \cite{End}.
A two-sided sequence $(x_n)_{n\in\Z}$ is called an \emph{$\alpha$-trajectory}
(for $x_0$) if $\alpha(x_n)=x_{n+1}$ for all $n\in\Z$.
\end{numba}
The following observation allows
$\alpha$-regressive trajectories to be created from a given sequence.
\begin{la}\label{onekey}
Let $X$ be a Hausdorff topological space, $\alpha\colon X\to X$
be a continuous map and $(z_n)_{n\in\N_0}$ be a sequence in~$X$.
Let $x_0\in X$ and, for $n,m\in\N_0$,
\begin{equation}\label{doubseq}
z^{(m)}_n:=\left\{
\begin{array}{cl}
\alpha^{m-n}(z_m) &\mbox{if $n\in\{0,\ldots,m\}${\rm ;}}\\
x_0 & \mbox{if $n>m$.}
\end{array}
\right.
\end{equation}
Thus $z^{(m)}:=(z^{(m)}_n)_{n\in\N_0}\in X^{\N_0}$ for each $m\in\N_0$.
If $y=(y_n)_{n\in\N_0}$ is an accumulation point of $(z^{(m)})_{m\in\N_0}$
in $X^{\N_0}$ with respect to the product topology,
then $(y_n)_{n\in\N_0}$ is an $\alpha$-regressive trajectory.
\end{la}
\begin{proof}
Let $(z^{(m_a)})_{a\in A}$ be a convergent subnet of $(z^{(m)})_{m\in\N_0}$
with limit $y=(y_n)_{n\in\N_0}$.
For every $n\in \N$, there is $a_0\in A$ with $m_a\geq n$ for all $a\geq a_0$
and thus
\[
\alpha(z^{(m_a)}_n)=z^{(m_a)}_{n-1}.
\]
Since $\alpha$ is continuous, passing to the limit we obtain that $\alpha(y_n)=y_{n-1}$.
Hence $(y_n)_{n\in\N_0}$ is an $\alpha$-regressive trajectory.
\end{proof}
For fixed~$n\in\N_0$, we have $z^{(m)}_n=\alpha^{m-n}(z_m)$ for all $m\geq n$.
Therefore, accumulation points of $(z^{(m)})_{m\in\N_0}$
are independent of the choice of~$x_0$.
\begin{defn}\label{defdoubseq}
We call $(z^{(m)})_{m\in\N_0}=((z^{(m)}_n)_{n\in\N_0})_{m\in\N_0}$
(as in (\ref{doubseq})) a \emph{double sequence associated to~$(z_n)_{n\in\N_0}$
with respect to~$\alpha$.}
If $X$ is a topological group, we shall always choose $x_0:=e$.
\end{defn}
\begin{numba}
Let $(X,\cE)$ be a uniform space,
with filter $\cE$ of entourages.
For $V\in\cE$, $x\in X$ and $M\sub X$, we write
$V[x]:=\{y\in X\colon (x,y)\in V\}$
and $V[M]:=\bigcup_{z\in M}V[z]$.
\end{numba}
\begin{numba}
Let $(X,\cE)$ be a uniform space and $S\sub X$.
Given a subset $S\sub X$,
we say that a sequence $(x_n)_{n\in\N_0}$ in $X$
\emph{converges to~$S$} with respect to~$\cE$
(and write $x_n\to_\cE S$)
if, for each $V\in \cE$,
there exists $N\in\N_0$ such that $x_n\in V[S]$ for all $n\geq N$.
Further, let $\alpha\colon X\to X$ be a continuous map.
The \emph{domain of attraction of $S$ with respect to~$\alpha$}
is defined as
\[
\con(\alpha,S):=\big\{ x\in X\colon \alpha^n(x)\to_\cE S\big\}.
\]
Given a subset $S\sub X$, we write $\con^-(\alpha,S)$
for the set of all $x\in X$ for which there exists an $\alpha$-regressive
trajectory $(x_n)_{n\in\N_0}$ such that $x_n\to_\cE S$.
\end{numba}
\begin{numba}
Recall that the \emph{right uniform structure} $\cR$ on
a topological group~$G$ is the filter generated by the filter basis of the sets
\[
\wt{V}:=\{(x,y)\in G\times G\colon y\in Vx\},
\]
for $V$ ranging through the set of identity neighbourhoods in~$G$
(see \cite[(4.11)]{HaR}).
Then $\wt{V}[x]=Vx$ and $\wt{V}[M]=VM$ for subsets $M\sub G$.
Thus, if $H\sub G$ is a closed subgroup and $(x_n)_{n\in\N_0}$
is a sequence in~$G$, then
\[
\mbox{$x_n\to e$ modulo~$H\, $ if and only if $\, x_n\to_\cR H$.}
\]
If $\alpha$ is an endomorphism of~$G$ with $\alpha(H)\sub H$,
we deduce that the sets $\con(\alpha,H)$ and $\con^-(\alpha,H)$
coincide with those obtained for
the uniform space $(G,\cR)$.
\end{numba}
\begin{numba}
Recall that a subset $R$ of a uniform space $(X,\cE)$
is called \emph{precompact} if, for each $V\in\cE$,
there exist $n\in\N_0$ and $x_1,\ldots, x_n$ in $X$ (or, equivalently, in~$R$)~with
\[
R\sub V[x_1]\cup\cdots\cup V[x_n].
\]
Every compact subset $K\sub X$ is precompact.
If $(X,\cE)$ is complete (e.g., in the case of $(G,\cR)$ with
$G$ a locally compact group,
see \cite[\S3.3, Corollary 1]{BouTop}),
then a subset $R \sub G$ is precompact if and only if it is relatively
compact (cf.\ \cite[p.\,22]{Isb}).
\end{numba}
The following variant of \cite[Lemma~3.9]{BaW} (devoted to locally compact groups)
is very useful for our purposes; \cite{Jaw} already applied the latter lemma also
for Hausdorff topological groups (without proof).
\begin{la}\label{specialkaes}
Let $(x_n)_{n\in\N_0}$ be a sequence in a uniform space $(X,\cE)$.
\begin{itemize}
\item[\rm(a)]
If $x_n\to_\cE K$ for some precompact subset
$K\sub X$, then $\{x_n\colon n\in\N_0\}=:R$
is precompact in~$X$. If $x_n\to_\cE K$ and $K$ is compact,
then~$R$ is relatively compact and thus $(x_n)_{n\in\N_0}$ is bounded.
\item[\rm(b)]
If $x_n\to_\cE S$ for a closed subset $S\sub X$,
then each accumulation point of $(x_n)_{n\in \N_0}$ is contained in~$S$.
\item[\rm(c)]
If $(x_n)_{n\in\N_0}$ is bounded and $S\sub X$ is
closed, then $x_n\to_\cE S$ if and only if each accumulation point of $(x_n)_{n\in\N_0}$
lies in~$S$.
\end{itemize}
\end{la}
\begin{proof}
(a) If $V\in\cE$,
choose $W\in\cE$ with $W\circ W\sub V$.
Since $K$ is precompact, there is a finite subset $\Phi\sub K$ such that $K\sub W[\Phi]$.
Hence $W[K]\sub W[W[\Phi]]=(W\circ W)[\Phi]\sub V[\Phi]$.
There exists $N\in\N_0$ such that $x_n\in W[K]$ for all $n\geq N$.
Hence $R:=\{x_n\colon n\in\N_0\}$ is contained in the finite union
\[
V[\Phi] \cup\bigcup_{n=0}^{N-1} V[x_n],
\]
and thus $R$ is precompact.\\[2.3mm]
If, moreover, $K$ is compact, let us show that the closure $\wb{R}$ is complete
in the uniform structure induced by $(X,\cE)$.
Then~$\wb{R}$, being a uniform space which is both precompact
and complete, will be compact (see, e.g., \cite[p.\,22]{Isb}).
Let $y\in \wb{R}$ such that $y\not\in R$; we show that $y\in K$.
Let $(y_j)_{j\in J}$ be a net in~$R$ which converges to~$y$.
Given $U\in \cE$, pick $V\in\cE$ with $V\circ V\sub U$.
There is $N\in\N$ with
\begin{equation}\label{reu1}
x_n\in V[K]\quad\mbox{for all $\, n\geq N$.}
\end{equation}
Since none of $x_0,\ldots, x_{N-1}$ is an accumulation point of $(y_j)_{j\in J}$,
we find $j_V\in J$ such that $y_j\not\in \{x_0,\ldots, x_{N-1}\}$ for all $j\geq j_V$
and thus
\begin{equation}\label{reu2}
y_j\in V[K]\quad\mbox{for all $\,j\geq j_V$.}
\end{equation}
Since $y_j\to y$, we deduce that $y\in \wb{V[K]}\sub V[V[K]]\sub U[K]$.
As $\bigcap_{U\in\cE} U[K]=\wb{K}=K$, we see that $y\in K$.\\[2.3mm]
If $\wb{R}$ was not complete, we could
find a Cauchy net
$(y_j)_{j\in J}$ in $\wb{R}$ which does not converge in~$\wb{R}$,
and which therefore does not have a convergent subnet (nor an accumulation point).
Make $\cE\times J$ a directed set by declaring $(V_1,j_1)\leq (V_2,j_2)$ if
$V_1\supseteq V_2$ and $j_1\leq j_2$.
As before, for $U\in\cE$ and $V\in \cE$ with $V\circ V\sub U$, we can pick
$N\in\N$ with (\ref{reu1})
and find $j_V\in J$ such that (\ref{reu2}) holds.
Hence, with the induced order,
\[
A:=\{(V,j)\in \cE \times J\colon y_j\in V[K]\}
\]
is a directed set, and $(y_j)_{(V,j)\in A}$ is a subnet of $(y_j)_{j\in J}$.
For $a=(V,j)\in A$,
we find $k_a\in K$ such that
\begin{equation}\label{will-revert}
y_j\in V[k_a].
\end{equation}
Since~$K$ is compact, the net $(k_a)_{a\in A}$ has a subnet $(k_{a(i)})_{i\in I}$ which
converges to some $k\in K$.
Write $a(i)=(V(i),j(i))$. 
Given an entourage $U\in \cE$, pick $V\in\cE$ such that $V\circ V\sub U$.
There exists $i_0\in I$ such that
\[
k_{a(i)}\in V[k]
\]
and $V(i)\sub V$ for all $i\ge i_0$ in~$I$.
Hence, using (\ref{will-revert}),
\[
y_{j(i)}\in V(i)[k_{a(i)}]\sub V[k_{a(i)}]\sub V[V[k]]\sub U[k]
\]
for all $i\geq i_0$.
The subnet $(y_{j(i)})_{i\in I}$ of $(y_j)_{j\in J}$ therefore
converges to~$k$,
contrary to our assumption. Thus $(y_j)_{j\in J}$
cannot exist and $\wb{R}$ must be complete.\\[2.3mm]
(b) Assume that a subnet $(x_{n(j)})_{j\in J}$ of the given sequence converges to~$x\in S$.
For all $U\in\cE$ and $V\in\cE$ with $V\circ V\sub U$,
we find $j_0\in J$ such that $x_{n(j)}\in V[x]$ for all $j\geq j_0$ and thus
\[
x\in\wb{V[x]}\sub (V\circ V)[x]\sub U[x]\sub U[S].
\]
Hence $x\in\bigcap_{U\in\cE} U[S]=\wb{S}=S$.\\[2.3mm]
(c) If not $x_n\to_\cE S$, then there exists $V\in \cE$ such that $J:=\{n\in\N_0\colon
x_n\not\in V[S]\}$
is an infinite set. After shrinking~$V$, we may assume that $V[S]$ is open.
Now $P:=\{x_n\colon n\in J\}$ has compact closure~$L$ in~$X$.
Since $P$ is a subset of the closed set $X\setminus V[S]$, also $L\sub X\setminus V[S]$.
Now~$L$ being compact, the subnet $(x_n)_{n\in J}$ of $(x_n)_{n \in\N_0}$
has a subnet which converges to some $c\in L$. Then $c$ is an accumulation point of $(x_n)_{n\in\N_0}$
such that $c\not\in V[S]$ and hence $c\not\in S$.
Together with~(b), this establishes~(c).
\end{proof}
\begin{la}\label{intersect-fam-seq}
Let $(X,\cE)$ be a uniform space
and $(A_j)_{j\in J}$ be a family of closed subsets of~$X$, with $J\not=\emptyset$.
Let $(x_n)_{n\in\N_0}$ be a sequence in~$X$ such that $x_n\to_\cE A_j$
for each $j\in J$.
If $A_{j_0}$ is compact for some $j_0\in J$, then
\begin{equation}\label{modu-inters}
x_n\; \to_\cE \; \bigcap_{j\in J} A_j.\vspace{-.5mm}
\end{equation}
\end{la}
\begin{proof}
As $A_{j_0}$ is compact and $x_n\to_\cE A_{j_0}$,
the sequence $(x_n)_{n\in\N_0}$ is bounded, by Lemma~\ref{specialkaes}\,(a).
If $c\in X$ is an accumulation point of $(x_n)_{n\in\N_0}$,
then $c\in A_j$ for each $j\in J$, by Lemma~\ref{specialkaes}\,(b),
and thus $c\in \bigcap_{j\in J}A_j=:A$.
Hence $x_n\to_\cE A$, by Lemma~\ref{specialkaes}\,(c).
\end{proof}
\begin{la}\label{intersect-famly}
Let $(X,\cE)$ be a uniform space,
$\alpha\colon X\to X$ be a continuous map 
and $(A_j)_{j\in J}$ be a family of closed subsets of~$X$, with $J\not=\emptyset$.
If $A_{j_0}$ is compact for some $j_0\in J$, then
\begin{equation}\label{in-family-con}
\bigcap_{j\in J}\con(\alpha, A_j)=\con\left(\alpha,\bigcap_{j\in J}A_j\right).
\end{equation}
\end{la}
\begin{proof}
If $x\in\bigcap_{j\in J}\con(\alpha,A_j)$,
then $\alpha^n(x)\to_\cE A_j$ for each $j\in J$,
whence $\alpha^n(x)\to_\cE \bigcap_{j\in J}A_j=:A$ (by Lemma~\ref{intersect-fam-seq})
and thus $x\in \con(\alpha,A)$. Thus $\bigcap_{j\in J}\con(\alpha,A_j)\sub\con(\alpha,A)$.
The converse inclusion trivially holds.
\end{proof}
\begin{la}\label{stable-enough}
Let $(X,\cE)$ be a uniform space, $\alpha\colon X\to X$
be  a continuous map and
$K\sub X$ be a compact subset.
Then
\[
\con(\alpha,K)=\con(\alpha, K_+\cap K_-)
\quad\mbox{and}\quad
\con^-(\alpha,K)=\con^-(\alpha, K_+\cap K_-).
\]
Moreover, every $\alpha$-regressive trajectory which converges to~$K$ in $(X,\cE)$
also converges to~$K_+\cap K_-$.
\end{la}
\begin{proof}
To prove the first assertion,
we show $\con(\alpha,K)\sub \con(\alpha, K_+\cap K_-)$
(as the other inclusion is trivial).
If $x\in \con(\alpha, K)$, then $\alpha^n(x)\to e$ modulo~$K$,
whence $R:=\{\alpha^n(x) \colon n\in\N_0\}$ is relatively compact in~$X$ (see Lemma~\ref{specialkaes}\,(a)).
Thus, by Lemma~\ref{specialkaes}\,(c), we shall have $x\in\con(\alpha, K_+\cap K_-)$
if we can show that each accumulation point~$c$ of $(\alpha^n(x))_{n\in\N_0}$
is an element of $K_+\cap K_-$.
By Lemma~\ref{specialkaes}\,(b), we know that each accumulation point of $(\alpha^n(x))_{n\in\N_0}$
is an element of~$K$.
Let $(\alpha^{n_a}(x))_{a\in A}$ be a subnet which converges to~$c$.
Given $m\in \N_0$, the subnet
\[
(\alpha^{n_a+m}(x))_{a\in A}=(\alpha^m(\alpha^{n_a}(x)))_{a\in A}
\]
of $(\alpha^n(x))_{n\in\N_0}$
converges to $\alpha^m(c)$, by continuity of $\alpha^m$.
Thus $\alpha^m(c)$ is an accumulation point of $(\alpha^n(x))_{n\in\N_0}$ and thus $\alpha^m(c)\in K$.
Hence $c\in \bigcap_{m\in\N_0}\alpha^{-m}(K)=K_-$.
To see that also $c\in K_+$ (which shows our first claims),
let us show that $c\in K_{m,\alpha}$ for each $m\in\N_0$.
There is $a_m\in A$ such that $n_a\geq m$ for all $a\geq a_m$.
Then $(\alpha^{n_a-m}(x))_{a\geq a_m}$ is a subnet of $(\alpha^n(x))_{n\in\N_0}$
and hence has a subnet converging to some $b\in K$.
Applying $\alpha^j$ to the entries of the latter subnet (for given $j\in\{0,\ldots,m\}$), we obtain a convergent net with limit $\alpha^j(b)=:x_{m-j}$,
which has to be an element of~$K$ as the convergent net is a subnet of $(\alpha^n(x))_{n\in\N_0}$.
For $j=m$, the convergent net just constructed is a subnet of $(\alpha^{n_a}(x))_{a\geq a_m}$;
its limit $x_0=\alpha^m(b)$ therefore has to coincide with~$c$.
Thus $x_0,x_1,\ldots, x_m\in K$, $x_0=c$ and
\[
\alpha(x_j)=\alpha(\alpha^{m-j}(b))=\alpha^{m-(j-1)}(b)=x_{j-1}
\]
for all $j\in\{1,\ldots, m\}$, whence $c\in K_{m,\alpha}$ (by Lemma~\ref{good-prepar}\,(b)).\\[2.3mm]
To complete the proof, it suffices to prove the final statement as it entails that $\con^-(\alpha,K)\sub \con^-(\alpha, K_+\cap K_-)$ (and the other inclusion trivially holds). Thus, let $(x_n)_{n\in\N_0}$ be an $\alpha$-regressive
trajectory in~$G$ such that $x_n\to e$ modulo~$K$.
By Lemma~\ref{specialkaes}\,(a),
the closure $\smash{L:=\wb{\{x_n\colon n\in\N_0\}}\sub X}$
is compact. Hence, by Lemma~\ref{specialkaes}\,(c),
$(x_n)_{n\in \N_0}$
will converge to~$e$ modulo $K_+\cap K_-$ if we can show that each accumulation point $c$
of $(x_n)_{n\in\N_0}$ is an element of $K_+\cap K_-$.
By Lemma~\ref{specialkaes}\,(b), we have $c\in K$.
Now $\alpha(c)$ is an accumulation
point of $(\alpha(x_n))_{n\in\N_0}$, by continuity of~$\alpha$,
and hence also an accumulation point of $(\alpha(x_{n+1}))_{n\in\N_0}=(x_n)_{n\in\N_0}$.
Thus, the set~$C$ of accumulation points of $(x_n)_{n\in\N_0}$ is $\alpha$-invariant, i.e., $\alpha(C)\sub C$.
Hence $C\sub K_-$.\\[2.3mm]
For $c$ as before and $m\in\N$, let $(x_{n_\beta})_{\beta\in B}$
be a subnet of $(x_n)_{n\in\N_0}$ which converges to~$c$.
Then $t_\beta:=(x_{n_\beta},x_{n_\beta+1},\ldots, x_{n_\beta+m})_{\beta\in B}$ is a net
in the compact topological space $L^{m+1}$ and hence has a convergent subnet,
with limit $(c_0,c_1,\ldots, c_m)\in L^{m+1}$, say.
Then $c_0=c$ and $\alpha(c_j)=c_{j-1}$ for all $j\in\{1,\ldots, m\}$
(as the entries of each $t_\beta$ have this property and~$\alpha$
is continuous). Thus $c\in K_{m,\alpha}$ for each $m\in \N_0$ and thus
\[
c\in \bigcap_{m\in\N_0}K_{m,\alpha}=K_+.
\]
Hence $c\in K_+\cap K_-$, as required.
\end{proof}
We close the section with some straightforward observations,
for later use.
\begin{la}\label{imageconv}
Let $f\colon X\to Y$
be a uniformly continuous map between uniform spaces $(X,\cE)$
and $(Y,\cF)$
Let $S\sub X$ be a subset.
\begin{itemize}
\item[\rm(a)]
If $(x_n)_{n\in\N_0}$
is a sequence in~$X$ such that $x_n\to_\cE S$,
then $f(x_n)\to_\cF f(S)$.
\item[\rm(b)]
If $\alpha\colon X\to X$ and $\beta\colon Y\to Y$ are continuous
maps such that
$\beta\circ f=f\circ\alpha$, then $f(\con(\alpha, S))\sub\con(\beta,f(S))$.
\end{itemize}
\end{la}
\begin{proof}
(a) If $U\in\cF$, then $V:=(f\times f)^{-1}(U)\in\cE$, whence there exists $N\in\N_0$
such that $x_n\in V[S]$ for all $n\geq N$.
Thus, for such~$n$, there is $s_n\in S$ such that $x_n\in V[s_n]$
and thus $(s_n,x_n)\in V$. Hence $(f(s_n),f(x_n))\in U$ and thus $f(x_n)\in U[f(s_n)]\sub U[f(S)]$.
So $f(x_n)\to_\cF f(S)$.\\[2.3mm]
(b) For $x\in \con(\alpha, S)$, we have $\beta^n(f(x))=f(\alpha^n(x))\to_\cF f(S)$, by~(a).
\end{proof}
The next lemma is obvious.
\begin{la}\label{subuni}
Let $(X,\cE)$ be a uniform space, $S\sub Y\sub X$ be subsets
and $(x_n)_{n\in\N_0}$ be a sequence in~$Y$ such that $x_n\to_\cE S$.
Let $\cF$ be the uniform structure induced by~$\cE$ on~$Y$.
Then also $x_n\to_\cF S$.
\end{la}
\begin{la}\label{sammelsub}
Let $G$ be a topological group, $S\sub G$ be a subgroup and $K\sub G$.
\begin{itemize}
\item[\rm(a)]
If $K\sub S$ and $S$ is $\alpha$-invariant for an endomorphism~$\alpha$ of~$G$,
then $\con(\alpha,K)\cap S=\con(\alpha|_S,K)$.
\item[\rm(b)]
If $S$ is open and $(x_n)_{n\in\N_0}$ is a sequence in~$S$
such that $x_n\to_\cR K$ in~$G$, then $x_n\to_\cR K\cap S$ in~$G$ $($and in~$S)$.
\end{itemize}
\end{la}
\begin{proof}
(a) Using the inclusion map $f\colon S\to G$ and $\beta:=\alpha|_S$, Lemma~\ref{imageconv}
yields $\con(\alpha|_S,K)\sub\con(\alpha,K)\cap S$. If, conversely, $x\in\con(\alpha,K)\cap S$,
then
$\alpha^n(x)\to_\cR K$ in~$G$ and hence in~$S$ (by Lemma~\ref{subuni}),
whence $x\in\con(\alpha|_S,K)$.\\[2.3mm]
(b) If $V\sub G$ is an identity neighbourhood such that $V\sub S$,
then $x_n\in VK$ if and only if $x_n\in V(K\cap S)$.
Hence $x_n\to_\cR K$ in~$G$ if and only if $x_n\to_\cR (K\cap S)$ in~$G$.
By Lemma~\ref{subuni}, then also $x_n\to_\cR (K\cap S)$ in~$S$.
\end{proof}
\section{Proof of Theorem~A for metrizable {\boldmath$G$}}\label{sec1}
In this section we prove Theorem~A in the metrizable case.
We include two lemmas the first one of which concerns
how tidiness above passes to subgroups.
\begin{la}\label{simple-obs}
Let $\alpha$ be an endomorphism of a totally disconnected, locally compact group~$G$.
Let $H\sub G$ be an $\alpha$-invariant closed subgroup
and $W\sub G$ be a compact open subgroup.
Then there exists $\ell_0\in\N_0$ such that
\[
\left(\bigcap_{j=0}^\ell\alpha^{-j}(W)\right)\cap H\vspace{-1mm}
\]
is tidy above for $\alpha|_H$, for all $\ell\geq\ell_0$.
\end{la}
\begin{proof}
	Since $\alpha(H)\sub H$, we have
	\begin{eqnarray}
	H\cap \bigcap_{j=0}^\ell\alpha^{-j}(W) &=& \{w\in H\colon
	(\forall j=0,\ldots, \ell)\; \alpha^j(w)\in W\}\notag \\
	&=& \{w\in H\colon
	(\forall j=0,\ldots, \ell)\; \alpha^j(w)\in H\cap W\}\notag \\
	&=&\bigcap_{j=0}^\ell(\alpha|_H)^{-j}(H\cap W).\label{thisgroup}
	\end{eqnarray}
The group in (\ref{thisgroup}) is tidy above for~$\alpha|_H$ by \cite[Proposition~3]{End} for large $\ell$.
\end{proof}
To prove Theorem~A,
we shall use the following lemma
(which parallels the case of automorphisms formulated in \cite[Lemma~3.10]{BaW}):
\begin{la}\label{BaWla}
Let $\alpha$ be an endomorphism of a totally disconnected, locally compact group~$G$ and $H\sub G$ be a closed subgroup.
Assume that
\begin{itemize}
\item[\rm(i)] $\alpha(H)=H$; or
\item[\rm(ii)] $H$ is compact.
\end{itemize}
Let $x\in\con(\alpha,H)$ and $O\sub G$ be a compact open subgroup.
Then there exists $h\in H$ and $N\in\N_0$ such that $\alpha^n(xh)\in O$ for all $n\geq N$ and $\alpha^n(xh)\to e$
modulo $(O\cap H)_+\cap (O\cap H)_-$ $($which is a compact subgroup of $O \cap H)$.
\end{la}
\begin{proof}
If $H$ is compact, then $\con(\alpha,H)=\con(\alpha,H_+\cap H_-)$,
where $H_+\cap H_-$ is compact and $\alpha(H_+\cap H_-)=H_+\cap H_-$
(see Lemma~\ref{good-prepar}\,(d) and Lemma~\ref{stable-enough}). After replacing $H$ with $H_+\cap H_-$, we may therefore
assume that $\alpha(H)=H$. Thus, it suffices to discuss the case~(i).\\[2.3mm]
Assume $\alpha(H)=H$ now.
After shrinking~$O$ if necessary, we may assume that
\[
O\cap H=(O\cap H)_- (O\cap H)_+,
\]
by Lemma~\ref{simple-obs}. Let $V\sub O$ be a compact open subgroup
such that $\alpha(V)\sub O$.
Since $\alpha^n(x)\to e$ modulo~$H$, there exists $N\in\N_0$ such that
$\alpha^n(x)\in VH$ for all $n\geq N$.
As $\alpha^N(H)=H$,
we find $h_0\in H$
with $\alpha^N(xh_0)\in V$.
We now complete $h_0$
to a sequence $(h_i)_{i\in\N_0}$
of elements in $h_0(O\cap H)$
such that
\[
\alpha^{N+j}(xh_i)\in V(O\cap H)\quad\mbox{for $0\leq j\leq i$.}
\]
Suppose that $h_0,\ldots, h_k$ have already been constructed to satisfy this condition.
Using that $\alpha(O\cap H)\sub (O\cap H)\alpha((O\cap H)_+)$, we obtain
\[
\alpha^{N+k+1}(xh_k)\in\alpha(V(O\cap H))\sub O\alpha(O\cap H)\sub O\alpha((O\cap H)_+).
\]
Choose $\ell_{k+1}\in \alpha((O\cap H)_+)$ such that $\alpha^{N+k+1}(xh_k)\ell_{k+1}\in O$.
There exists an $\alpha$-regressive
trajectory $(g_n)_{n\in\N_0}$ for $\ell_{k+1}$ such that $g_n\in O\cap H$ for all $n\geq 1$. We define
\[
h_{k+1}:=h_k g_{N+k+1}.
\]
Then $h_{k+1}\in h_0(O\cap H)$ and $\alpha^{N+k+1}(xh_{k+1})\in O$.
Moreover, $\alpha^{N+k+1}(xh_{k+1})=\alpha^{N+k+1}(x)\alpha^{N+k+1}(h_{k+1})\in VH$
and thus
\[
\alpha^{N+k+1}(xh_{k+1})\in O\cap VH=V(O\cap H).
\]
For all $j\in\{0,\ldots,k\}$, we have
\[
\alpha^{N+j}(xh_{k+1})=\alpha^{N+j}(xh_k)\alpha^{N+j}(g_{N+k+1})=\alpha^{N+j}(xh_k)g_{k-j+1}\in V(O\cap H)
\]
as well, whence $h_{k+1}$ is as desired. This completes the recursive construction.\\[2.3mm]
Let~$h$ be an accumulation point of the sequence $(h_i)_{i\in\N_0}$ in $h_0(O\cap H)\sub H$.
Then $\alpha^n(xh)$ is in the compact set $V(O\cap H)$ for all $n\geq N$, and so $\alpha^n(xh)\in O$.
Thus $xh\in \con(\alpha,O)$.
As also $xh\in \con(\alpha,H)$, Lemma~\ref{intersect-famly}
shows that $xh\in\con(\alpha,O \cap H)$.
Hence $\alpha^n(xh)\to e$ modulo $(O\cap H)_+\cap (O\cap H)_-$,
by Lemma~\ref{stable-enough}.
\end{proof}
\begin{numba}\label{birkhoff-kakutani}
For the proofs of Theorem~A and Theorem~B we recall that a topological group is metrizable if and only if it is Hausdorff and first-countable, see for instance \cite[(8.3)]{HaR}.
\end{numba}
\begin{proof}
(Theorem~A, for metrizable $G$). The inclusion $\con(\alpha)H\sub\con(\alpha, H)$ is trivial.
To establish the converse inclusion,
we adapt the proof of \cite[Theorem~3.8]{BaW}:
Since $G$ is assumed metrizable, using \ref{birkhoff-kakutani} we find a sequence
$O_1\supseteq O_2\supseteq\cdots$ of compact open subgroups
of~$G$ which form a basis of identity neighbourhoods.
Set $H_0:=H$ and
\[
H_j:=(O_j\cap H_{j-1})_+\cap (O_j\cap H_{j-1})_-\sub O_j \cap H_{j-1}
\;\;\mbox{for $j\in\N$.}
\]
Let $x_0\in\con(\alpha,H)$.
We claim that, for $j\in\N$, there exist elements $x_j \in x_{j-1}H_{j-1}$
and  integers $N_j\in\N_0$ such that
\[
\alpha^n(x_j)\to e\;\; \mbox{modulo $\,H_j$}
\]
and $\alpha^n(x_j)\in O_j$ for all $n\geq N_j$.\\[2.3mm]
Lemma~\ref{BaWla} (with $x_0$ in place of~$x$ and $O_1$ in place of~$O$)
provides $h_0\in H_0=H$ and $N_1\in\N_0$
such that $x_1:=x_0 h_0$ satisfies
$\alpha^n(x_1)\in O_1$ for all $n\geq N_1$ and $\alpha^n(x_1)\to e$ modulo $(O_1\cap H)_+\cap
(O_1\cap H)_- =H_1$.
This shows the case~$j=1$.\\[2.3mm]
If $x_1,\ldots, x_j$ have been constructed, we can apply Lemma~\ref{BaWla} with $x_j$ in place of~$x$,
$O_{j+1}$ in place of~$O$ and
$H_j$ in place of~$H$. We obtain an $h_j \in H_j$ and $N_{j+1}\in\N_0$
such that $x_{j+1}:=x_j h_j$ satisfies
$\alpha^n(x_{j+1})\in O_{j+1}$ for all $n\geq N_{j+1}$ and $\alpha^n(x_{j+1})\to e$ modulo $(O_{j+1}\cap H_j)_+\cap
(O_{j+1}\cap H_j)_- =H_{j+1}$. This completes the recursive construction.\\[2.3mm]
Note that
\[
x_i\in x_jH_j\cdots H_{i-1}\sub x_j O_j\cdots O_{i-1}=x_jO_j
\]
for all $i\geq j\geq 1$. Hence $(x_j)_{j\in\N_0}$ is a Cauchy sequence in~$G$
(for the left uniform structure) and hence convergent,
as~$G$ is locally compact and therefore complete by \cite[Paragraph 3.3, Corollary 1]{BouTop}. Let
\[
y:=\lim_{n\to\infty} x_n\quad\mbox{and}\quad h:=x^{-1}_0y.
\]
Since $x_j\in x_0H_0\cdots H_{j-1}\sub x_0H$ for all $j\in \N$, passing to the limit we deduce
that $y\in x_0H$ and thus $h\in H$.\\[2.3mm]
To see that $y\in\con(\alpha)$,
let $W\sub G$ be an identity neighbourhood. After shrinking~$W$,
we may assume that $W=O_j$ for some $j\in\N$. As
\[
x_i\in x_jH_j\cdots H_{i-1}\sub x_jH_j
\]
for all $i\geq j$, we deduce that $y\in x_jH_j$ and thus
\[
\alpha^n(y)\in \alpha^n(x_j)\alpha^n(H_j)=\alpha^n(x_j)H_j\sub O_j O_j=O_j=W
\]
for all $n\geq N_j$. Hence $y\in \con(\alpha)$, which completes the proof.
\end{proof}
\section{Proof of Theorem~B for metrizable {\boldmath$G$}}\label{sec4}
In this section we prove Theorem~B for metrizable $G$, using
the following analogue of Lemma~\ref{BaWla}.
\begin{la}\label{lemenable}
Let $G$ be a totally disconnected, locally compact group.
Let $\alpha\colon G\to G$ be an endomorphism
and $H\sub G$ be an $\alpha$-invariant closed subgroup.
Let $O\sub G$ be a compact open subgroup,
$x\in \con^-(\alpha,H)$
and $(x_n)_{n\in\N_0}$ be an $\alpha$-regressive trajectory for~$x$
such that $x_n\to e$ modulo~$H$.
Then there exist $N\in \N_0$ and an $\alpha$-regressive
trajectory $(y_n)_{n\in\N_0}$ such that
\[
y_n\in O \quad\mbox{for all $\,n\geq N$,}
\]
$y_n\in x_n H$ for all $n\in \N_0$,
and $y_n\to e$ modulo $(H\cap O)_+\cap (H\cap O)_-$.
\end{la}
\begin{proof}
By Lemma~\ref{simple-obs},
after replacing~$O$ with $\bigcap_{j=0}^\ell\alpha^{-j}(O)$ for some $\ell\in\N_0$,
we may assume that $O\cap H$ is tidy above for $\alpha|_H$, i.e.,
\begin{equation}\label{Wtidyabove}
H\cap O=(H\cap O)_+(H\cap O)_-.
\end{equation}
We choose a compact open subgroup $V\sub O$ such that $\alpha(V)\sub O$.
Since $x_n\to e$ modulo~$H$, there exists $N\in\N_0$ such that
\[
x_n\in VH \quad \mbox{for all $n\geq N$.}
\]
We show by induction that there exist elements $z_j\in x_{N+j}H$ for $j\in \N_0$
with
\[
\alpha^i(z_j)\in V(H\cap O)\quad\mbox{for all $\,i\in\{0, 1,\ldots, j\}$}.
\]
Since $x_N\in VH$, we have $x_N\in z_0 H$
for some $z_0\in V$. Then $z_0$ is as required.\\[2.3mm]
If $j\in\N_0$ and $z_j$ has already been found, then $z_j\in x_{N+j}H$
implies that
\[
z_j=x_{N+j}a
\]
for some $a\in H$.
Since $x_{N+j+1}\in VH$, we can write
\[
x_{N+j+1}= v b
\]
with $v\in V$ and $b\in H$. Thus
\[
z_j=x_{N+j}a=\alpha(x_{N+j+1})a=\alpha(v)\alpha(b)a,
\]
with $\alpha(b)\in\alpha(H)\sub H$.
Since $z_j\in V(H\cap O)\sub O$ and $\alpha(v)\in \alpha(V)\sub O$,
we conclude that $\alpha(b)a\in H\cap O$ and thus
\[
\alpha(b)a=h_+h_-
\]
for suitable $h_+\in (H\cap O)_+$ and $h_-\in (H\cap O)_-$.
Since $\alpha((H\cap O)_+)\supseteq (H\cap O)_+$,
we find $h\in (H\cap O)_+$ such that
$h_+=\alpha(h)$.
Thus
\[
z_{j+1}:=vh\in V(H\cap O)_+\sub V(H\cap O)
\]
and $\alpha(z_{j+1})=\alpha(v)h_+=z_j (h_-)^{-1}$, entailing that for $i\in \{1,\ldots, j+1\}$
\begin{eqnarray*}
\alpha^i(z_{j+1})&=& \alpha^{i-1}(\alpha(z_{j+1}))=\alpha^{i-1}(z_j)\alpha^{i-1}(h_-)^{-1}\\
& \in&  V(H\cap O)(H\cap O)_-=V(H \cap O).
\end{eqnarray*}
This completes the induction.\\[2.3mm]
For $j\in \N_0$, define the sequence $g^{(j)}=(g^{(j)}_i)_{i\in\N_0}$ in $V(H\cap O)$
via
\[
g^{(j)}_i:=\left\{
\begin{array}{rl}
\alpha^{j-i}(z_j) &\mbox{if $i\in\{0,\ldots,j\}$;}\\
e & \mbox{if $i>j$.}
\end{array}
\right.
\]
Then $(g^{(j)})_{j\in \N_0}$ is a sequence in the compact space $(V(H\cap O))^{\N_0}$
and hence has a convergent subnet $(g^{(j_a)})_{a\in A}$,
whose limit $f=(f_i)_{i\in\N_0}$ is an $\alpha$-regressive trajectory by Lemma~\ref{onekey}\,(a).
Then also
\[
y_n:=\left\{
\begin{array}{rl}
f_{n-N}&\mbox{if $n\geq N$}\\
\alpha^{N-n}(f_0)& \mbox{if $n\in\{0,1,\ldots, N\}$}
\end{array}
\right.
\]
is an $\alpha$-regressive trajectory.
Let $i\in\N_0$; since
\[
g^{(j)}_i=\alpha^{j-i}(z_j)\in x_{N+i} H
\]
for all $j\geq i$ and $x_{N+i}H$ is closed, we deduce that
\[
y_{N+i}=f_i\in  x_{N+i} H.
\]
In particular, $f_0=x_N h_0$ for some $h_0\in H$ and thus also
\[
y_n=\alpha^{N-n}(f_N)=\alpha^{N-n}(x_N)\alpha^{N-n}(h_0)=x_n\alpha^{N-n}(h_0)\in x_n H
\]
for $n\in\{0,1,\ldots, N\}$.\\[2.3mm]
Since $y_n\in V(H\cap O)\sub O$ for $n\geq N$, we have $y_n\to e$ modulo $O$.
As $y_n\in x_n H$ and $x_n\to e$ modulo~$H$,
also $y_n\to e$ modulo~$H$.
Hence $y_n\to e$ modulo $H\cap O$, by Lemma~\ref{intersect-fam-seq},
and thus $y_n\to e$ modulo $(H\cap O)_+\cap (H\cap O)_-$, by
Lemma~\ref{stable-enough}.
\end{proof}
\begin{proof}
(Theorem~B). The inclusion $\con^-(\alpha)H\sub \con^-(\alpha,H)$ is trivial.
Let $y_0\in \con^-(\alpha,H)$.
Thus $y_0$ has an $\alpha$-regressive trajectory $(y^{(0)}_n)_{n\in\N_0}$ with
\[
y^{(0)}_n\to e\;\;  \mbox{modulo~$\,H$.}
\]
Let $(O^{(j)})_{j\in \N}$
be a sequence of compact open subgroups $O^{(1)}\supseteq O^{(2)}\supseteq\cdots $ of~$G$ which forms
a basis of identity neighbourhoods. Set $H_0:=H$ and
\[
H_j:=(O^{(j)}\cap H_{j-1})_+\cap (O^{(j)}\cap H_{j-1})_-\sub O^{(j)} \cap H_{j-1}
\;\;\mbox{for $j\in\N$.}
\]
We claim that there exist integers $N_j\in\N_0$ for $j\in\N$ and $\alpha$-regressive trajectories
$y^{(j)}=(y^{(j)}_n)_{n\in\N_0}$
such that
\[
y^{(j)}_n\to e\;\; \mbox{modulo $\,H_j$,}
\]
$y^{(j)}_n\in W^{(j)}$ for all $n\geq N_j$, and $y^{(j)}_n\in y^{(j-1)}_n H_{j-1}$ for all $n\in\N_0$.\\[2.3mm]
Lemma~\ref{lemenable} (with $y_0$ in place of~$x$ and $O^{(1)}$ in place of~$O$)
provides $N_1\in\N_0$
and an $\alpha$-regressive trajectory $(y^{(1)}_n)_{n\in\N_0}$
with $y^{(1)}_n\in y^{(0)}_n H$ for all $n\in\N_0$,
\[
y^{(1)}_n\in O^{(1)}\quad\mbox{for all $\,n\geq N_1$,}
\]
and $y^{(1)}\to e$ modulo $(O^{(1)}\cap H)_+\cap
(O^{(1)}\cap H)_- =H_1$.
This shows the case~$j=1$.\\[2.3mm]
If $N_1,\ldots, N_j$ and $\alpha$-regressive trajectories $(y^{(1)}_n)_{n\in\N_0},\ldots
(y^{(j)}_n)_{n\in\N_0}$
have been constructed, we can apply Lemma~\ref{lemenable} with $y^{(j)}_0$ in place of~$x$,
$(y^{(j)}_n)_{n\in\N_0}$ in place of $(x_n)_{n\in\N_0}$,
$O^{(j+1)}$ in place of~$O$ and
$H_j$
in place of~$H$.
We obtain an integer $N_{j+1}\in\N_0$
and an $\alpha$-regressive trajectory $y^{(j+1)}=(y^{(j+1)}_n)_{n\in\N_0}$
such that $y^{(j+1)}_n\in O^{(j+1)}$ for all $n\geq N_{j+1}$, and moreover
$y^{(j+1)}_n\in y^{(j)}_n H_j$ for all $n\in \N_0$ and $y^{(j+1)}_n\to e$ modulo $(O^{(j+1)}\cap H_j)_+\cap
(O^{(j+1)}\cap H_j)_-=H_{j+1}$. This completes the recursive construction. \\[2.3mm]
We have $y^{(j+1)}_n\in y^{(j)}_n H_j\sub y^{(j)}_n O^{(j)}\sub y^{(j)}_n O^{(1)}$ for all $n\in\N_0$ and $j\in\N$.
Hence $y^{(j)}_n\in y^{(1)}_n O^{(1)}$ for all $j\in\N$ and $n\in\N_0$, showing that $y^{(j)}\in \prod_{n\in\N_0} y^{(1)}_n O^{(1)}$
for all $j\in\N$. The product being compact and metrizable, we find a convergent
subsequence of $(y^{(j)})_{j\in\N}$, with limit $z=(z_n)_{n\in\N_0}\in \prod_{n\in\N_0}y^{(1)}_n O^{(1)}$,
say. As each $y^{(j)}$ is an $\alpha$-regressive trajectory and $\alpha$ is continuous,
we see that also $(z_n)_{n\in\N_0}$ is an $\alpha$-regressive trajectory.
Since $y^{(j)}_0\in y^{(0)}_0H=y_0H$ for each $j\in \N$ and $y_0 H$ is closed,
we have $z_0\in y_0H$.\\[2.3mm]
If $j\in\N$ and $n\in\N_0$, then $y^{(i)}_n \in y^{(j)}_n O^{(j)}$ for all $i\geq j$, entailing that also
$z_n\in y^{(j)}_n O^{(j)}$. Hence $z_n\in O^{(j)}$ for all $n\geq N_j$, whence $z_n\to e$
and thus $z_0\in \con^-(\alpha)$.
\end{proof}
\section{Metrizable quotient groups}\label{sec2}
This section prepares the discussion of endomorphisms of non-metrizable groups.
As in the case of automorphisms treated in~\cite{Jaw},
we consider metric quotients of such as a tool.
\begin{la}\label{when-quot-metr}
Let $G$ be a locally compact group and $N$ a compact normal subgroup.
Then the following conditions are equivalent:
\begin{itemize}
\item[\rm(a)]
$G/N$ is metrizable.
\item[\rm(b)]
There exists a sequence $(V_n)_{n\in\N}$ of open identity neighbourhoods
$V_n\sub G$ such that $\bigcap_{n\in \N} V_n=N$.
\end{itemize}
\end{la}
\begin{proof}
Let $q\colon G\to G/N$, $x\mto xN$ be the canonical quotient map, which is continuous and open.
(a)$\impl$(b):
If $G/N$ is metrizable, then $G/N$ has a countable basis $\{W_n\colon n\in\N\}$
of open identity neighbourhoods, and the sets $V_n:=q^{-1}(W_n)$ ($n\in\mathbb{N}$) are as desired.

(b)$\impl$(a):
Let $(V_n)_{n\in\N}$ be a sequence of open identity neighbourhoods in~$G$
with intersection~$N$.
As~$N$ is compact, the open neighbourhood $V_n$ of~$N$
contains a uniform neighbourhood of the form $K_n N$,
for some identity neighbourhood $K_n\sub G$.
Since $G$ is locally compact, $K_n$ can be chosen compact.
We may also assume that $K_1\supseteq K_2\supseteq\cdots$.
Thus
\[
\cU:=\{q(K_n)\colon n\in \N\}
\]
is a filter basis of compact identity neighbourhoods in $G/N$.
Now
$\bigcap_{n\in\N} q(K_n)$ $=\!\{e\}$
(as the sets $q^{-1}(q(K_n))=K_nN$ have intersection~$N$),
entailing that $\cU$ converges to $\{e\}$.
Thus $\cU$ is a countable basis of identity neighbourhoods in~$G/N$
and thus $G/N$ is metrizable by \ref{birkhoff-kakutani}.
\end{proof}
\begin{la}\label{kernels-directed}
Let $G$ be a locally compact group and $\cN$ be a countable set of closed normal subgroups
$N\sub G$ such that $G/N$ is metrizable.
Set $M:=\bigcap_{N\in\cN}N$. Then also $G/M$ is metrizable.
\end{la}
\begin{proof}
For $N\in\cN$, let $q_N\colon G\to G/N$ be the canonical quotient map.
Also, let $q_M\colon G\to G/M$ be the canonical quotient map.
Then
\[
\phi :=(q_N)_{N\in\cN}\colon G\to\prod_{N\in\cN}G/N,\quad x\mto (xN)_{N\in\cN}
\]
is a continuous homomorphism to a metrizable topological group. As the kernel is $\ker(\phi)=M$,
the induced continuous homomorphism
\[
\wb{\phi}\colon G/M\to\prod_{N\in\cN}G/N
\]
(determined by $\wb{\phi}\circ q_M=\phi$) is injective.
If $K\sub G$ is a compact identity neighbourhood in~$G$, then $q_M(K)$
is a compact identity neighbourhood in $G/M$. Since $\wb{\phi}(q_M(K))$ is metrizable
and $\wb{\phi}$ restricts to a homeomorphism from the compact set $q_M(K)$ onto $\wb{\phi}(q_M(K))$,
we see that $q_M(K)$ is metrizable. Since $q_M(K)$ is an identity neighbourhood in $G/M$, this topological group
has a countable
basis of identity neighbourhoods and hence is metrizable.
\end{proof}
The following fact will be used in the proofs of the non-metrizable case of Theorem~A and Theorem~B.
\begin{rem}\label{opensigmacomp}
If $G$ is a locally compact group, $M\sub G$ a $\sigma$-compact subset and $\alpha\colon G\to G$
an endomorphism, then there exists an $\alpha$-invariant, $\sigma$-compact open subgroup
$S\sub G$ such that $M\sub S$.
To see this, let $V\sub G$ be a compact identity neighbourhood.
Then the subgroup~$S$ of $G$ generated by the $\sigma$-compact set
\[
\bigcup_{n\in\N_0}\alpha^n(M\cup V)
\]
is $\sigma$-compact and open. Moreover, $M\sub S$ and $\alpha(S)\sub S$.
\end{rem}
\begin{la}\label{proj-lim-adapted}
Let $G$ be a $\sigma$-compact, locally compact group,
and $\alpha\colon G\to G$ be an endomorphism. Let
$\cK$ be the set of all compact, normal subgroups
$K\sub G$ such that $\alpha(K)\sub K$ and $G/K$ is metrizable.
Then $\bigcap\cK=\{e\}$.
\end{la}
\begin{proof}
Let $U\sub G$ be a neighbourhood of~$e$.
By~\cite[(8.7)]{HaR}, $G$ has a compact, normal subgroup~$N$
such that $N\sub U$ and $G/N$ is metrizable. Hence
\[
N=\bigcap_{n\in \N} V_n\vspace{-1mm}
\]
for suitable open identity neighbourhoods~$V_n\sub G$,
by Lemma~\ref{when-quot-metr}. Then
\[
K:=\bigcap_{m\in\N_0}\alpha^{-m}(N)=\bigcap_{m\in\N_0}\bigcap_{n\in\N}\alpha^{-m}(V_n)\vspace{-1mm}
\]
is an $\alpha$-invariant, compact normal subgroup of $G$ contained in~$U$.
As $K$ is
an intersection of~a countable set of open identity neighbourhoods,
$G/K$ is metrizable
(by Lemma~\ref{when-quot-metr}).
\end{proof}
The following fact, whose proof is left to the reader, is useful for our purposes.
\begin{la}\label{cphencepl}
Let $(I,\leq)$ be a directed set and $\cS:=((G_j)_{j\in I},(q_{i,j})_{i\leq j})$
be a projective system of topological groups~$G_j$ and continuous homomorphisms
$q_{i,j}\colon G_j\to G_i$. Let $G$ be a topological group and
$(q_j)_{j\in I}$ be a family of continuous
homomorphisms $q_j\colon G\to G_j$ such that $q_{i,j}\circ q_j=q_i$ for $i\leq j$ in~$I$.
If $\bigcap_{j\in I}\ker(q_j)=\{e\}$
and $q_j$ is a quotient homomorphism with compact kernel
for each $j\in I$, then $(G,(q_j)_{j\in I})$ is the projective limit
of~$\cS$.
\end{la}
\begin{rem}\label{subgpPL}
(a)
In the situation
of Lemma~\ref{cphencepl},
let $(x_j)_{j\in I}$ be a family of elements $x_j\in G_j$
such that $x_i=q_{i,j}(x_j)$
if
$i\leq j$.
Then $(x_j)_{j\in I}\in P$, where $P$ is the standard realization of the projective limit. Thus there exists
a unique element $x\in G$ such that $q_j(x)=x_j$
for all $j\in I$.\\[2.3mm]
(b) In the
situation of Lemma~\ref{cphencepl}, assume that $H\sub G$ is a closed subgroup. Then $(H,(q_j|_H)_{j\in I})$
is the projective limit of $((q_j(H))_{j\in I},(q_{ij}|_{q_j(H)})_{i\leq j})$.\\[2.3mm]
In fact, since $q_j$ is a quotient homomorphism with compact kernel,
it is a closed mapping (see \cite[(5.18)]{HaR}),
whence $q_j|_H\colon H\to q_j(H)$
is a surjective, continuous, closed mapping and hence a quotient map.
We can therefore apply Lemma~\ref{cphencepl} with~$H$ in place of~$G$.
\end{rem}
\begin{defn}\label{thekappa}
If $\alpha$ is an endomorphism of a $\sigma$-compact, totally disconnected, locally compact group~$G$,
then there exists a
non-empty set~$\cK$ of
$\alpha$-invariant, compact normal subgroups $K\sub G$
such that $G/K$ is metrizable for each $K\in\cK$ and $\bigcap\cK=\{e\}$
(see Lemma~\ref{proj-lim-adapted}). Let $\kappa(G,\alpha):=|\cK|$ be the minimum cardinality of such~$\cK$.
\end{defn}
It is not difficult to see that $\kappa(G,\alpha)$ equals the character of $G$
in the situation of Definition~\ref{thekappa} unless $G$ is first-countable
in which case $\kappa(G,\alpha)=1$ as we may choose $K:=\{e\}$. Notably, $\kappa(G,\alpha)$ is independent of~$\alpha$.
We suppress the details as the additional information shall not be used.
\section{Proof of Theorem~A for non-metrizable {\boldmath$G$}}\label{sec3}
To pass from metrizable to general groups in Theorem~A,
we could adapt the arguments from~\cite{Jaw}.
However, as the treatment of Theorem~B requires a different
approach, we prefer to apply this new approach
also to Theorem~A.
\begin{defn}\label{defndeco}
We say that an endomorphism~$\alpha$ of a topological group~$G$
\emph{has the decomposition property}
if $\con(\alpha,K)=\con(\alpha)K$ for each $\alpha$-invariant compact subgroup $K\sub G$.
\end{defn}
\begin{la}\label{keytwoA}
Let $\alpha$ be an endomorphism of a topological group~$G$
and $x\in\con(\alpha,H)$
for an $\alpha$-invariant compact subgroup $H\sub G$.
Let
$q\colon G\to Q$ be a quotient homomorphism to a topological group~$Q$
with $\alpha$-invariant kernel~$N$.
Let $\wb{\alpha}\colon Q\to Q$
be the map determined by $\wb{\alpha}\circ q=q\circ\alpha$.
Let $\wb{y}\in\con(\wb{\alpha})\sub Q$
with $\wb{y}\in q(x) q(H)$.
Then there exists $y\in\con(\alpha,H\cap N)$ such that
\begin{equation}\label{onlypre}
q(y)=\wb{y} \quad\mbox{and}\quad y\in xH.
\end{equation}
If $\alpha$ has the decomposition property, then there exists $y\in\con(\alpha)$ with~{\rm(\ref{onlypre})}.
\end{la}
\begin{proof}
Let $h\in H$ such that $\wb{y}=q(x)q(h)$.
As $q(xh)=\wb{y}\in\con(\wb{\alpha})$,
we have $xh\in\con(\alpha,N)$.
Since $x\in\con(\alpha,H)=\con(\alpha,H)H$,
we deduce $xh\in\con(\alpha,H)$. Hence $xh\in\con(\alpha,H\cap N)$, by Lemma~\ref{intersect-fam-seq}
(and we can take $y:=xh$).
If $\alpha$ has the decomposition property,
we can write $xh=yn$ with $y\in \con(\alpha)$ and $n\in H\cap N$.
Then $y\in xH$ and $q(y)=q(xh)=\wb{y}$.
\end{proof}
\begin{numba}\label{settordi}
Let $I$ be an interval of ordinal numbers of the form $[0,k[$ or $[0,k]$.
Let $\alpha$ be an endomorphism of a topological group~$G$
and $(N_j)_{j\in I}$
be a family of $\alpha$-invariant, compact normal subgroups $N_j\sub G$
such that $N_j\sub N_i$ for $i\leq j$ in~$I$.
For $j\in I$, let $q_j\colon G\to G/N_j$ be the canonical quotient map
and $\alpha_j$ be the endomorphism of $G/N_j$ determined
by $\alpha_j\circ q_j=q_j\circ \alpha$.
For $i\leq j$, let
\[
q_{i,j}\colon G/N_j\to G/N_i
\]
be the map determined
by $q_{i,j}\circ q_j=q_i$.
Then $((G/N_j)_{j\in I},(q_{i,j})_{i\leq j})$ $($with $i,j\in I)$
is a projective system of topological groups.
\end{numba}
The next lemma is a variant of \cite[Lemma~7]{Jaw}.
\begin{la}\label{myway-1A}
Let $0\not=k$ be a limit ordinal, $\alpha$ an endomorphism of a topological group~$G$
and $(N_j)_{j\in [0,k[}$
be a family of $\alpha$-invariant, compact normal subgroups $N_j\sub G$
such that $N_j\sub N_i$ for $i\leq j$
in $[0,k[$. Let $N_k:=\bigcap_{j\in [0,k[}N_j$
and use notation as in {\rm\ref{settordi}}
with $I:=[0,k]$.
Let $H\sub G$ be an $\alpha$-invariant closed subgroup
and $x\in\con(\alpha,H)$.
Assume that, for each $j\in [0,k[$,
an element $y_j\in\con(\alpha_j)\sub G/N_j$
is given
such that
\begin{eqnarray}
q_{i,j}(y_j) = y_i & & \mbox{for all $\,i\leq j$ in $[0,k[$, and}\label{compble1A}\\
y_j \in  q_j(x)q_j(H) & & \mbox{for all $\,j\in[0,k[$.}\label{passendA}
\end{eqnarray}
Then there is a unique element $y_k \in G/N_k$
such that
\begin{equation}\label{sincePL1A}
q_{i,k}(y_k)= y_i \quad\mbox{for all $\,i\in[0,k[$.}
\end{equation}
Moreover, $y_k\in\con(\alpha_k)\sub G/N_k$ and
\begin{equation}\label{sincePL3A}
y_k \in q_k(x)q_k(H).
\end{equation}
\end{la}
\begin{proof}
By Lemma~\ref{cphencepl},
$G/N_k$ is the projective limit
of the projective system $((G/N_j)_{j\in[0,k[},(q_{i,j})_{i\leq j})$,
with the limit maps $q_{i,k}$.
Thus, the identities (\ref{compble1A})
determine a unique element
$y_k\in G/N_k$ such that~(\ref{sincePL1A})
holds.
By Remark~\ref{subgpPL}\,(b),
$q_k(H)$ is the projective limit of its quotients $q_i(H)$ for $i\in[0,k[$,
with the limit maps $q_{i,k}|_{q_k(H)}\colon q_k(H)\to q_i(H)$.
Since
\[
q_{i,k}(q_k(x)^{-1}y_k)=q_i(x)^{-1} y_i \in q_i(H)
\;\;\mbox{for all $i\in[0,k[$,}
\]
we deduce that $q_k(x)^{-1}y_k \in q_k(H)$.
To see that $y_k\in\con(\alpha_k)$,
let $U\sub G/N_k$ be an identity neighbourhood.
Then $q_{i,k}^{-1}(V)\sub U$ for some $i\in[0,k[$ and identity neighbourhood $V\sub G/N_i$.
Since $\lim_{n\to\infty}(\alpha_i)^n(q_i(y_i))=e$, there is $n_0\in\N$
such that
\[
q_{i,k}((\alpha_k)^n(y_k))=(\alpha_i)^n(y_i)\in V
\]
for all $n\geq n_0$ and thus $(\alpha_k)^n(y_k)\in q_{i,k}^{-1}(V)\sub U$.
\end{proof}
The following lemma is a variant of \cite[Lemma~8]{Jaw}.
\begin{la}\label{myway2A}
Let $k>0 $ be an ordinal number, $\alpha$ be an endomorphism of a topological group~$G$
and $(N_j)_{j\in [0,k[}$ be a family of $\alpha$-invariant, compact normal subgroups $N_j\sub G$
such that $N_j\sub N_i$ for all $i\leq j$ in $[0,k[$ and
\[
N_j=\bigcap_{i<j}N_i\vspace{-1mm}
\]
for each limit ordinal $0\not=j\in [0,k[$.
Let $H\sub G$ be an $\alpha$-invariant compact subgroup
and $x\in\con(\alpha,H)\sub G$. Use notation as in {\rm\ref{settordi}} with $I:=[0,k[$.
Assume that~$\alpha_j$ has the decomposition property
for each $j\in [0,k[$.
Then there exist elements $y_j\in\con(\alpha_j)\sub G/N_j$
for $j\in [0,k[$,
such that
conditions
{\rm(\ref{compble1A})}
and {\rm(\ref{passendA})} from Lemma~{\rm\ref{myway-1A}} are satisfied.
\end{la}
\begin{proof}
For all $j\in[0,k[$, we have $q_j(x)\in \con(\alpha_j,q_j(H))$, by Lemma~\ref{imageconv}\,(b).
Let $M$ be the set of all families $(y_j)_{j\in[0,\ell[}$ with $\ell\in\,]0,k]$
such that {\rm(\ref{compble1A})}
and {\rm(\ref{passendA})} are satisfied, with~$\ell$ in place of~$k$.
Since $\alpha_0$ has the decomposition
property and $q_0(x)\in\con(\alpha_0,q_0(H))$,
there exists $y_0\in \con(\alpha_0)$
such that
\[
y_0\in q_0(x)q_0(H).
\]
Then $(y_j)_{j\in\{0\}}\in M$ and thus $M\not=\emptyset$.
For $f=(y_j)_{j\in[0,\ell[}$ and $g=(y'_j)_{j\in[0,\ell'[}$ in~$M$,
write $f\leq g$ if $\ell\leq \ell'$ and $f=g|_{[0,\ell[}$.
To see that $(M,\leq)$ is inductively ordered, let~$\Gamma$
be a totally ordered subset of~$M$.
If $\Gamma=\emptyset$, then every element of~$M$ (which is non-empty)
is an upper bound for~$\Gamma$. If $\Gamma\not=\emptyset$,
write $[0,\ell_f[$ for the domain of $f\in\Gamma$.
Let $\ell=\sup\{\ell_f\colon f\in\Gamma\}\in\,]0,k]$.
Then the union of the graphs of the $f\in\Gamma$
is the graph of a function $g\in M$ which is an upper bound for~$\Gamma$.
Thus $(M,\leq )$ is inductively ordered and thus~$M$ has a maximal
element $f=(y_j)_{j\in[0,\ell[}$, by Zorn's Lemma.
If $\ell=k$, then the assertion of the lemma holds for~$\alpha$.
If $\ell<k$, we obtain a contradiction
because
$f$ could then be extended to an element of~$M$
defined on $[0,\ell]=[0,\ell+1[$, as we now verify.
It suffices to find an element
$y_\ell \in\con(\alpha_\ell)\sub G/N_\ell$ such that
(\ref{sincePL1A}) and (\ref{sincePL3A}) are satisfied, with $\ell$ in place of~$k$.\\[2.3mm]
If $\ell$ is a limit ordinal, then Lemma~\ref{myway-1A}
yields $y_\ell\in\con(\alpha_\ell)$ as desired.\\[2.3mm]
If $\ell$ has a precursor~$j$, we apply Lemma~\ref{keytwoA}
with $G/N_\ell$, the quotient map $q_{j,\ell}\colon G/N_\ell\to G/N_j$, $\alpha_\ell$, $q_\ell(N_j)$, $q_\ell(H)$, $q_\ell(x)$,
and $y_j$
in place of~$G$, $q$, $\alpha$, $N$, $H$, $x$, and $\wb{y}$,
respectively. Since $\alpha_\ell$ has the decomposition property,
we obtain $y_\ell\in\con(\alpha_\ell)\sub G/N_\ell$
such that $q_{j,\ell}(y_\ell)=y_j$ and $y_\ell\in q_\ell(x)q_\ell(H)$.
\end{proof}
\begin{proof}
(Theorem A, completed). We show by transfinite induction on cardinals~$k\geq 1$ that each endomorphism
$\alpha$ of a $\sigma$-compact, totally disconnected, locally compact
group~$G$ with $\kappa(G,\alpha)\leq k$ has the
decomposition property (where $\kappa(G,\alpha)$ is as in Definition~\ref{thekappa}).
If $k\leq\aleph_0$, then $\kappa(G,\alpha)\leq \aleph_0$ implies that~$G$ is metrizable (see Lemma~\ref{kernels-directed});
thus $\alpha$ has the decomposition property (as verified in Section~\ref{sec1}).\\[2.3mm]
Let~$k>\aleph_0$ now and assume that all endomorphisms~$\beta$ of
$\sigma$-compact, totally disconnected, locally compact groups~$H$
with $\kappa(H,\beta)<k$ have the decomposition property. The assertion will hold for~$k$ if we can show
that every endomorphism
$\alpha$ of a $\sigma$-compact,
totally disconnected, locally compact group~$G$ with $\kappa(G,\alpha)=k$
has the decomposition property.
In this situation, we choose a set~$\cK$ with $|\cK|=k$ as in Definition~\ref{thekappa}
and a bijection
\[
[0,k[\, \to \cK,\quad j\mto K_j.
\]
We define
\[
N_j:=\bigcap_{i<j} K_i\quad\mbox{for all $j\in\,]0,k[$ and $N_0:=N_1$.}\vspace{-1mm}
\]
This yields a decreasing family $(N_j)_{j\in[0,k[}$ of $\alpha$-invariant,
compact normal subgroups of~$G$ with $N_j=\bigcap_{i<j}N_i$
for each limit ordinal $0\not=j\in[0,k[$.
For $j\in[0,k[$, let $q_j\colon G\to G/N_j$ be the canonical quotient map
and $\alpha_j$ be the endomorphism of~$G/N_j$ induced by~$\alpha$.
Then $\cK_j:=\{q_j(K_i)\colon i<j\}$ is a set of cardinality at most~$j$
which has the properties of~$\cK$ required in Definition~\ref{thekappa},
with~$G/N_j$ in place of~$G$ and~$\alpha_j$ in place of~$\alpha$.
Hence
\[
\kappa(G/N_j,\alpha_j)\leq |\cK_j|\leq j<k,
\]
and thus $\alpha_j$ has the decomposition property.
Note that~$k$, being an infinite cardinal, is a limit ordinal.
Given $x\in\con(\alpha, H)\sub G$,
we apply Lemma~\ref{myway2A} to find
$y_j\in\con(\alpha_j)\sub G/N_j$ for $j\in[0,k[$ such that (\ref{compble1A}) and (\ref{passendA}) are satisfied.
Now Lemma~\ref{myway-1A}
provides $y\in G\cong G/\bigcap_{j<k}N_j$
such that $\alpha^n(y)\to e$ and $x^{-1}y\in H$.
Hence $\alpha$ has the decomposition property, as required.\\[2.3mm]
Now let $\alpha$ be an endomorphism of a totally disconnected, locally compact group~$G$
which need not be $\sigma$-compact.
Let $H\sub G$ be an $\alpha$-stable closed
subgroup or an $\alpha$-invariant compact subgroup.
Given $x\in\con(\alpha, H)$,
Lemma~\ref{BaWla} provides an $h\in H$
such that $xh\in \con(\alpha, K)$ with $K:=(O\cap H)_+\cap (O\cap H)_-$.
As observed in Remark~\ref{opensigmacomp},
$G$ has an $\alpha$-invariant $\sigma$-compact open subgroup
$S$ such that $K\cup\{xh\}\sub S$. Now, by Lemma~\ref{sammelsub}\,(a),
\[
xh\in\con(\alpha, K)\cap S=\con(\alpha|_S,K).
\]
Since $\alpha|_S$
has the decomposition property (as $S$ is $\sigma$-compact), we find $y\in \con(\alpha|_S)\sub\con(\alpha)$
and $k\in K$ such that $xh=yk$. Thus $x=ykh^{-1}\in \con(\alpha)KH=\con(\alpha)H$.
\end{proof}
\section{Theorem~B for non-metrizable {\boldmath$G$}}\label{proofB}
To deal with anti-contraction groups (rather than contraction groups),
we need to consider regressive trajectories rather than
individual group elements. This consideration leads to the following
analogue of Definition~\ref{defndeco}.
\begin{defn}
Let $\alpha$ be an endomorphism of a topological group~$G$.
We say that \emph{$\alpha$-regressive trajectories decompose} if,
for each $\alpha$-invariant closed subgroup $H\sub G$
and $\alpha$-regressive trajectory $(x_n)_{n\in\N_0}$
such that $x_n\to e$ modulo~$H$, there exists an $\alpha$-regressive trajectory
$(y_n)_{n\in\N_0}$ with $y_n\in x_nH$ for all $n\in\N_0$
and $\lim_{n\to\infty}y_n=e$. Then
\[
\con^-(\alpha,H)=\con^-(\alpha)H
\]
in particular
(since
$x_0=y_0 h$ for some $h\in H$, we have $\con^-(\alpha,H)\sub\con^-(\alpha)H$;
the converse inclusion is trivial).
\end{defn}
\begin{la}\label{keytwo}
Let $\alpha$ be an endomorphism of a topological group~$G$
and $(x_n)_{n\in\N_0}$ be an $\alpha$-regressive trajectory in~$G$
such that $x_n\to e$ modulo~$H$ for some $\alpha$-invariant closed subgroup $H\sub G$.
Let $N\sub G$ be an $\alpha$-invariant, compact normal subgroup and
$q\colon G\to G/N$ be the canonical quotient map.
Let $\wb{\alpha}\colon G/N\to G/N$
be the induced map determined by $\wb{\alpha}\circ q=q\circ\alpha$
and $(\wb{y}_n)_{n\in\N_0}$ be an $\wb{\alpha}$-regressive trajectory in $G/N$
such that $\wb{y}_n\in q(x_n) q(H)$ for each $n\in\N_0$.
Then there exists an $\alpha$-regressive trajectory $(y_n)_{n\in\N_0}$ in~$G$
such that
\[
q(y_n)=\wb{y}_n\quad\mbox{and}\quad y_n\in x_nH\quad\mbox{for all $n\in\N_0$.}
\]
If $\wb{y}_n\to e$ as $n\to\infty$ and $x_n\to e$ modulo~$H$,
then $y_n\to e$ modulo $H\cap N$.
\end{la}
\begin{proof}
For each $n\in\N_0$, there is $h_n\in H$ such that $\wb{y}_n=q(x_n)q(h_n)$.
We let $(z^{(m)})_{m\in\N_0}=((z^{(m)}_n)_{n\in\N_0})_{m\in\N_0}$ be the double sequence
associated with $(x_nh_n)_{n\in\N_0}$ (as in Definition~\ref{defdoubseq}).
Given $n\in\N_0$, we have
\[
q(z^{(m)}_n)=q(\alpha^{m-n}(x_mh_m))=\wb{\alpha}^{m-n}(q(x_mh_m))=\wb{\alpha}^{m-n}(\wb{y}_m)=
\wb{y}_n=q(x_nh_n)
\]
for all $m\geq n$. Moreover, $z^{(m)}_n=
\alpha^{m-n}(x_mh_m)=x_n\alpha^{m-n}(h_m)\in x_n H$.
Thus $z^{(m)}_n$ is contained in the compact set $x_n(H\cap h_n N)$
for all $m\geq n$ (and always in $x_n(H\cap h_nN) \cup\{e\}$),
whence $(z^{(m)})_{m\in\N_0}$ has an accumulation point $y=(y_n)_{n\in\N_0}$
in $\prod_{n\in\N_0}x_n(H\cap h_nN)$. Then $q(y_n)=q(x_nh_n)=\wb{y}_n$
and $y_n\in x_n H$. By Lemma~\ref{onekey}, $(y_n)_{n\in\N_0}$ is an $\alpha$-regressive
trajectory.\\[2.3mm]
If $q(y_n)=\wb{y}_n\to e$, then $y_n\to e$ modulo~$N$.
If, moreover, $x_n\to e$ modulo~$H$, then $y_n\in x_nH$ implies
that also $y_n\to e$ modulo~$H$.
Hence
$y_n\to e$ modulo $H\cap N$ (by Lemma~\ref{intersect-fam-seq}),
as~$N$ is compact.
\end{proof}
\begin{la}\label{myway-1}
Let $0\not=k$ be a limit ordinal, $\alpha$ be an endomorphism of a topological group~$G$
and $(N_j)_{j\in [0,k[}$
be a family of $\alpha$-invariant, compact normal subgroups $N_j\sub G$
such that $N_j\sub N_i$ for $i\leq j$
in $[0,k[$. Let $N_k:=\bigcap_{j\in [0,k[}N_j$
and use notation as in {\rm\ref{settordi}}
with $I:=[0,k]$.
Let $H\sub G$ be an $\alpha$-invariant, closed subgroup
and $(x_n)_{n\in\N_0}$ be
an $\alpha$-regressive trajectory in~$G$ such that $x_n\to e$ modulo~$H$.
Assume that, for each $j\in [0,k[$,
an $\alpha_j$-regressive trajectory $(y^{(j)}_n)_{n\in\N_0}$ in $G/N_j$
with ${\displaystyle\lim_{n\to\infty}y^{(j)}_n=e}$\vspace{-1mm} is given
such that, for each $n\in\N_0$,
\begin{eqnarray}
q_{i,j}(y^{(j)}_n) = y^{(i)}_n & & \mbox{for all $\,i\leq j$ in $[0,k[$,}\label{compble1}\\
y^{(j)}_n \in  q_j(x_n)q_j(H) & & \mbox{for all $\,j\in[0,k[$.}\label{passend}
\end{eqnarray}
Then there are unique elements $y^{(k)}_n\in G/N_k$
such that
\begin{equation}\label{sincePL1}
q_{i,k}(y^{(k)}_n)= y^{(i)}_n\quad\mbox{for all $\,i\in[0,k[$.}
\end{equation}
Moreover, $(y^{(k)}_n)_{n\in\N_0}$ is an $\alpha_k$-regressive
trajectory in $G/N_k$ which converges to~$e$ and
\begin{equation}\label{sincePL3}
y^{(k)}_n\in q_k(x_n)q_k(H) \quad
\mbox{for all $\, n\in\N_0$.}
\end{equation}
\end{la}
\begin{proof}
By Lemma~\ref{cphencepl},
$G/N_k$ is the projective limit
of the projective system $((G/N_j)_{j\in[0,k[},(q_{i,j})_{i\leq j})$,
with the limit maps $q_{i,k}$.
Thus, for $n\in\N_0$,\linebreak
(\ref{compble1})
determines a unique
$y^{(k)}_n\in G/N_k$ with~(\ref{sincePL1})
(see Remark~\ref{subgpPL}\,(a)).~As
\[
q_{i,k}(\alpha_k (y^{(k)}_{n+1}))=\alpha_i(q_{i,k}(y^{(k)}_{n+1}))=\alpha_i(y^{(i)}_{n+1})=
y^{(i)}_n=q_{i,k}(y^{(k)}_n)\;\mbox{for all $i\in[0,k[$,}
\]
we have $\alpha_k(y^{(k)}_{n+1})=y^{(k)}_n$. Thus $(y^{(k)}_n)_{n\in\N_0}$
is an $\alpha_k$-regressive trajectory.
Now
$q_k(H)$ is the projective limit of its quotients $q_i(H)$ for $i\in[0,k[$,
with the limit maps $q_{i,k}|_H\colon H\to q_{i,k}(H)$ (see Remark~\ref{subgpPL}\,(b)).
Since
\[
q_{i,k}(q_k(x_n)^{-1}y^{(k)}_n)=q_i(x_n)^{-1} y_n^{(i)} \in q_i(H)
\;\;\mbox{for all $i\in[0,k[$,}
\]
we deduce that $q_k(x)^{-1}y^{(k)}_n \in q_k(H)$.
To see that $y^{(k)}_n\to e$,
let $U\sub G/N_k$ be an identity neighbourhood.
Then $q_{i,k}^{-1}(V)\sub U$ for some $i\in[0,k[$ and identity neighbourhood $V\sub G/N_i$.
Since $q_i(y^{(i)}_n)\to e$, there is $n_0\in\N$
such that
\[
q_{i,k}(y_n^{(k)})=y^{(i)}_n\in V
\]
for all $n\geq n_0$ and thus $y_n^{(k)}\in q_{i,k}^{-1}(V)\sub U$.
\end{proof}
\begin{la}\label{myway2}
Let $k$ be an ordinal number, $\alpha$ be an endomorphism of a topological group~$G$
and $(N_j)_{j\in [0,k[}$ be a family of $\alpha$-invariant, compact normal subgroups $N_j\sub G$
such that $N_j\sub N_i$ for all $i\leq j$ in $[0,k[$ and
\[
N_j=\bigcap_{i<j}N_i\vspace{-1mm}
\]
for each limit ordinal $0\not=j\in [0,k[$.
Let $H\sub G$ be an $\alpha$-invariant closed subgroup.
Let $(x_n)_{n\in\N_0}$ be an $\alpha$-regressive
trajectory in~$G$ such that $x_n\to e$ modulo~$H$.
Using notation as in {\rm\ref{settordi}},
assume that $\alpha_j$-regressive trajectories in $G/N_j$ decompose
for each $j\in [0,k[$.
Then there exist
$\alpha_j$-regressive trajectories $(y^{(j)}_n)_{n\in\N_0}$ in $G/N_j$
which converge to~$e$,
for all $j\in [0,k[$,
such that conditions
{\rm(\ref{compble1})}
and {\rm(\ref{passend})} from Lemma~{\rm\ref{myway-1}} are satisfied.
\end{la}
\begin{proof}
For each $j\in[0,k[$, the sequence
$(q_j(x_n))_{n\in\N_0}$ is an $\alpha_j$-regressive trajectory in~$G/N_j$.
Moreover, $q_j(x_n)\to e$ modulo $q_j(H)$, by Lemma~\ref{imageconv}\,(a).
Let $M$ be the set of all families $f=(y^{(j)}_n)_{(j,n)\in[0,\ell[\,\times \N_0}$ with $\ell\in\,]0,k]$
such that {\rm(\ref{compble1})}
and {\rm(\ref{passend})} are satisfied, with~$\ell$ in place of~$k$.
Since $\alpha_0$-regressive trajectories in $G/N_0$
decompose and $q_0(x_n)\to e$ modulo~$q_0(H)$,
there exists an $\alpha_0$-regressive trajectory $(y^{(0)}_n)_{n\in\N_0}$
such that
\[
\mbox{$y^{(0)}_n\in q_0(x_n)q_0(H)$ for all $n\in\N_0$ and $y^{(0)}_n\to e$.}
\]
Then $(y^{(j)}_n)_{(j,n)\in \{0\}\times\N_0}\in M$
and thus $M\not=\emptyset$.
For $f=(y^{(j)}_n)_{(j,n)\in[0,\ell[\,\times\N_0}$ and $g=(z^{(j)}_n)_{(j,n)\in[0,\ell'[\,\times\N_0}$ in~$M$,
write $f\leq g$ if $\ell\leq \ell'$ and $f=g|_{[0,\ell[\,\times\N_0}$.
To see that $(M,\leq)$ is inductively ordered, let~$\Gamma$
be a totally ordered subset of~$M$.
If $\Gamma=\emptyset$, then every element of~$M$ (which is non-empty)
is an upper bound for~$\Gamma$. If $\Gamma\not=\emptyset$,
write $[0,\ell_f[\,\times\N_0$ for the domain of $f\in\Gamma$.
Let $\ell=\sup\{\ell_f\colon f\in\Gamma\}\in\,]0,k]$.
Then the union of the graphs of the $f\in\Gamma$
is the graph of a function $g\in M$ which is an upper bound for~$\Gamma$.
Hence $(M,\leq )$ is inductively ordered and thus~$M$ has a maximal
element $f=(y^{(j)}_n)_{(j,n)\in[0,\ell[\,\times \N_0}$, by Zorn's Lemma.
If $\ell=k$, then the assertion of the lemma holds for $\alpha$ and $(x_n)_{n\in\N_0}$.
If $\ell<k$, we reach a contradiction, as follows:
Since $[0,\ell+1[\,=[0,\ell]$, we can extend $(y^{(j)}_n)_{(j,n)\in[0,\ell[\,\times\N_0}$
to a properly larger element in $(M,\leq)$
if we can find an $\alpha_\ell$-regressive
trajectory $(y^{(\ell)}_n)_{n\in\N_0}$ in $G/N_\ell$ with $y^{(\ell)}_n\to e$
such that
(\ref{sincePL1}) and (\ref{sincePL3}) are satisfied, with $\ell$ in place of~$k$.\\[2.3mm]
If $\ell$ is a limit ordinal, then Lemma~\ref{myway-1}
yields $(y^{(\ell)}_n)_{n\in\N_0}$ in $G/N_\ell$ as desired (contradiction).\\[2.3mm]
If $\ell$ has a precursor~$j$, we apply Lemma~\ref{keytwo}
with $G/N_\ell$, the quotient map $q_{j,\ell}\colon G/N_\ell\to G/N_j$, $\alpha_\ell$, $q_\ell(N_j)$, $q_\ell(H)$, $(q_\ell(x_n))_{n\in\N_0}$,
and $(y^{(j)}_n)_{n\in\N_0}$
in place of~$G$, $q$, $\alpha$, $N$, $H$, $(x_n)_{n\in\N_0}$, and $(\wb{y}_n)_{n\in\N_0}$,
respectively. We obtain an $\alpha_\ell$-regressive trajectory $(y_n)_{n\in\N_0}$
in $G/N_\ell$ such that
\[
\mbox{$y_n\to e$ modulo $q_\ell(H)\cap q_\ell(N_j)$}
\]
and $q_{j,\ell}(y_n)=y^{(j)}_n$ as well as $y_n\in q_\ell(x_n)q_\ell(H)$ for all $n\in\N_0$.
Since $\alpha_\ell$-regressive trajectories decompose,
we find an $\alpha_\ell$-regressive trajectory $(y^{(\ell)}_n)_{n\in\N_0}$
in $G/N_\ell$ such that $y^{(\ell)}_n\to e$
and
\[
y^{(\ell)}_n\in y_n(q_\ell(H)\cap q_\ell(N_j))\sub q_\ell(x_n)q_\ell(H),
\]
whence $q_{j,\ell}(y^{(k)}_n)=q_{j,\ell}(y_n)=y^{(j)}_n$. Again, a contradiction is obtained.\vspace{1mm}
\end{proof}
\begin{proof}
(Theorem~B, completed). We show by transfinite induction on cardinals~$k\geq 1$ that $\alpha$-regressive trajectories decompose
for each endomorphism~$\alpha$ of a $\sigma$-compact, totally disconnected, locally compact
group~$G$ such that $\kappa(G,\alpha)\leq k$ (where $\kappa(G,\alpha)$ is as in Definition~\ref{thekappa}).
If $k\leq\aleph_0$, then $\kappa(G,\alpha)\leq \aleph_0$ implies that~$G$ is metrizable (see Lemma~\ref{kernels-directed});
thus $\alpha$-regressive trajectories decompose (as verified in Section~\ref{sec4}).\\[2.3mm]
Let~$k>\aleph_0$ now and assume that $\beta$-regressive trajectories decompose
for all endomorphisms~$\beta$ of $\sigma$-compact, totally disconnected, locally compact groups~$H$
with $\kappa(H,\beta)<k$. The assertion will
hold for~$k$ if we can show
that $\alpha$-regressive trajectories decompose for each endomorphism
$\alpha$ of a $\sigma$-compact,
totally disconnected, locally compact group~$G$ such that $\kappa(G,\alpha)=k$.
In this situation, we choose a set~$\cK$ with $|\cK|=k$ as in Definition~\ref{thekappa}
and a bijection
\[
[0,k[\, \to \cK,\quad j\mto K_j.
\]
We define
\[
N_j:=\bigcap_{i<j} K_i\quad\mbox{for all $j\in\, ]0,k[$ and $N_0:=N_1$.}\vspace{-1mm}
\]
This yields a decreasing family $(N_j)_{j\in[0,k[}$ of $\alpha$-invariant,
compact normal subgroups of~$G$ with $N_j=\bigcap_{i<j}N_i$
for each limit ordinal $0\not=j\in[0,k[$.
For $j\in[0,k[$, let $q_j\colon G\to G/N_j$ be the canonical quotient map
and $\alpha_j$ be the endomorphism of~$G/N_j$ induced by~$\alpha$.
Then $\cK_j:=\{q_j(K_i)\colon i<j\}$ is a set of cardinality~$\leq j$
which has the properties of~$\cK$ required in Definition~\ref{thekappa},
with~$G/N_j$ in place of~$G$ and~$\alpha_j$ in place of~$\alpha$.
Hence
\[
\kappa(G/N_j,\alpha_j)\leq |\cK_j|\leq j<k,
\]
and thus $\alpha_j$-regressive trajectories decompose.
Note that~$k$, being an infinite cardinal, is a limit ordinal.
Given an $\alpha$-regressive trajectory $(x_n)_{n\in\N_0}$ in~$G$,
we apply Lemma~\ref{myway2} to find $\alpha_j$-regressive trajectories
$(y^{(j)}_n)_{n\in\N_0}$ for $j\in[0,k[$ such that $\lim_{n\to\infty}y^{(j)}_n=e$
and both~(\ref{compble1}) and (\ref{passend}) are satisfied.
Now Lemma~\ref{myway-1}
provides an $\alpha$-regressive trajectory $(y_n)_{n\in\N_0}$ in $G\cong G/\bigcap_{j<k}N_j$
such that $y_n\to e$ and $x_n^{-1}y_n\in H$ for all $n\in\N_0$.
Hence $\alpha$-regressive trajectories decompose, as required.\\[2.3mm]
Now let $\alpha$ be an endomorphism of a totally disconnected, locally compact group~$G$
which need not be $\sigma$-compact, $H\sub G$ be an $\alpha$-invariant closed
subgroup and $(x_n)_{n\in\N_0}$ be an $\alpha$-regressive
trajectory such that $x_n\to e$ modulo~$H$.
There exists a $\sigma$-compact, $\alpha$-invariant open subgroup $S\sub G$
which contains the countable (hence $\sigma$-compact) set $\{x_n\colon n\in\N_0\}$
(see Remark~\ref{opensigmacomp}). Then $S\cap H$ is an $\alpha|_S$-invariant closed
subgroup of~$S$ and $(x_n)_{n\in\N_0}$ an $\alpha|_S$-regressive trajectory in~$S$
such that $x_n\to e$ modulo~$S\cap H$ (see Lemma~\ref{sammelsub}\,(b)).
Since $\alpha|_S$-regressive trajectories decompose (by $\sigma$-compactness of~$S$),
there exists an $\alpha|_S$-regressive trajectory $(y_n)_{n\in\N_0}$ in~$S$
such that $y_n\to e$ in~$S$ (hence in~$G$)
and $y_n\in x_n (S\cap H)\sub x_n H$ for all $n\in\N_0$.
Thus $\alpha$-regressive trajectories decompose.
Notably, $\con(\alpha,H)=\con(\alpha)H$.
\end{proof}
\section{Subgroups and quotients}\label{sec:scale_sub_quot}
In this section, we generalize several results of \cite{Fur} about how tidy subgroups and the scale behave with respect to taking subgroups and quotients from automorphisms to endomorphisms, paralleling the topological entropy study in~\cite{GBV}.
\begin{numba}\label{tidyingp}
For many results in this section we rely on the tidying procedure from \cite{End} which we shall outline for the benefit of the reader:
Given a totally disconnected locally compact group $G$, an endomorphism $\alpha$ and $U\in\COS(G)$, the following steps produce a tidy subgroup for $\alpha$.
\begin{enumerate}
	\item There exists $n\in\N$ such that $U_{-n}:=\bigcap_{k = 0}^{n}\alpha^{-k}(U)$ is tidy above for $\alpha$. By replacing $U$ with $U_{-n}$ we assume $U$ is tidy above.
	\item Define 
	\[\mathcal L_{U}: = \{x\in G:\exists y\in U_{+}\hbox{ }\exists m,n\in\N \hbox{ with }\alpha^{m}(y) = x \hbox{ and }\alpha^n(x)\in U_{-}\}\]
	and $L_{U}:=\overline{\mathcal L_{U}}$.
	\item Set $\widetilde{U} := \{x\in U: xL_U\sub L_{U}U\}$.
	\item Then $\widetilde{U}L_{U}$ is a compact open subgroup of $G$ which is tidy for $\alpha$.
\end{enumerate}
If the original subgroup $U$ was already tidy, then $\widetilde{U}L_{U} = U$.
\end{numba}
\subsection*{Subgroups}
We first explore the effect of taking subgroups on tidiness and the scale. The following lemma shows that tidy subgroups behave well when passing to subgroups. It is applied in Proposition \ref{prop:scale_subgroup} which deals with the scale.

\begin{la}\label{lem:tidy_subgroup}
Let $G$ be a totally disconnected, locally compact group, $\alpha$ an endomorphism of $G$ and $H\sub G$ a closed subgroup with $\alpha(H)\sub H$. Further, let $U\in\mathrm{COS}(G)$ be tidy for $\alpha$. Set $V:=U\cap H$. Then there is $N\in\mathbb{N}$ such that $V_{-N}$ is tidy for $\alpha|_{H}$.
\end{la}
\begin{proof}
Note that $V$ is a compact open subgroup of $H$. By \cite[Proposition~3]{End} there is $N\in\mathbb{N}$ such that $V_{-N}$ is tidy above for $\alpha$. Since $U$ is
minimizing, the same proposition implies that $U_{-N}$ is tidy for $\alpha$.
By Lemma~\ref{simple-obs},
replacing $U$ by $U_{-N}$, we may assume that $V$ is tidy above for $\alpha|_{H}$.
To see that this $V$ is tidy, we show that $\mathcal{L}_{V}\sub V$ where $\mathcal{L}_{V}$ is given in \ref{tidyingp}. Since $V\sub H$ is closed this implies that $L_{V}=\overline{\mathcal{L}}_{V}\sub V$ and hence $V$ is tidy below and therefore tidy for $\alpha|_{H}$ by \cite[Proposition~8]{End}. First, note that
\begin{displaymath}
 V_{-}=\bigcap_{n\ge 0}V_{-n}=U_{-}\cap H\vspace{-.8mm}
\end{displaymath}
(see (\ref{thisgroup})).
Also, since $V_{+}$ is the collection of all elements in $V$ that admit an $\alpha$-regressive trajectory in $V= U\cap H$, it follows that $V_{+}\sub U_{+}\cap H$. Now, suppose that $x\in\mathcal{L}_{V}$. Then $x\in H$ and there are $y\in V_{+}$ and $m,n\in\mathbb{N}$ such that $\alpha^{m}(y)=x$ and $\alpha^{n}(y)\in V_{-}$. By the above, $y\in U_{+}$ and $\alpha^{n}(y)\in U_{-}$. Therefore, $x\in\mathcal{L}_{U}\cap H$. Since $U$ is tidy for $\alpha$ we have $\mathcal{L}_{U}\sub U$ and thus conclude $x\in U\cap H=V$. This shows $\mathcal{L}_{V}\sub V$ as required.
\end{proof}
\begin{prop}\label{prop:scale_subgroup}
Let $G$ be a totally disconnected, locally compact group, $\alpha$ an endomorphism of $G$ and $H\sub G$ a closed subgroup with $\alpha(H)\sub H$. Then
\begin{displaymath}
 s_{H}(\alpha|_{H})\le s_{G}(\alpha).
\end{displaymath}
Furthermore, if $H$ is normal in $G$ and $U\in\mathrm{COS}(G)$ is tidy for $\alpha$ such that $U\cap H$ is tidy for $\alpha|_{H}$, then $\alpha((U\cap H)_{+})U_{+}$ is a
subgroup of $G$ and
\begin{displaymath}
 s_{H}(\alpha|_{H})=[\alpha((U\cap H)_{+})U_{+}:U_{+}].
\end{displaymath}
\end{prop}
\begin{proof}
By Lemma \ref{lem:tidy_subgroup} there is a tidy subgroup $U\in\mathrm{COS}(G)$ for $\alpha$ such that $V:=U\cap H$ is tidy for $\alpha|_{H}$. In particular, $s_{H}(\alpha|_{H})=[\alpha(V_{+}):V_{+}]$ and $s_{G}(\alpha)=[\alpha(U_{+}):U_{+}]$. Define $\varphi:\alpha(V_{+})/V_{+}\to\alpha(U_{+})/U_{+}$ by setting 
$\varphi(uV_{+}):=uU_{+}$ for all $uV_{+}\in\alpha(V_{+})/V_{+}$. Clearly $\varphi$ is well defined as 
$V_+\subseteq U_+$. For the first claim it suffices to show that $\varphi$ is injective. Indeed, assume that $\varphi(uV_{+})=\varphi(vV_{+})$ for some $uV_{+},vV_{+}\in\alpha(V_{+})/V_{+}$. Then it follows that $x:=v^{-1}u\in\alpha(V_{+})\cap U_{+}$ where $\alpha(V_{+})=\alpha((U\cap H)_{+})\sub H$. By \cite[Lemma~1]{End} we conclude $x\in U\cap H\cap\alpha(V_{+})=V\cap\alpha(V_{+})=V_{+}$.

For the second claim suppose that $H$ is normal in $G$. It suffices to show that $\alpha((U\cap H)_{+})U_{+}=U_{+}\alpha((U\cap H)_{+})$: Indeed, this implies that $\alpha((U\cap H)_{+})U_{+}$ is a group and the assertion then follows from the previous paragraph.
Now, $(U\cap H)_0:=U\cap H$ is a normal subgroup of $U_0:=U$
and $(U\cap H)_{n+1}:=\alpha((U\cap H)_n)\cap U\cap H$
is normal in $U_{n+1}:=\alpha(U_n)\cap U$ for each $n\in\N_0$ by the following inductive argument:
By the inductive hypothesis, $(U\cap H)_n$ is normal
in~$U_n$, whence $\alpha((U\cap H)_n)$ is normal in $\alpha(U_n)$.
Since $U\cap H$ is normal in $U$, it follows that $\alpha((U\cap H)_n)\cap U\cap H$
is normal in $\alpha(U_n)\cap U$ which completes the induction.
As a consequence,
\[
(U\cap H)_+:=\bigcap_{n\in\N_0} (U\cap H)_n
\quad\mbox{is normal in}\quad
U_+:=\bigcap_{n\in\N_0}U_n.\vspace{-.7mm}
\]
Let $u\in U_+$. Then $u=\alpha(w)$ for some $w\in U_+$.
Applying $\alpha$ to $(U\cap H)_+w=w(U\cap H)_+$,
we deduce that $\alpha((U\cap H)_+)u=u\alpha((U\cap H)_+)$.
\end{proof}
\subsection*{Quotients}
We now turn our attention to quotients. Again, we first consider tidy subgroups and then apply our findings to gain insight into the scale. Our first lemma is a tool that provides control over $\alpha$-regressive trajectories.
In the following, $L_{U}$ and $\widetilde{U}$ are
as in the tidying procedure, \ref{tidyingp}.
\begin{la}\label{lem:scale_quotients_1}
Let $G$ be a totally disconnected, locally compact group and $\alpha$ an endomorphism of $G$. Further, let $U\in\COS(G)$ be tidy above for $\alpha$. Then
\begin{displaymath}
 U\cap\widetilde{U}L_{U}=\widetilde{U}.
\end{displaymath}
\end{la}
\begin{proof}
The definitions imply $\widetilde{U}\sub U\cap\widetilde{U}L_{U}$ as $\widetilde{U}\sub U$ and $\widetilde{U}\sub\widetilde{U}L_{U}$. Now suppose $x\in U\cap\widetilde{U}L_{U}$. We need to show $x L_{U}\subseteq L_{U}U$. Indeed, we have $xL_{U}\subseteq\widetilde{U}L_{U}L_{U}=\widetilde{U}L_{U}\subseteq L_{U}U$.
\end{proof}
There are examples of automorphisms \cite{Fur} and associated tidy below subgroups which do not behave well when passing to quotients. Lemma~\ref{lem:scale_quotients_2} shows that although we can not expect a tidy below subgroup to be tidy below when passing to a quotient, the original subgroup can be chosen such that the quotient is as close as possible to being tidy below using the tidying procedure. The proof of Lemma~\ref{lem:scale_quotients_2} relies on the following result which is immediate from the proof of \cite[Lemma~16]{End}. We state it for convenience.
\begin{la}\label{lem:ObsFromEnd}
	Let $G$ be a totally disconnected locally compact group and $\alpha$ an endomorphism of $G$. Suppose $u\in \widetilde{U}$ where $U$ is tidy above for $\alpha$. Then for any pair $u_{\pm}\in U_{\pm}$ with $u = u_{+}u_{-}$ we have $u_{\pm}\in \widetilde{U}_{\pm}$.
\end{la}
\begin{la}\label{lem:scale_quotients_2}
Let $G$ be a totally disconnected, locally compact group, $\alpha$ an endomorphism of $G$ and $H\subseteq G$ a closed normal subgroup with $\alpha(H)\sub H$. Denote by $\overline{\alpha}$ the endomorphism induced by $\alpha$ on $G/H$ and by $q\colon G\to G/H$ the quotient map. Then there is a compact open subgroup $U$ of $G$ such that
\begin{itemize}
 \item[\rm(a)] $U$ tidy for $\alpha$,
 \item[\rm(b)] $U\cap H$ is tidy for $\alpha|_{H}$, and
 \item[\rm(c)] $q(U)$ is tidy above for $\overline{\alpha}$ and $L_{q(U)}q(U)=q(U)L_{q(U)}$.
\end{itemize}
\end{la}
\begin{proof}
Applying Lemma~\ref{lem:tidy_subgroup},
choose $V\in\mathrm{COS}(G)$ which is tidy for $\alpha$ and such that $V\cap H$ is tidy for $\alpha|_{H}$.
Then $q(V)$ is tidy above for $\overline{\alpha}$: On the one hand
\begin{align*}
 q(V_{-})=q\left(\bigcap_{n\ge 0}\alpha^{-n}(V)\right)\subseteq\bigcap_{n\ge 0}q(\alpha^{-n}(V))=\bigcap_{n\ge 0}\overline{\alpha}^{-n}(q(V))=q(V)_{-}.
\end{align*}
Also, $V_{+}=\{x\in V\mid \text{there is an $\alpha$-regressive trajectory for $x$ in $V$}\}$ and therefore $q(V_{+})\subseteq q(V)_{+}$ as $\alpha$-regressive trajectories descend to the quotient. Combining the above we conclude
\begin{displaymath}
 q(V)=q(V_{+}V_{-})=q(V_{+})q(V_{-})\subseteq q(V)_{+}q(V)_{-}.
\end{displaymath}
That is, $q(V)$ is tidy above for $\overline{\alpha}$. Now define $U:=V\cap q^{-1}(q(V)\wt{\;})$,
where $q(V)\wt{\;}$ is as in Step 3 of the tidying procedure, \ref{tidyingp}.
Then $q(U)=q(V)\wt{\;}$ and hence $q(U)$ is tidy above for $\overline{\alpha}$ by
\cite[Lemma~16]{End}. In addition, 
by applying \cite[Proposition~6\,(3)]{End} we see that $L_{q(U)} = L_{q(V)\wt{\;}} = L_{q(V)}$. It follows from \cite[Lemma~13]{End} and $q(U)=q(V)\wt{\;}$ that $q(U){L}_{q(U)}={L}_{q(U)}q(U)$.
Furthermore, $V\cap H\sub\ker q\sub  q^{-1}(q(V)\wt{\;})$.
Hence
\begin{displaymath}
 U\cap H=V\cap H\cap q^{-1}(q(V)\wt{\;})=V\cap H
\end{displaymath}
is tidy for $\alpha|_{H}$.

It remains to show that $U$ is tidy for $\alpha$. We begin by proving that $U$ is tidy above for $\alpha$. Let $u\in U$. Then since $V$ is tidy above, $u=v_{+}v_{-}$ for some $v_{\pm}\in V_{\pm}$ and we aim to show that $v_{\pm}\in U_{\pm}$. Note that $q(u)=q(v_{+})q(v_{-})$ with $q(v_{\pm})\in~q(V_{\pm})\sub~q(V)_{\pm}$.
Since $q(u)\in~q(V)\wt{\;}$, we deduce $q(v_{\pm})\in~(q(V)\wt{\;})_{\pm}$ by applying Lemma \ref{lem:ObsFromEnd}. Since $\alpha^{n}(v_{-})\in V_{-}$ and $\overline{\alpha}^{n}(q(v_{-}))\in (q(V)\wt{\;})_{-}$ for all $n\ge 0$ we have $q(\alpha^{n}(v_{-}))\in (q(V)\wt{\;})_{-}$. Therefore, the orbit of $v_{-}~\in~V\cap~q^{-1}(q(V)\wt{\;})=U$ stays in $U$ and we conclude $v_{-}\in U_{-}$.

As to $v_+$, choose an $\alpha$-regressive trajectory $(v_i)_{i\in\mathbb{N}_0}$ for $v_+$ contained in $V_+$. We will show this sequence is contained within $U$. It is clear that $v_0 = v_+\in U$. Suppose for the purpose of induction that $v_n\in U$. Applying \cite[Lemma~15]{End} we see that
$q(v_n)\in q(U)\cap q(V_+)\subseteq q(V)\wt{\;}\cap q(V)_{+} = (q(V)\wt{\;})_{+}$. There exists $w\in (q(V)\wt{\;})_{+}$ such that \[\overline{\alpha}(w) = q(v_n) = \overline{\alpha}(q(v_{n+1})).\]
Now $w$, $q(v_n)$ and $q(v_{n+1})$ are all elements of $\parb^{-}(\overline{\alpha})$. By \cite[Proposition~20]{End}, there exists $b\in\bik(\overline{\alpha})$ such that $q(v_{n+1}) = wb$. Since $q(V)\wt{\;}L_{q(V)}$ is tidy, $b\in q(V)\wt{\;}L_{q(V)}$. Hence $q(v_{n+1})\in q(V)\wt{\;}L_{q(V)}$. By Lemma \ref{lem:scale_quotients_1}, $q(v_{n+1})\in q(V)\wt{\;}$ which implies $v_{n+1}\in U$. Inductively $v_i\in U$ for all $i\in\mathbb{N}_0$ and so $v_+\in U_+$.

To see that $U$ is tidy below, note that $V$ is tidy below and $U\subseteq V$. Hence $L_{U}\subseteq V_{+}\cap V_{-}$.
Clearly, $q(V_{+}\cap V_{-})\subseteq L_{q(V)}$ and so $q(V_{+}\cap V_{-})\subseteq {q(V)\wt{\;}}$. Hence $V_{+}\cap V_{-}\subseteq U$. As a consequence, $L_{U}\subseteq U$ which implies that $U$ is tidy below, see \cite[Proposition~8]{End}.
\end{proof}
In the following lemma, we factor the subgroup used to calculate the scale. Later on, we turn this into a factorization of the scale itself.
\begin{la}\label{lem:scale_quotients_3}
Let $G$ be a totally disconnected, locally compact group, $\alpha$ an endomorphism of $G$ and $H\subseteq G$ a closed normal subgroup with $\alpha(H)\sub H$. Denote by $\overline{\alpha}$ the endomorphism induced by $\alpha$ on $G/H$. Then there is a closed subgroup $J$ of $G$ with $\alpha((H\cap U)_{+})U_{+}\sub J\sub\alpha(U_{+})$ and
\begin{displaymath}
  s_{G/H}(\overline{\alpha})=[\alpha(U_{+}):J].
\end{displaymath}

\end{la}
\begin{proof}
Let $U$ satisfy the conclusions of Lemma \ref{lem:scale_quotients_2}.
Let $q\colon G\to G/H$ denote the quotient map.
Then $q(U)L_{q(U)}$ is tidy for $\overline{\alpha}$ and
\begin{displaymath}
 s_{G/H}(\overline{\alpha})=[\overline{\alpha}(q(U)_{+})L_{q(U)}:q(U)_{+}L_{q(U)}]
\end{displaymath}
using \cite[Proposition~4 and Proposition~6\,(2)]{End}.
Now consider the map
\begin{displaymath}
 \overline{\alpha}(q(U)_{+})/(\overline{\alpha}(q(U)_{+})\cap q(U)_{+}L_{q(U)})\to\overline{\alpha}(q(U)_{+}L_{q(U)})/q(U)_{+}L_{q(U)}
\end{displaymath}
given by
\begin{displaymath}
 g(\overline{\alpha}(q(U)_{+})\cap q(U)_{+}L_{q(U)})\mto g(q(U)_{+}L_{q(U)}).
\end{displaymath}
This map is well-defined as $\overline{\alpha}(q(U)_{+})\cap q(U)_{+}L_{q(U)})\subseteq q(U)_{+}L_{q(U)}$. It is injective because any two elements in the domain which have the same image have coset representatives which differ by an element in $\overline{\alpha}(q(U)_{+})\cap q(U)_{+}L_{q(U)}$. To see surjectivity, simply note that
$\overline{\alpha}(L_{q(U)})\sub L_{q(U)}\sub q(U_+)L_{q(U)}$ by \cite[Lemma 6]{End}. This shows
\begin{eqnarray}
s_{G/H}(\overline{\alpha})&=& [\overline{\alpha}(q(U)_{+})L_{q(U)}:q(U)_{+}L_{q(U)}] \notag \\
& =& [\overline{\alpha}(q(U)_{+}):\overline{\alpha}(q(U)_{+})\cap q(U)_{+}L_{q(U)}].\label{newnum}
\end{eqnarray}
We know that $\overline{\alpha}(q(U)_{+})\cap q(U)_{+}L_{q(U)}$ is closed in $G/H$
because~$\overline{\alpha}$
and $q$ are continuous, $U$ is compact and $L_{q(U)}$ is closed. Set
\begin{displaymath}
 J:=q^{-1}\left(\overline{\alpha}(q(U)_{+})\cap q(U)_{+}L_{q(U)}\right)\cap\alpha(U_{+}).
\end{displaymath}
By the above, $J\sub\alpha(U_{+})$ is a closed subgroup. To see $\alpha((H\cap U)_{+})U_{+}\sub J$, note that
\begin{equation}\label{newnum2}
 q(\alpha((H\cap U)_{+})U_{+})=q(U_{+})\sub q(U)_{+}\sub\overline{\alpha}(q(U)_{+})\cap q(U)_{+}L_{q(U)}=:S
\end{equation}
because $\alpha((H\cap U)_{+})U_{+}=U_{+}\alpha((H\cap U)_{+})$ and $\alpha((H\cap U)_{+})$ is contained in~$H$.
The formula
\[
x.(yS):=q(x)yS\quad\mbox{for $x\in \alpha(U_+)$ and $y\in q(U_+)$}
\]
defines a left action of
$\alpha(U_+)$ on $X:=\overline{\alpha}(q(U_+))/S$
which is transitive as $q(\alpha(U_+))=\overline{\alpha}(q(U_+))$.
Since $S\in X$ has stabilizer $q^{-1}(S)\cap \alpha(U_+)=J$
under the action,
the Orbit Formula (as in \cite[1.6.1\,(i)]{Rob}) shows that
\[
[\alpha(U_+):J]=|X|=[\overline{\alpha}(q(U_+)):S].
\]
Combining this with (\ref{newnum2}) and (\ref{newnum}) we obtain $s_{G/H}(\overline{\alpha})=[\alpha(U_{+}):J]$.
\end{proof}
\begin{thm}\label{thm:ScaleQuotDiv}
Let $G$ be a totally disconnected, locally compact group, $\alpha$ an endomorphism of $G$ and $H\subseteq G$ a closed normal subgroup with $\alpha(H)\sub H$. Denote by $\overline{\alpha}$ the endomorphism induced by $\alpha$ on $G/H$. Then $s_{H}(\alpha|_{H})s_{G/H}(\overline{\alpha})$ divides $s_{G}(\alpha)$.
\end{thm}
\begin{proof}
Let $U$ satisfy the conclusions of Lemma \ref{lem:scale_quotients_2}. By Lemma \ref{lem:scale_quotients_3} there is a closed subgroup $J$ of $G$ such that
\begin{displaymath}
 U_{+}\sub\alpha((U\cap H)_{+})U_{+}\sub J\sub\alpha(U_{+}).
\end{displaymath}
Recall that by Proposition \ref{prop:scale_subgroup}, the set $\alpha((U\cap H)_{+})U_{+}$ is indeed a subgroup of $G$. Applying Lemma \ref{lem:scale_quotients_3} and Proposition \ref{prop:scale_subgroup} yields
\begin{align*}
 s_{G}(\alpha)&=[\alpha(U_{+}):U_{+}] \\
 &=[\alpha(U_{+}):J][J:\alpha((U\cap H)_{+})U_{+}][\alpha((U\cap H)_{+})U_{+}:U_{+}] \\
 &=s_{G/H}(\overline{\alpha})[J:\alpha((U\cap H)_{+})U_{+}]s_{H}(\alpha|_{H}).
\end{align*}
which completes the proof.
\end{proof}
We end this section by considering the special case of nested subgroups inside $\parb^{-}(\alpha)$ for which we achieve equality in Theorem \ref{thm:ScaleQuotDiv}.
\begin{la}
	\label{lem:parlemma}
	Let $G$ be a a totally disconnected locally compact group, $\alpha$ be an endomorphism of $G$ and $H\sub \parb^{-}(\alpha)$ a closed $\alpha$-stable subgroup. Then $\parb^{-}(\alpha|_{H}) = H$.
\end{la}
\begin{proof}
	Suppose $x\in H$.
We can find an $\alpha$-regressive trajectory $(x = x_0, x_1, \ldots)$ which is contained in some compact set $K$. Since $\alpha(H) = H$ we can choose another $\alpha$-regressive trajectory $(x = y_0, y_1, \ldots)$ such that $y_n\in H$ for all $n\in\N$.
Now $y_n,x_n\in \parb^{-}(\alpha)$ for all $n\in \N$ and hence $x_n^{-1}y_n\in \bik(\alpha)$. This shows $y_n\in x_n\bik(\alpha)\subseteq K\bik(\alpha)$. Since both $K$ and $\bik(\alpha)$ are compact, $K\bik(\alpha)$ is compact and hence $K\bik(\alpha)\cap H$ is a compact subset of $H$. This shows that $(y_0,y_1,\ldots)$ is bounded and hence $x\in \parb^{-}(\alpha|_{H})$.
\end{proof}
The following result is known for automorphisms, see \cite[Proposition 3.21~(2)]{CON}. In the case of automorphisms, the proof utilizes the modular function which is not defined for endomorphisms. Instead we consider the factoring of the scale given by Theorem \ref{thm:ScaleQuotDiv}. 
\begin{prop}\label{prop:Janalysis}
	Suppose $\alpha$ is an endomorphism of a totally disconnected locally compact group $G$ and $H\sub \parb^{-}(\alpha)$ is a closed $\alpha$-stable subgroup. Also suppose that $N\subseteq H$ is a closed normal $\alpha$-stable subgroup. Denote by $\overline{\alpha}$ the endomorphism induced by $\alpha|_{H}$ on $H/N$. Then 
	\[s_H(\alpha|_{H}) = s_{H/N}(\overline{\alpha})s_{N}(\alpha|_N).\]
\end{prop}
\begin{proof}
	For simplicity, we write $\alpha$ for $\alpha|_{H}$ as the enveloping group will play no further role. By Lemma \ref{lem:parlemma}, $\parb^{-}(\alpha) = H$ and so if $U\in \COS(H)$ is tidy for $\alpha$,
	then $U = U_+$ by \cite[Proposition 11]{End}.
By Lemma~\ref{lem:tidy_subgroup},
we may assume that $U\cap N$ is tidy for $\alpha|_N$.
Let $q:H\to H/N$ denote the quotient map. Choose $U\in\COS(H)$ satisfying conditions of Lemma \ref{lem:scale_quotients_2} with respect to the normal subgroup $N$. From the proof of Theorem \ref{thm:ScaleQuotDiv} we have
	\[s_H(\alpha) = s_{H/N}(\overline{\alpha})[J:\alpha((U\cap N)_+)U_+]s_{N}(\alpha|_{N}),\]
	where $J$ is given in the proof of Lemma \ref{lem:scale_quotients_2} by
	\[J = q^{-1}(\overline{\alpha}(q(U)_+)\cap q(U)_+L_{q(U)})\cap \alpha(U_+).\]
	It suffices to show $J \sub \alpha((U\cap N)_+)U_+$.	Since $q(U_+)\sub q(U)_+$, as seen in the proof of Lemma \ref{lem:scale_quotients_2}, and $U_+ = U$ we have $q(U_+)\sub q(U)_+ \sub q(U) = q(U_+),$
	which gives equality throughout. Thus $J = q^{-1}\left(\overline{\alpha}(q(U))\cap q(U)L_{q(U)}\right)\cap \alpha(U)$.
Since $q(U)$ is an open identity neighbourhood,
we obtain
\[
q(U)L_{q(U)}=q(U)\overline{{\mathcal{L}}_{q(U)}}=q(U){\mathcal{L}}_{q(U)}.
\]
Suppose that $x\in q^{-1}(q(U) L_{q(U)})$. We can write
$x = ul$ where $u\in U$ and $l\in q^{-1}(\mathcal L_{q(U)})$. Consider $q(l) = lN\in \mathcal L_{q(U)}$. There exists $n\in\N$ with \[\overline{\alpha}^n(lN) = \alpha^{n}(l)N \in q(U).\]
	This implies $\alpha^n(l)m\in U$ for some $m\in N$. Then $\alpha^n(l)m$ has an $\alpha$-regressive trajectory contained in $U = U_+$.
Using that fact that $N$ is assumed to be $\alpha$-stable, choose $m'\in N$ such that $\alpha^n(m')=m$.
Since \cite[Proposition 20]{End} implies that any two elements in the preimage of an element of $\parb^{-}(\alpha) = H$ are equal modulo $\bik(\alpha)$,
we have
$lm'\in U\bik(\alpha)$ by comparing $\alpha^n(lm')=\alpha^n(l)m$ with the $\alpha$-regressive trajectory for $\alpha^n(l)m$ contained in $U$. But $U$ is tidy and so $\bik(\alpha)\sub U$. Hence $l\in UN$
and thus $x\in UN$.
This shows that $J \subset UN\cap \alpha(U)$. Suppose now that $x\in UN\cap \alpha(U)$. Then we can write $x = un$ where $u\in U$ and $n\in N$. Choose $\alpha$-regressive trajectories 
	\[(u = u_0, u_1,\ldots)\hbox{, }(un = v_0,v_1,\ldots) \hbox{, and }(n = n_0, n_1, \ldots)\]
	such that $u_i,v_{i+1}\in U$ for all $i\ge 0$ and $n_i\in N$ for all $i\in\mathbb{N}$. Note that $(un=u_0n_0,u_1n_1,\ldots)$ is also an $\alpha$-regressive trajectory. For all $i\ge 1$ we have $u_in_i\in v_i\bik(\alpha)$. Noting that $\bik(\alpha)\sub U$,
we have $n_i\in U$ for all $i\ge 1$. Then $n_1\in (U\cap N)_+$ and so $n = n_0 = \alpha(n_1)\in \alpha((U\cap N)_+)$.
As $x=un$,
this shows $x\in U\alpha((U\cap N)_{+})=\alpha((U\cap N)_+)U$ (with equality by Proposition~\ref{prop:scale_subgroup}).
\end{proof}
\section{Small tidy subgroups and Theorem~D}\label{sec5}
We now turn to Theorem~D which characterizes admitting small tidy subgroups for an endomorphism in terms of its contraction group. The following link between (anti)-contraction groups and the subgroups of a compact open subgroup associated to tidiness is known for automorphisms (cf.\ \cite[Proposition~3.16]{BaW}).
\begin{la}\label{modVpVm}
Let $\alpha$ be an endomorphism of a totally disconnected, locally
compact group~$G$, and $V\sub G$ be a compact open subgroup.
Then
\[
V_{--}=\con(\alpha, V_+\cap V_-)\quad\mbox{and} \quad V_{++}=\con^-(\alpha,V_+\cap V_-).
\]
\end{la}
\begin{proof}
If $x\in \con(\alpha, V_+\cap V_-)$,
then there exists $n\in\N_0$ such that $\alpha^m(x)\in V(V_+\cap V_-)=V$
for all $m\geq n$. Thus $\alpha^k(\alpha^n(x))\in V$ for all $k\in\N_0$,
whence $\alpha^n(x)\in V_-$ and hence $x\in\alpha^{-n}(V_-)\sub V_{--}$.
Thus $\con(\alpha, V_+\cap V_-)\sub V_{--}$\\[2.3mm]
If $x\in V_{--}$, there exists $m\in\N_0$ such that
$\alpha^m(x)\in V_-$ and hence $\alpha^n(x)=\alpha^{n-m}(\alpha^m(x))\in V$
for all $n\geq m$. Thus $\alpha^n(x)\to e$ modulo~$V$.
By Lemma~\ref{stable-enough}, also $\alpha^n(x)\to e$ modulo $V_+\cap V_-$.\\[2.3mm]
If $x\in \con^-(\alpha, V_+\cap V_-)$,
let $(x_n)_{n\in\N_0}$ be an $\alpha$-regressive trajectory for~$x$ such that $x_n\to e$
modulo $V_+\cap V_-$.
There exists $n\in\N_0$ such that $x_m\in V(V_+\cap V_-)=V$
for all $m\geq n$. Then $(x_{n+m})_{m\in\N_0}$ is an $\alpha$-regressive trajectory for $x_n$ in~$V$,
whence $x_n\in V_+$ and $x_0=\alpha^n(x_n)\in\alpha^n(V_+)\sub V_{++}$.\\[2.3mm]
If $x\in V_{++}$, then $x=\alpha^m(y)$ for some $m\in\N_0$ and $y\in V_+$.
Let $(y_n)_{n\in\N_0}$ be an $\alpha$-regressive trajectory for~$y$
in~$V$. Then
\[
x=\alpha^m(y),\alpha^{m-1}(y),\ldots, \alpha(y),y=y_0,y_1,y_2,\ldots
\]
is an $\alpha$-regressive trajectory for~$x$ which eventually lies in~$V$ and hence converges to~$e$
modulo~$V$. As it also converges to~$e$ modulo $V_+\cap V_-$ (by Lemma~\ref{stable-enough}),
we see that $x\in \con^-(\alpha,V_+\cap V_-)$.
\end{proof}
\begin{la}\label{conclosure}
Let $\alpha$ be an endomorphism of a totally disconnected,
locally compact group~$G$ and $V\sub G$ by a tidy subgroup for~$\alpha$.
Then
\[
\wb{\con(\alpha)}\sub \con(\alpha,V_+\cap V_-)\quad\mbox{and}\quad
\wb{\con^-(\alpha)}\sub\con^-(\alpha,V_+\cap V_-).
\]
\end{la}
\begin{proof}
Since $V$ is tidy, both $V_{--}=\con(\alpha,V_+\cap V_-)$ and $V_{++}=\con^-(\alpha, V_+\cap V_-)$
are closed in~$G$ (with equalities by Lemma~\ref{modVpVm}).
So $\wb{\con(\alpha)}\sub \con(\alpha,V_+\cap V_-)$
and $\wb{\con^-(\alpha)}\sub\con^-(\alpha,V_+\cap V_-)$.
\end{proof}
If $\alpha$ is an endomorphism of a totally disconnected locally compact group,
we write $\TID(\alpha)$ for the set of all tidy subgroups of~$\alpha$.
The following auxiliary result will be improved in Proposition~\ref{keynub}\,(d).
\begin{la}\label{conclosure2}
Let $\alpha$ be an endomorphism of a totally disconnected,
locally compact group~$G$.
Then $\, \wb{\con(\alpha)}\sub \con(\alpha)\nub(\alpha)$.
\end{la}
\begin{proof}
For each $V\in\TID(\alpha)$, we have $\wb{\con(\alpha)}\sub\con(\alpha, V_+\cap V_-)\sub \con(\alpha, V)$.
and therefore
\begin{eqnarray*}
\wb{\con(\alpha)}&\sub& \bigcap_{V\in\TID(\alpha)}\con(\alpha, V)=\con\left(\alpha,\bigcap_{V\in\TID(\alpha)} V\right)\\
&=& \con(\alpha,\nub(\alpha))=\con(\alpha)\nub(\alpha),
\end{eqnarray*}
using Lemma~\ref{intersect-famly} for the first equality and Theorem~A for the last.
\end{proof}
The following result is known for automorphisms (cf.\ \cite[Lemma~3.31\,(3)]{BaW}).
\begin{prop}\label{when-TB}
Let $\alpha$ be an endomorphism of a totally disconnected, locally compact group~$G$
and $V\sub G$ be a compact open subgroup which is tidy above for~$\alpha$.
Then $V$ is tidy if and only if $V$ contains the nub of~$\alpha$.
\end{prop}
\begin{proof}
If $V$ is tidy, then $V\supseteq \nub(\alpha)$ by definition of the nub. For the converse,
assume that $V$ is tidy above
and $V\supseteq \nub(\alpha)$. Since $W_1\cap W_2$ is tidy for all tidy subgroups $W_1$
and $W_2$ (see \cite[Proposition~12]{End}), $\TID(\alpha)$ is a filter basis.
By Lemma \ref{la:filter_basis}, there exists $W\in\TID(\alpha)$ such that $W\sub V$.
Using both Lemma~\ref{modVpVm} and Theorem~A twice, we see that
\begin{eqnarray*}
V_{--}&=& \con(\alpha,V_+\cap V_-)=\con(\alpha)(V_+\cap V_-)=\con(\alpha)(W_+\cap W_-)(V_+\cap V_-)\\
&=& \con(\alpha, W_+\cap W_-)(V_+\cap V_-)=
W_{--}(V_+\cap V_-).
\end{eqnarray*}
As $W_{--}$ is closed and $V_+\cap V_-$ is compact, $V_{--}$ is closed.
Hence $V$ is tidy.
\end{proof}
The following proposition subsumes Theorem~D.
For the special case of automorphisms, see already
\cite[Theorem 3.32]{BaW} (for metrizable~$G$).
\begin{prop}\label{betterC}
Let $\alpha$ be an endomorphism of a totally disconnected, locally compact group~$G$.
Then the following conditions are equivalent:
\begin{itemize}
\item[\rm(a)]
$\alpha$ has small tidy subgroups.
\item[\rm(b)]
$\nub(\alpha)=\{e\}$.
\item[\rm(c)]
Every compact open subgroup $V\sub G$ which is tidy
above for~$\alpha$ is tidy.
\item[\rm(d)]
$\con(\alpha)$ is closed in~$G$.
\end{itemize}
\end{prop}
\begin{proof}
The implication ``(a)$\impl$(b)" is trivial, while ``(b)$\impl$(c)" is immediate from Proposition~\ref{when-TB}.
If (c) holds and $U\sub G$ is an identity neighbourhood,
pick a compact open subgroup $V\sub U$. After replacing $V$ with $\bigcap_{j=0}^\ell\alpha^{-j}(V)$
for some $\ell\in\N_0$, we may assume that $V$ is tidy above for~$\alpha$ and hence tidy,
as we assume~(c). Hence (c) implies (a).\\[2.3mm]
(b)$\impl$(d) holds by Lemma~\ref{conclosure2}.\\[2.3mm]
(d)$\impl$(c): By Lemma~\ref{modVpVm} and Theorem~A, we have $V_{--}=\con(\alpha, V_+\cap V_-)=\con(\alpha)(V_+\cap V_-)$,
which is closed in~$G$ as $V_+\cap V_-$ is compact and we assume that $\con(\alpha)$ is closed.
Hence $V$ is tidy.
\end{proof}
Tidy automorphisms were introduced in~\cite{Tid}.
\begin{defn}
An endomorphism
$\alpha$ of a totally disconnected, locally compact group
is called \emph{tidy} if it satisfies
one (hence any) of the conditions (a)--(d)
in Proposition~\ref{betterC}.
\end{defn}
\begin{cor}\label{subtidy}
	Let $\alpha$ be an endomorphism of a totally disconnected, locally compact group~$G$ and $H\sub G$ an $\alpha$-invariant closed subgroup. If $\alpha$ is tidy, then also $\alpha|_H$ is tidy.
	If $H$ is also normal and compact, then the endomorphism
	$\overline{\alpha}$ induced by $\alpha$ on $G/H$ is tidy. 
\end{cor}
\begin{proof}
	If $\alpha$ is tidy, then $\con(\alpha)$ is closed in~$G$ by Proposition \ref{betterC},
	whence $\con(\alpha)\cap H\!=\!\con(\alpha|_H)$ is closed in~$H$.
	Thus $\alpha|_H$ is tidy by the same~proposition.\\[2.3mm]
	Assume also $H$ is normal in $G$ and compact. Applying Theorem A, we have
$q^{-1}(\con(\wb{\alpha}))=\con(\alpha,H) = \con(\alpha)H$
	which is closed since $\con(\alpha)$ is closed by Proposition~\ref{betterC} and $H$ is compact.
	Hence $\con(\wb{\alpha})$ is closed and thus $\wb{\alpha}$ is tidy,
	by Proposition~\ref{betterC}.
\end{proof}
\section{Dynamics on {\boldmath$\parb^-(\alpha)$ and $\lev(\alpha)$}}
We now prove results involving
the restriction of an endomorphism\linebreak
$\alpha\colon G\to G$ to
the anti-parabolic subgroup $\parb^-(\alpha)$ or
the Levi subgroup $\lev(\alpha)=\parb(\alpha)\cap\parb^{-}(\alpha)$, to be used in the proof of Theorem~E.
\begin{la}\label{hencebeta}
Let $\alpha$ be an endomorphism of a totally disconnected, locally compact group~$G$
and $K$ be a compact subgroup of~$G$ such that $\alpha(K)=K$.
Set $\beta:=\alpha|_{\parb^-(\alpha)}$.
Then $K\sub\parb^-(\alpha)$ and
\[
\con^{-}(\alpha,K)=\con^{-}(\beta,K)\sub\parb^-(\alpha).
\]
In particular, $\con^-(\alpha)=\con^-(\beta)\sub\parb^-(\alpha)$.
\end{la}
\begin{proof}
Since $\alpha(K)=K$, we find an $\alpha$-regressive trajectory in~$K$ for each $x\in K$,
showing that $K\sub\parb^-(\alpha)$.
If $x\in\con^-(\alpha,K)$, then
there exists an $\alpha$-regressive trajectory $(x_n)_{n\in\N_0}$ for~$x$
such that $x_n\to e$ modulo $K$.
As $K$ is compact, $\{x_n\colon n\in\N_0\}$ is relatively compact
(see Lemma~\ref{specialkaes}\,(a)), entailing that $x_n\in\parb^-(\alpha)$ for each $n\in\N_0$ and thus
$x\in\con^-(\beta, K)$ (cf.\ Lemma~\ref{subuni}).
Thus $\con^-(\alpha,K)\sub \con^-(\beta,K)$. The converse inclusion is trivial (cf.\ Lemma~\ref{imageconv}\,(a)).
\end{proof}
Among other things, the next lemma is useful for the proof of Theorem~E.
Independently, part~(c) of the lemma was also obtained in \cite[Lemma~4.5\,(3)]{GBV}.
\begin{la}\label{fi-steps}
Let $\alpha$ be an endomorphism of a totally disconnected, locally compact group~$G$.
Then the following holds:
\begin{itemize}
\item[\rm(a)]
A compact open subgroup $U\sub G$ is tidy below for $\alpha$
if and only if $U \cap \parb^-(\alpha)$ is tidy below for
$\alpha|_{\parb^-(\alpha)}$.
\item[\rm(b)]
If $\alpha|_{\parb^-(\alpha)}$ is injective, then $\alpha|_{\parb^-(\alpha)}$ is an automorphism
of $\parb^-(\alpha)$.
\end{itemize}
If $U\sub G$ is a tidy subgroup for~$\alpha$, then
\begin{itemize}
\item[\rm(c)]
$U_+$ is a compact open subgroup of $\parb^-(\alpha)$
which is tidy for $\alpha|_{\parb^-(\alpha)}$,
and $s(\alpha)=s(\alpha|_{\parb^-(\alpha)})$.
\item[\rm(d)]
$U_+$ is a compact open subgroup of $U_{++}$ which is tidy for $\alpha|_{U_{++}}$,
and $s(\alpha)=s(\alpha|_{U_{++}})$.
\end{itemize}
\end{la}
\begin{proof}
(a) The subgroup $V:=U\cap\parb^-(\alpha)$ is compact and open in $P^-:=\parb^-(\alpha)$.
Let $\beta:=\alpha|_{P^-}$ and $V_+$ be the group of all
$x\in V$ for which there exists a $\beta$-regressive trajectory in~$V$.
If $x\in U_+$, then there exists an $\alpha$-regressive trajectory $(x_n)_{n\in\N_0}$
for~$x$ in~$U$. Since $\{x_n\colon n\in\N_0\}\sub U$ is relatively compact,
we have $x_n\in U\cap P^-=V$ for all $n\in \N_0$, whence $x\in V_+$.
Thus $U_+\sub V_+$ and
\begin{equation}\label{mightneed1}
U_+=V_+.
\end{equation}
If $x\in U_{++}$, then there exists an $\alpha$-regressive trajectory $(x_n)_{n\in \N_0}$
for~$x$ and $N\in\N_0$ such that $x_n\in U_+=V_+$ for all $n\geq N$.
Again, we deduce that $x_n\in P^-$ for all $n\in\N_0$.
This entails $x\in V_{++}$. Thus $U_{++}\sub V_{++}$ and hence
\begin{equation}\label{mightneed2}
U_{++}=V_{++}.
\end{equation}
Since $P^-\sub G$ is closed, $U_{++}$ is closed in~$G$
if and only if $V_{++}$ is closed in~$P^-$.
As $U_+=V_+\sub P^-$, we have $\alpha^n(U_+)=\beta^n(V_+)$
for each $n\in \N_0$ and thus
\begin{equation}\label{mightneed3}
[\alpha^{n+1}(U_+):\alpha^n(U_+)]=[\beta^{n+1}(V_+):\beta^n(V_+)]\quad\mbox{for all $\,n\in\N_0$.}
\end{equation}
As a consequence, the sequence
$([\alpha^{n+1}(U_+):\alpha^n(U_+)])_{n\in\N_0}$ is constant if and only if
$([\beta^{n+1}(V_+):\beta^n(V_+)])_{n\in\N_0}$ is a constant sequence.
Thus~$U$ is tidy below for $\alpha$ if and only if~$V$ is tidy below for~$\beta$.\\[2.3mm]
(c) Since $U_+ = P^- \cap U$ (see \cite[Proposition~11(1)]{End}),
$U_+$ is open in $P^-$.
As $U_+\sub \alpha(U_+)$,
the subgroup $V:=U_+=P^-\cap U$ is tidy above for $\beta:=\alpha|_{P^-}$
(since $V\sub \beta(V)$ implies $V_+=V$). By~(a), $V$ is also tidy below.
Therefore~$V$ is tidy for~$\beta$ and now (\ref{mightneed3}) yields
\[
s(\alpha)=[\alpha(U_+):U_+]=[\beta(V_+):V_+]=s(\beta).
\]
(d) Since $U_+=U_{++}\cap U$ by \cite[Proposition~11(1)]{End} and $U_{++}$ is closed,
$U_+$ is a compact open subgroup of~$U_{++}$. It is tidy above for $\alpha|_{U_{++}}$
since $\alpha(U_+)\supseteq U_+$
and tidy below since~$U_{++}$ and the indices $[\alpha^{n+1}(U_+):\alpha^n(U_+)]$
for $n\in\N_0$ are unchanged
if~$\alpha$ is replaced with~$\alpha|_{U_{++}}$.
Notably, $s(\alpha)=[\alpha(U_+):U_+]$ remains unchanged.\\[2.3mm]
(b) The self-map $\beta:=\alpha|_{P^-}$ of~$P^-$
is surjective as $\alpha(P^-)=P^-$, and injective by hypothesis.
By the proof of~(c), $P^-$ has a compact open subgroup~$V$ such that
$V\sub \beta(V)$, whence also $\beta(V)$ is a
compact open subgroup of~$P^-$.
Since~$V$ is compact,
$\beta|_V$ is a quotient homomorphism onto its image and hence open onto~$\beta(V)$,
which is open in~$P^-$. Thus, the endomorphism $\beta$
is an open map.
\end{proof}
\begin{rem}\label{clo-aut}
Note that $\alpha|_{\parb^-(\alpha)}$ is injective if and only if $\bik(\alpha)=\{e\}$.
Hence $\alpha|_{\parb^-(\alpha)}$ is injective if $\alpha$ is tidy
(then $\bik(\alpha)\sub\nub(\alpha)=\{e\}$).
\end{rem}
The following result is known for automorphisms
(cf.\ \cite[Theorem 3.32]{BaW}).
\begin{prop}\label{conanticon}
Let $\alpha$ be an endomorphism of a totally disconnected, locally compact group~$G$.
If $\con(\alpha)$ is closed in~$G$,
then also $\con^-(\alpha)$ closed.
\end{prop}
\begin{proof}
As $\con(\alpha)$ is closed, $\beta:=\alpha|_{\parb^-(\alpha)}$ is an automorphism (see
Remark~\ref{clo-aut} and Proposition~\ref{betterC}).
Since $W_+$ is tidy for~$\beta$ for each compact open subgroup $W\sub G$
which is tidy for~$\alpha$ (see Lemma~\ref{fi-steps}\,(c)),
$\beta$ has small tidy subgroups. As each of these is also tidy for~$\beta^{-1}$,
Proposition~\ref{betterC} shows that $\con(\beta^{-1})$ is closed in $\parb^-(\alpha)$
and hence in~$G$.
Since $\con^-(\alpha)=\con^-(\beta)=\con(\beta^{-1})$ by Lemma~\ref{hencebeta},
the assertion follows.
\end{proof}
\begin{rem}
If $\con^-(\alpha)$ is closed, then $\con(\alpha)$ need not be closed.
For example, let $F\not=\{e\}$ be a finite group and $G:=F^\N$ with the product topology.
Consider the left shift
\[
\alpha\colon G\to G,\quad (x_n)_{n\in\N}\mto (x_{n+1})_{n\in\N}
\]
with kernel $F\times \{e\}^{\{2,3,\ldots\}}$. Then $\con(\alpha)$ is the dense proper subgroup of all
finitely supported sequences (hence not closed). Let
\[
\sigma\colon G\to G, \quad x=(x_n)_{n\in \N}\mto (e,x_1,x_2,\ldots)
\]
be the right shift.
If $x\in G$, then $(\sigma^n(x))_{n\in\N_0}$ is an $\alpha$-regressive trajectory for $x$ such that $\sigma^n(x)\to e$,
whence $x\in\con^-(\alpha)$. Thus $\con^-(\alpha)=G$ is closed.
\end{rem}
We mention that the proof of Lemma~3.31 in~\cite{BaW}
(and hence \cite[Theorem~3.32]{BaW})
uses a tidying procedure
for compact open subgroups which was developed in \cite{WJO} and is different from the one of Section \ref{tidyingp}.
It involves a certain compact subgroup $K\subseteq G$ which,
following \cite[\S2, Step 2a, p.\,4]{WJO},
is defined as
\begin{equation}\label{defntheK}
K:=\bigcap_{O\in \COS(G)} K_O,
\end{equation}
where 
$K_O:= \wb{\lev(\alpha)\cap O_{--}}$
(in \cite{WJO}, the abbreviation $\cK_O:=\lev(\alpha)\cap O_{--}$
is used).
If $G$ is metrizable, then $K$ admits a simpler description,
as the closure
\begin{equation}\label{simplerdes}
K\;=\; \wb{\con(\alpha)\cap \lev(\alpha)}
\end{equation}
(i.e., $K=\wb{\con(\alpha)\cap \parb^-(\alpha)}$\hspace{.3mm});
this was asserted without proof in \cite{WJO},
and left open for non-metrizable~$G$. It was later shown by W. Jaworski (unpublished, personal communication,
2008). We now give a shorter argument.
\begin{prop}\label{the-group-K}
Let $\alpha$ be an endomorphism of a totally disconnected,
locally compact
group~$G$. Then
\begin{equation}\label{shorter-sol}
\wb{\con(\alpha)\cap\lev(\alpha)}\;=\; \bigcap_{O\in \COS(G)} K_O\,.
\end{equation}
\end{prop}
\begin{proof}
For $O\in\COS(G)$, we have $O_+\cap O_-\sub\lev(\alpha)$ and thus
\begin{equation}\label{fdsin}
O_{--}\cap\lev(\alpha)= \con(\alpha,O_+\cap O_-)\cap \lev(\alpha)=\con(\alpha|_{\lev(\alpha)},O_+\cap O_-),
\end{equation}
using Lemma~\ref{modVpVm} for the first equality and then Lemma~\ref{sammelsub}\,(a).
Hence
\begin{eqnarray*}
K&\supseteq & \bigcap_{O\in\COS(G)}\con(\alpha|_{\lev(\alpha)}, O_+\cap O_-)\\
&=& \con\Big( \alpha|_{\lev(\alpha)},\bigcap_{O\in\COS(O)}(O_+\cap O_-)\Big)
=\con(\alpha|_{\lev(\alpha)})
\end{eqnarray*}
using Lemma~\ref{intersect-famly} and the fact that $\bigcap_{O\in\COS(G)}(O_+\cap O_-)\sub \bigcap_{O\in\COS(G)}O=\{e\}$.
So $K\supseteq \wb{\con(\alpha|_{\lev(\alpha)})}=\wb{\con(\alpha)\cap\lev(\alpha)}$,
with equality by Lemma~\ref{sammelsub}\,(a).\\[2.3mm]
To verify the converse inclusion, note that
\[
K_O=\wb{\con(\alpha|_{\lev(\alpha)},O_+\cap O_-)}=\wb{\con(\alpha|_{\lev(\alpha)})(O_+\cap O_-)}
=\wb{\con(\alpha|_{\lev(\alpha)})}(O_+\cap O_-)
\]
by (\ref{fdsin}), Theorem~A and compactness of $O_+\cap O_-$. If
\[
x\in K=\bigcap_{O\in \COS(G)}K_O=\bigcap_{O\in\COS(G)}\wb{\con(\alpha|_{\lev(\alpha)})}(O_+\cap O_-),
\]
for each $O\in \COS(G)$ we can write
\[
x=x_Oy_O
\]
with $x_O \in \wb{\con(\alpha|_{\lev(\alpha)})}$ and $y_O\in O_+\cap O_-$.
Direct $\COS(G)$ by declaring $O_1\leq O_2$ if $O_2 \sub O_1$.
Then the net $(y_O)_{O\in\COS(G)}$ converges to~$e$, whence
\[
x_O=xy_O^{-1}\to x.
\]
Thus $x\in \wb{\con(\alpha|_{\lev(\alpha)})}$, showing that
$K\sub\wb{\con(\alpha|_{\lev(\alpha)})}=\wb{\con(\alpha)\cap\lev(\alpha)}$.\vspace{.5mm}
\end{proof}
We shall see later that $K=\nub(\alpha)$ (Lemma~\ref{keynub}\,(e)).
\section{Proof of Theorem~E}\label{sec:proof_E}
We shall use a lemma concerning induced endomorphisms on quotients.
For a precursor of~(a) and~(b)
for (inner) automorphisms,
compare \cite[Lemma~3.6]{MIX}
and its variant in \cite[Lemma~4.9\,(iv)]{Re2}.
\begin{la}\label{quottidy}
Let $\alpha$ be an endomorphism of a totally disconnected,
locally compact group~$G$ and $N\sub G$ be an $\alpha$-stable compact
normal subgroup.
Let $q\colon G\to G/N$ be the canonical quotient map
and $\wb{\alpha}$ be the endomorphism
of $G/N$ induced by~$\alpha$. Then the following holds:
\begin{itemize}
\item[\rm(a)]
Let $U\sub G$ be a compact open subgroup such that $N\sub U$.
Then~$U$ is tidy for~$\alpha$ if and only if~$q(U)$ is tidy for~$\wb{\alpha}$.
\item[\rm(b)]
$s(\alpha)=s(\wb{\alpha})$.
\item[\rm(c)]
If $N\sub\nub(\alpha)$, then $q(\nub(\alpha))=\nub(\wb{\alpha})$ and $\nub(\alpha)=q^{-1}(\nub(\wb{\alpha}))$.
\end{itemize}
\end{la}
\begin{proof}
(a) For $x\in U$ and $n\in\N_0$, we have $\wb{\alpha}^n(q(x))=q(\alpha^n(x))\in q(U)$
if and only if $\alpha^n(x)\in UN=U$. Hence $q(x)\in q(U)_-$ if and only if $x\in U_-$,
entailing
\[
U_-=q^{-1}(q(U)_-)\quad \mbox{and hence also}\quad q(U_-)=q(U)_-.
\]
If $x\in U_+$, then there exists an $\alpha$-regressive trajectory $(x_n)_{n\in\N_0}$ for
$x$ in~$U$. Then $(q(x_n))_{n\in\N_0}$ is an $\wb{\alpha}$-regressive trajectory for $q(x)$ in $q(U)$
and hence $q(x)\in q(U)_+$. Thus
\begin{equation}\label{semigood}
q(U_+)\sub q(U)_+.
\end{equation}
If $q(x)\in q(U)_+$ and $(y_n)_{n\in\N_0}$ is an $\wb{\alpha}$-regressive trajectory for~$q(x)$ in $q(U)$,
pick $z_n\in q^{-1}(\{y_n\})$ for $n\in\N_0$.
Then
\[
q(\alpha^i(z_n))=\wb{\alpha}^i(q(z_n))=\wb{\alpha}^i(y_n)=y_{n-i}\in q(U)
\]
for $i\in \{0,1,\ldots,n\}$, showing that $\alpha^i(z_n)\in q^{-1}(q(U))=U$.
Since $q(\alpha^n(z_n))=y_0=q(x)$, we have $x\alpha^n(z_n)^{-1}\in\ker(q)=N$.
Since $\alpha(N)=N$, we can choose $a\in N$ such that $\alpha^n(a)=x\alpha^n(z_n)^{-1}$.
Then $a z_n\in U$, $\alpha^i(a z_n)=\alpha^i(a)\alpha^i(z_n)\in NU=U$ for $i\in \{0,\ldots,n\}$
and $\alpha^n(az_n)=\alpha^n(a)\alpha^n(z_n)=x$. Hence $x\in U_{n,\alpha}$ and thus
\[
x\in\bigcap_{n\in\N_0} U_{n,\alpha}=U_+.
\]
Thus $q^{-1}(q(U)_+)\sub U_+$. Together with (\ref{semigood}),
this implies $U_+=q^{-1}(q(U)_+)$ and $q(U_+)=q(U)_+$.\\[2.3mm]
Since $U=UN$ and
\[
U_+U_-=U_+U_-N,
\]
the final equality in the chain $q(U)_+q(U)_-=q(U_+)q(U_-)=q(U_+U_-)=q(U)$
holds if and only if $U_+U_-=U$. Thus $U$ is tidy above for $\alpha$ if and only if $q(U)$
is tidy above for~$\wb{\alpha}$.
We have
\[
q(\alpha^n(U_+))=\wb{\alpha}^n(q(U_+))=\wb{\alpha}^n(q(U)_+)
\]
for each $n\in\N_0$. Since $\alpha^n(U_+)\supseteq U_+=q^{-1}(q(U)_+)\supseteq N$,
we deduce that
\[
\alpha^n(U_+)=q^{-1}(\wb{\alpha}^n(q(U)_+)).
\]
Hence
\begin{eqnarray*}
U_{++}&=&\bigcup_{n\in\N_0}\alpha^n(U_+)=\bigcup_{n\in\N_0}q^{-1}(\wb{\alpha}^n(q(U)_+))\\
&=& q^{-1}\left(\bigcup_{n\in\N_0}\wb{\alpha}^n(q(U)_+)\right)=q^{-1}(q(U)_{++}),
\end{eqnarray*}
showing that $U_{++}$ is closed if and only if $q(U)_{++}$ is closed.
Finally,
\begin{eqnarray}
[\alpha^{n+1}(U_+):\alpha^n(U_+)]&=& [q^{-1}(\wb{\alpha}^{n+1}(q(U)_+)):q^{-1}(\wb{\alpha}^n(q(U)_+))]\notag \\
&=& [\wb{\alpha}^{n+1}(q(U_+)):\wb{\alpha}^n(q(U_+))],\label{for-scaa}
\end{eqnarray}
whence the sequence $([\alpha^{n+1}(U_+):\alpha^n(U_+)])_{n\in\N_0}$
is constant if and only if so is $([\wb{\alpha}^{n+1}(q(U)_+):\wb{\alpha}^n(q(U)_+)])_{n\in\N_0}$.
Thus $U$ is tidy if and only if $q(U)$ is tidy.\\[2.3mm]
(b) Let $V\sub G/N$ be a compact open subgroup which is tidy for~$\wb{\alpha}$.
Then $U:=q^{-1}(V)$ is a compact open subgroup of~$G$ such that $N\sub U$ and $q(U)=V$ is tidy,
whence~$U$ is tidy, by~(a). Using~(\ref{for-scaa}), we get
\[
s(\alpha)=[\alpha(U_+):U_+]=[\wb{\alpha}(q(U)_+),q(U)_+]=[\wb{\alpha}(V_+):V_+]=s(\wb{\alpha}).
\]
(c) As $N\sub\nub(\alpha)\sub U$ for each $U\in \TID(\alpha)$,
we deduce from~(a) that
\[
\TID(\alpha)=\{q^{-1}(V)\colon V\in\TID(\wb{\alpha})\}.
\]
Hence
\[
\nub(\alpha)=\bigcap_{V\in\TID(\wb{\alpha})}q^{-1}(V)=q^{-1}\left(\bigcap_{V\in\TID(\wb{\alpha})}V\right)
=q^{-1}(\nub(\wb{\alpha}))
\]
and thus $q(\nub(\alpha))=\nub(\wb{\alpha})$.
\end{proof}
\begin{proof}
(Theorem E). Let $\beta:=\alpha|_{\parb^-(\alpha)}$ and $\wb{\alpha}$
be the automorphism induced by~$\beta$ on $\parb^-(\alpha)/\bik(\alpha)$.
Let $q\colon \parb^-(\alpha)\to\parb^-(\alpha)/\bik(\alpha)$ be the
canonical quotient map, whose kernel $\bik(\alpha)$ is compact.
Then
\begin{equation}\label{the-chain}
s(\alpha)=s(\beta)=s(\wb{\alpha})=s(\wb{\alpha}|_{\wb{\con^-(\wb{\alpha})}}),
\end{equation}
using Lemma~\ref{fi-steps}\,(c), Lemma~\ref{quottidy}\,(b) or Proposition \ref{prop:Janalysis},
as well as \cite[Proposition~3.21]{BaW}.
The first and second equality were obtained independently in \cite[Lemma~4.5 (3), (4)]{GBV}.
We have $\con^-(\beta,\bik(\alpha))=\con^-(\beta)\bik(\alpha)$ by Theorem~B
and thus $q(\con^-(\alpha))$
$=q(\con^-(\beta))=\con^-(\wb{\alpha})$.
Since $\bik(\alpha)$ is compact, we deduce that
\[
\wb{\con^-(\alpha)}\bik(\alpha)
\]
is closed.
It then coincides with the set $\wb{\con^-(\alpha)\bik(\alpha)}$, which is mapped onto $\wb{\con^-(\wb{\alpha})}$
by~$q$. Hence
\[
p:=q|_{\wb{\con^-(\alpha)}}\colon\wb{\con^-(\alpha)}\to\wb{\con^-(\wb{\alpha})}
\]
is a continuous surjective homomorphism and hence a quotient morphism,
as both groups are $\sigma$-compact (whence the open mapping theorem, \cite[(5.29)]{HaR} applies):
In fact, as the left hand side is a closed subgroup of $\wb{\con^-(\alpha)}\bik(\alpha)$ and this group has $\wb{\con^-(\wb{\alpha})}$
as a quotient with compact kernel, it will be $\sigma$-compact if $\wb{\con^-(\wb{\alpha})}$
is $\sigma$-compact, which
holds by Lemma~\ref{con-sep} to come.\\[2.3mm]
Since~$p$ has compact kernel, by applying 
Lemma~\ref{quottidy}\,(b) or Proposition \ref{prop:Janalysis} we deduce that  $s(\alpha|_{\wb{\con^-(\alpha)}})
=s(\wb{\alpha}|_{\wb{\con^-(\wb{\alpha})}})$.
Together with (\ref{the-chain}), we get $s(\alpha)=s(\alpha|_{\wb{\con^-(\alpha)}})$.\vspace{2mm}
\end{proof}

\noindent
We applied the following lemma to the automorphism $\wb{\alpha}^{\,-1}$:
\begin{la}\label{con-sep}
For each automorphism $\alpha\colon G\to G$ of a locally compact group, the closure
$\wb{\con(\alpha)}$ of the contraction group
is $\sigma$-compact.
\end{la}
\begin{proof}
Let $K$ be a compact identity neighbourhood in~$\wb{\con(\alpha)}$.
Then
\[
\wb{\con(\alpha)}=\con(\alpha)K.
\]
For each $x\in\con(\alpha)$, there is $n\in\N$ with $\alpha^n(x)\in K$
whence $x\in\alpha^{-n}(K)$. Thus
\[
\wb{\con(\alpha)}=\bigcup_{n\in\N} \alpha^{-n}(K)K
\]
is $\sigma$-compact.
\end{proof}
As an application of Theorem~B, Proposition~\ref{prop:Janalysis} and Theorem~E, we include the following addition to the study begun in Section~\ref{sec:scale_sub_quot}. Note that since $\alpha$-stable compact groups are contained within $\parb^{-}(\alpha)$, this result is a generalization of Lemma~\ref{quottidy}\,(b). It also generalizes Proposition~\ref{prop:Janalysis} which can be seen by choosing $G = \parb^{-}(\alpha)$.
As a consequence, the result
is also a generalization of \cite[Proposition~3.21\,(2)]{BaW} which is precisely Proposition~\ref{prop:Janalysis} restricted to automorphisms.
\begin{thm}\label{thm:scale_equality}
	Let $G$ be a totally disconnected locally compact group and $\alpha$ an endomorphism of $G$. Further, let $N\subseteq G$ be a closed normal subgroup which is $\alpha$-stable and contained in $\parb^{-}(\alpha)$. If $\overline{\alpha}$ is the endomorphism induced by $\alpha$ on $G/N$, then
	\[s_{G}(\alpha) = s_{G/N}(\bar{\alpha})s_{N}(\alpha).\]
\end{thm}
\begin{proof}
Applying
Lemma~\ref{fi-steps} and Proposition \ref{prop:Janalysis}, it suffices to show 
\[s_{G/N}(\overline{\alpha}) = s_{\parb^{-}(\alpha)/N}(\overline{\alpha}),\]
where the right hand side denotes the scale of the restriction of~$\wb{\alpha}$ to
$\parb^-(\alpha)/N$.
	Let $xN\in \con^{-}(\overline{\alpha})$. Then $x\in \con^{-}(\alpha, N)$. Thus $x\in \con^{-}(\alpha) N$ by Theorem~B. Therefore, $xN \in \parb^{-}(\alpha)/N$ as $\con^{-}(\alpha)\sub \parb^{-}(\alpha)$. Since $\parb^{-}(\alpha)$ is closed we have $\overline{\con^{-}(\overline{\alpha})}\sub \parb^{-}(\alpha)/N.$
	
	By Theorem~E we have $s(\overline{\alpha}) = s(\overline{\alpha}|_{\overline{\con^{-}(\overline{\alpha})}})$,
	so applying Proposition \ref{prop:scale_subgroup} we get $s_{G/N}(\overline{\alpha}) = s_{\overline{\con^{-}(\overline{\alpha})}}(\overline{\alpha}) \le s_{\parb^{-}(\alpha)/N}(\overline{\alpha}) \le s_{G/N}(\overline{\alpha})$ as required.
\end{proof}
\section{Further results concerning the nub}
In this section, we provide further results concerning the nub of an endomorphism,
as well as consequences, for example an improved description
of the closure of a contraction group.
\begin{la}\label{easy-lev}
If $\alpha$ is an endomorphism of a totally disconnected,
locally compact group~$G$, then
\[
\lev(\alpha)=\lev(\alpha|_{\parb^-(\alpha)})\qquad\mbox{and}\quad
\bik(\alpha)=\bik(\alpha|_{\lev(\alpha)}).
\]
\end{la}
\begin{proof}
Abbreviate $\gamma:=\alpha|_{\lev(\alpha)}$.
If $x\in \parb^-(\alpha)$ and $\alpha^n(x)=e$ for some $n$,
then $x\in\lev(\alpha)$. Hence
\[
\{x\in \parb^-(\alpha)\colon (\exists n\in \N_0)\; \alpha^n(x)=e\}
=\{x\in \lev(\alpha) \colon (\exists n\in\N_0)\;\gamma^n(x)=e\}.
\]
Since $\lev(\alpha)$ is closed, the closures of the preceding set
in $G$ and $\lev(\alpha)$ coincide. Thus $\bik(\alpha)=\bik(\gamma)$.\\[2.3mm]
If $x\in\lev(\alpha)$, there exists a bounded $\alpha$-trajectory $(x_n)_{n\in\Z}$
with $x_0=x$.
Then $x_n\in\lev(\alpha)$ for each~$n$, since $(x_{n+m})_{m\in \Z}$
is a bounded $\alpha$-trajectory for this element. Thus $(x_n)_{n\in \Z}$
is a bounded $\gamma$-trajectory for~$x$ in~$\lev(\alpha)$
(and hence also a bounded $\alpha|_{\parb^-(\alpha)}$-trajectory for~$x$
in $\parb^-(\alpha)$). Thus $x\in\lev(\alpha|_{\parb^-(\alpha)})$
and hence $\lev(\alpha)\sub \lev(\alpha|_{\parb^-(\alpha)})$.
The converse inclusion is trivial.
\end{proof}
The following statements (a) and (b) are implicit
in comments concerning the nub in \cite[lines following Remark~4]{End}
(where no proof is given);
independently, they were verified in
\cite[Lemma~4.5\,(5)]{GBV}.
\begin{prop}\label{keynub}
Let $\alpha$ be an endomorphism of a totally disconnected, locally compact
group~$G$. Let $q\colon \parb^-(\alpha)\!\to\!\parb^-(\alpha)/\bik(\alpha)=:Q$
be the canonical quotient map and $\wb{\alpha}$ be the automorphism of~$Q$
induced by $\alpha|_{\parb^-(\alpha)}$.
Then
\begin{itemize}
\item[\rm(a)]
$\nub(\alpha|_{\parb^-(\alpha)})=\nub(\alpha)$,
\item[\rm(b)]
$\nub(\wb{\alpha})=q(\nub(\alpha))$,
\item[\rm(c)]
$\nub(\alpha)=\wb{\con(\alpha)\cap\nub(\alpha)}$,
\item[\rm(d)]
$\wb{\con(\alpha)}=\bigcap_{U\in\TID(\alpha)}U_{--}=\con(\alpha)\nub(\alpha)$, and
\item[\rm(e)]
$\nub(\alpha)=\wb{\con(\alpha)\cap\lev(\alpha)}=\bigcap_{U\in\COS(G)}\wb{\lev(\alpha)\cap U_{--}}$\\
\hspace*{14mm}$=\lev(\alpha)\cap\bigcap_{U\in \TID(\alpha)} U_{--}=
\bigcap_{V\in\COS(\lev(\alpha))}\wb{V_{--}}$.
\end{itemize}
\end{prop}
\begin{proof}
(a) Abbreviate $P^-:=\parb^-(\alpha)$.
We recall from \cite[Corollary~5]{End} that $\nub(\alpha)
=\bigcap_{U\in\TID(\alpha)}U=\bigcap_{U\in \TID(\alpha)}(U_+\cap U_-)$, whence also
\begin{equation}\label{prenb}
\nub(\alpha)=\bigcap_{U\in\TID(\alpha)} U_+.
\end{equation}
The right hand side of (\ref{prenb}) contains $\nub(\alpha|_{P^-})$,
by Lemma~\ref{fi-steps}\,(c). Hence $\nub(\alpha)\supseteq \nub(\alpha|_{P^-})$.
For the reverse inclusion we show the contrapositive. Suppose $x\not\in \nub(\alpha|_{P^{-}})$. Then $x\not\in W$ for some compact open subgroup $W\sub P^-$ which is
tidy for $\alpha|_{P^-}$.
Let $O$ be a compact open subgroup of~$G$ with $O\cap P^-\sub W$.
Then
\[
V:=\bigcap_{y\in W}yOy^{-1}
\]
is a compact open subgroup of~$G$ which is normalized by~$W$,
entailing that $VW$ is a compact open subgroup of~$G$.
Note that $x\not\in VW$ (if we could write $x=vw$ with $v\in V$ and $w\in W$,
then $v=xw^{-1}\in V\cap P^- \sub O\cap  P^- \sub W$,
whence $x=vw\in W$ contrary to the choice of~$W$).
For large $\ell\in\N_0$,
\[
U:=\bigcap_{j=0}^\ell\alpha^{-j}(VW)
\]
is tidy above for~$\alpha$.
By Lemma~\ref{simple-obs},
after increasing $\ell$ if necessary, we may assume that, moreover,
$P^-\cap U$ is tidy above for $\alpha|_{P^-}$.
Since $\nub(\alpha|_{P^-})\sub W$ and
$\alpha(\nub(\alpha|_{P^-}))=\nub(\alpha|_{P^-})$,
we have
$\nub(\alpha|_{P^-})\sub U$
and hence
\[
\nub(\alpha|_{P^-})\sub U\cap P^-.
\]
As a consequence, $U\cap P^-$ is tidy for~$\alpha|_{P^-}$,
whence $U$ is tidy below for~$\alpha$ (by Lemma~\ref{fi-steps}\,(a))
and hence tidy. Thus $x\not\in \nub(\alpha)\sub U$.
This shows that $\nub(\alpha)=\nub(\alpha|_{P^-})$.\\[2.3mm]
(b) As $\ker(q)=\bik(\alpha)\sub \nub(\alpha)=\nub(\alpha|_{P^-})$ (using~(a)),
we have $q(\nub(\alpha))=q(\nub(\alpha|_{P^-}))=\nub(\wb{\alpha})$ by Lemma~\ref{quottidy}\,(c).\\[2.3mm]
(c) By (b), the map $p:=q|_{\nub(\alpha)}\colon \nub(\alpha)\to\nub(\wb{\alpha})$ is
onto and thus a quotient map by the open mapping theorem \cite[(5.29)]{HaR}, as $\nub(\alpha)$ is compact. Hence
\[
q(\con(\alpha)\cap \nub(\alpha))=q(\con(\alpha|_{\nub(\alpha}))=\con(\wb{\alpha}|_{\nub(\wb{\alpha})}),
\]
by Theorem~A,
which is dense in $\nub(\wb{\alpha})$ by \cite[Lemma~3.31, (1) and (2)]{BaW}.
As a consequence,
\begin{equation}\label{also-dense}
\con(\alpha|_{\nub(\alpha)})\ker(p)=
\con(\alpha|_{\nub(\alpha)})\bik(\alpha)
\end{equation}
is dense in $\nub(\alpha)$.
The subset
\begin{eqnarray*}
\{x\in \parb^-(\alpha)\colon (\exists n\in\N_0)\; \alpha^n(x)=e\} &\sub & \bik(\alpha)\cap \con(\alpha)\sub
\nub(\alpha)\cap \con(\alpha)\\
&=&\con(\alpha|_{\nub(\alpha)})
\end{eqnarray*}
is dense in $\bik(\alpha)$.
Combining this with the density of the set from (\ref{also-dense}) in $\nub(\alpha)$,
we find that $\con(\alpha|_{\nub(\alpha)})$ is dense in~$\nub(\alpha)$ (as asserted).\\[2.3mm]
(d) We know from Lemma~\ref{conclosure2} that $\wb{\con(\alpha)}=\con(\alpha)(\wb{\con(\alpha)}\cap\nub(\alpha))$.
By~(c), $\con(\alpha)\cap \nub(\alpha)$ is dense in $\nub(\alpha)$, and thus
$\wb{\con(\alpha)}\cap\nub(\alpha)=\nub(\alpha)$. Hence $\wb{\con(\alpha)}=\con(\alpha)\nub(\alpha)$.
It only remains to recall that
\begin{eqnarray*}
\bigcap_{U\in\TID(\alpha)}U_{--}&=& \bigcap_{U\in\TID(\alpha)}\con(\alpha,U_+\cap U_-)
=\bigcap_{U\in\TID(\alpha)}\con(\alpha, U)\\
&=& \con(\alpha,\nub(\alpha))=\con(\alpha)\nub(\alpha)
\end{eqnarray*}
by Lemma~\ref{modVpVm}, Lemma~\ref{stable-enough}, Lemma~\ref{intersect-famly}, and Theorem~A.\\[2.3mm]
(e) As $\{U\cap \lev(\alpha)\colon U\in\COS(G)\}$ is a basis
of identity neighbourhoods in~$\lev(\alpha)$
and $\lev(\alpha)\cap U_{--}=(\lev(\alpha)\cap U)_{--}$,
we have $\bigcap_{U\in\COS(G)}\wb{\lev(\alpha)\cap U_{--}}$
$=\bigcap_{V\in\COS(\lev(\alpha))}\wb{V_{--}}$.
The second and third groups in~(e) (from the left) coincide by Lemma~\ref{the-group-K},
and we denote them by~$K$ (as in~(\ref{defntheK})).
Now $K=\wb{\con(\alpha)\cap\lev(\alpha)}\supseteq \wb{\con(\alpha)\cap\nub(\alpha)}=\nub(\alpha)$,
by Proposition~\ref{keynub}\,(c). On the other hand,
\[
K=\bigcap_{U\in\COS(G)}\wb{\lev(\alpha)
\cap U_{--}}\sub \bigcap_{U\in\TID(\alpha)}\wb{\lev(\alpha)\cap U_{--}}\sub \nub(\alpha)
\]
as each of the sets $\lev(\alpha)\cap U_{--}$ is closed,
and contained in $U_+\cap U_-$, as we now verify:
If $x\in\lev(\alpha)\cap U_{--}$, then $x\in\lev(\alpha)$ and $\alpha^n(x)\in U_-$
for some $n\in \N_0$. Let $(x_n)_{n\in\N_0}$ be a bounded $\alpha$-regressive trajectory
for~$x$.
Then $\alpha^n(x),\alpha^{n-1}(x),\ldots, \alpha(x),x,x_1,x_2,\ldots$ is
a bounded $\alpha$-regressive trajectory for $\alpha^n(x)\in U$,
whence $\alpha^n(x)\in U_+$ by \cite[Proposition~11\,(a)]{End} and thus $\alpha^n(x)\in U_+\cap U_-$.
Since $\alpha(U_+\cap U_-)=U_+\cap U_-$, there exists and $\alpha$-regressive trajectory $(y_j)_{j\in\N_0}$
for $\alpha^n(x)$ inside $U_+\cap U_-$. Then $y_n^{-1}x\in (U_+\cap U_-)\lev(\alpha)\sub\lev(\alpha)\sub
\parb^-(\alpha)$ and $\alpha^n(y_n^{-1}x)=e$, showing that $y_n^{-1}x\in\bik(\alpha)\sub\nub(\alpha)$.
Hence
\[
x=y_n(y_n^{-1}x)\in (U_+\cap U_-)\nub(\alpha)=U_+\cap U_-.
\]
Thus $\lev(\alpha)\cap U_{--}\sub U_+\cap U_-$ and thus
$K\sub\bigcap_{U\in\TID(\alpha)}(\lev(\alpha)\cap U_{--})\sub \bigcap_{U\in\TID(\alpha)}U_+\cap U_-=\nub(\alpha)$.
\end{proof}
Proposition~\ref{keynub}\,(e) provides several characterizations
of the nub, and identifies it as the group considered in
Proposition~\ref{the-group-K} (which had been used in the literature
in the case of automorphisms, as recalled above).
A complementary characterization of the nub was obtained in \cite[Corollary 4.6]{GBV}.
\section{Theorem~F: Dynamics on the `big cell'}\label{sec-cell}
Before we prove Theorem~F, let us collect some tools.
For automorphisms, part (a), (d), (e) and (f) of the following lemma
are known (cf.\
Proposition~3.4,
Corollary 3.17, and
Lemma 3.18 in~\cite{BaW}).
\begin{la}\label{parblev}
Let $\alpha$ be an endomorphism of a totally disconnected,
locally compact group~$G$.
Then the following holds:
\begin{itemize}
\item[\rm(a)]
$\con(\alpha)$ is a normal subgroup of $\parb(\alpha)$
and $\con^-(\alpha)$ is normal in $\parb^-(\alpha)$.
\item[\rm(b)]
$V_-=\con(\alpha|_{V_-})(V_+\cap V_-)$ for each compact open subgroup $V\sub G$,
and $\con(\alpha|_{V_-})$ is normal in~$V_-$.
If $V$ is tidy for~$\alpha$, then $V_+\cap V_-=\lev(\alpha)\cap V$ and
$\con(\alpha|_{V-})=\con(\alpha)\cap V$.
\item[\rm(c)]
$V_+=(\con^-(\alpha)\cap V_+)(V_+\cap V_-)$ for each compact open subgroup $V\sub G$,
and $\con^-(\alpha)\cap V_+$
is normal in~$V_+$.
If $V$ is tidy, then $\con^-(\alpha)\cap V_+=\con^-(\alpha)\cap V$.
\item[\rm(d)]
$\parb(\alpha)=\con(\alpha)\lev(\alpha)$. Moreover, the surjective homomorphism\linebreak
$q\colon\lev(\alpha)\to \parb(\alpha)/\wb{\con(\alpha)}$, $g\mto g\, \wb{\con(\alpha)}$
is continuous and open.
\item[\rm(e)]
$\parb^-(\alpha)=\con^-(\alpha)\lev(\alpha)$. Moreover, the surjective homomorphism\linebreak
$p\colon
\lev(\alpha)\to \parb^-(\alpha)/\wb{\con^-(\alpha)}$, $g\mto g\, \wb{\con^-(\alpha)}$
is continuous and open.
\item[\rm(f)]
A compact open subgroup $V\sub\lev(\alpha)$ is tidy for $\alpha|_{\lev(\alpha)}$ if and only if
$\alpha(V)=V$.
\end{itemize}
\end{la}
\begin{proof}
(a) If $(x_n)_{n\in\N_0}$ and $(g_n)_{n\in\N_0}$
are sequences in~$G$
such that $x_n\to e$ and the set $\{g_n\colon n\in\N_0\}$
is relatively compact, then
also $g_nx_ng_n^{-1}\to e$ by virtue of \cite[(4.9)]{HaR}. The assertions follow from this observation.\\[2.3mm]
(b) If $x\in V_-\sub V_{--}$, then $\alpha^n(x)\to e$ modulo $V_+\cap V_-$
in~$G$ (see Lemma~\ref{modVpVm}) and hence also in~$V_-$ (see Lemma~\ref{subuni}).
Thus, using Theorem~A,
\[
V_-=\con(\alpha|_{V_-},V_+\cap V_-)=\con(\alpha|_{V_-})(V_+\cap V_-).
\]
By~(a), $\con(\alpha|_{V_-})$ is normal in $\parb(\alpha|_{V_-})=V_-$.
If $V$ is tidy and $x\in \con(\alpha)\cap V$,
then $\{\alpha^n(x)\colon n\in\N_0\}$ is relatively compact
and thus $x\in V_-$ (by \cite[Proposition~11\,(2)]{End}). Hence $x\in\con(\alpha|_{V_-})$
and so $\con(\alpha)\cap V\sub  \con(\alpha|_{V_-})$.
The converse inclusion is trivial.\\[2.3mm]
Finally, $V_+\cap V_-=\lev(\alpha)\cap V$ by \cite[Proposition~11]{End}.\\[2.3mm]
(c) As $\con^-(\alpha)$ is normal in $\parb^-(\alpha)$, the intersection
$\con^-(\alpha)\cap V_+$ is normal in the subgroup~$V_+$ of $\parb^-(\alpha)$.
If $x\in V_+\sub V_{++}$, then $x\in \con^-(\alpha, V_+\cap V_-)$ by Lemma~\ref{modVpVm}
and thus $x\in (\con^-(\alpha)\cap V_+)(V_+\cap V_-)$, by Theorem~B.
Trivially, $(\con^-(\alpha)\cap V_+)(V_+\cap V_-)\sub V_+$.
If $V$ is tidy, then $\con^-(\alpha)\cap V\sub V_+$ (by \cite[Proposition~11\,(1)]{End}),
whence $\con^-(\alpha)\cap V_+=\con^-(\alpha)\cap V$.\\[2.3mm]
(d) If $x\in\parb(\alpha)$, then the closure $K:=\wb{\{\alpha^n(x)\colon n\in\N_0\}}$
is compact. Since $\alpha^n(x)\in K$ for all~$n$, we have $\alpha^n(x)\to_\cR K$
and thus $\alpha^n(x)\to_\cR K_+\cap K_-$, by Lemma~\ref{stable-enough}.
Since $K_+\cap K_-$ is compact and $\alpha(K_+\cap K_-)=K_+\cap K_-$,
\[
K_+\cap K_-\sub\lev(\alpha).
\]
Thus $x\in \con(\alpha,K_+\cap K_-)\sub \con(\alpha,\lev(\alpha))=\con(\alpha)\lev(\alpha)$,
using Theorem~A. Thus $\parb(\alpha)\supseteq\con(\alpha)\lev(\alpha)$.
The converse inclusion is trivial.\\[2.3mm]
If $V\sub G$ is a compact open subgroup which is tidy for~$\alpha$,
then $V_-=V\cap \parb(\alpha)$ is open in $\parb(\alpha)$
and $V_+\cap V_-=V\cap\lev(\alpha)$ is open in $\lev(\alpha)$
(cf.\ \cite[Proposition~11]{End}).
Since $V_-=(\con(\alpha)\cap V)(V_+\cap V_-)$ by~(a), we see that
\[
q(V_+\cap V_-)=V_-\,\wb{\con(\alpha)}/\wb{\con(\alpha)}
\]
is open in~$\parb(\alpha)/\wb{\con(\alpha)}$. Since $V_+\cap V_-$ is compact,
the continuous surjection $q|_{V_+\cap V_-}\colon V_+\cap V_-\to q(V_+\cap V_-)$ is a quotient map
and thus an open map (being also a homomorphism). The assertion follows.\\[2.3mm]
(e) Every $x\in\parb^{-}(\alpha)$ admits a bounded $\alpha$-regressive trajectory $(x_{n})_{n\in\mathbb{N}_{0}}$. In particular, $K:=\overline{\{x_{n}\colon n\in\mathbb{N}_{0}\}}$ is compact. Since $x_{n}\in K$ for all $n$, we have $x_{n}(x)\to_{\mathcal{R}}K$ and thus $x_{n}\to_{\mathcal{R}}K_{+}\cap K_{-}$, by Lemma~\ref{stable-enough}. Since $K_{+}\cap K_{-}$ is compact and $\alpha(K_{+}\cap K_{-})=K_{+}\cap K_{-}$,
\begin{displaymath}
 K_{+}\cap K_{-}\subseteq\lev(\alpha).
\end{displaymath}
Thus $x\in\con^{-}(\alpha,K_{+}\cap K_{-})\subseteq\con^{-}(\alpha,\lev(\alpha))=\con^{-}(\alpha)\lev(\alpha)$, using Theorem~B. Thus $\parb^{-}(\alpha)\supseteq\con^{-}(\alpha)\lev(\alpha)$. The converse inclusion is trivial. \\[2.3mm]
If $V\sub G$ is a tidy subgroup for~$\alpha$,
then $V_+=V\cap \parb^-(\alpha)$ is open in $\parb^-(\alpha)$
and $V_+\cap V_-=V\cap\lev(\alpha)$ is open in $\lev(\alpha)$
(cf.\ \cite[Proposition~11]{End}).
Since $V_+=(\con^-(\alpha)\cap V_+)(V_+\cap V_-)$ by~(b), we see that
\[
p(V_+\cap V_-)=V_+\,\wb{\con^-(\alpha)}/\wb{\con^-(\alpha)}
\]
is open in~$\parb^-(\alpha)/\wb{\con^-(\alpha)}$. Since $V_+\cap V_-$ is compact,
the continuous surjection $p|_{V_+\cap V_-}\colon V_+\cap V_-\to p(V_+\cap V_-)$ is a quotient map
and thus an open map (being also a homomorphism). The assertion follows.\\[2.3mm]
(f) If $\alpha(V)=V$, then $V_+=V_-=V$ (whence $V=V_+V_-$) and $V_{--}$ is open (and thus closed);
so $V$ is tidy. If, conversely, $V$ is tidy, then every $x\in V$ lies on a bounded $\alpha$-trajectory $(x_n)_{n\in\Z}$
(as $x\in\lev(\alpha)$),
whence $x\in V_+\cap V_-$ by \cite[Proposition~11]{End}. Thus $V=V_+\cap V_-$
and thus $\alpha(V)=V$.
\end{proof}
\begin{proof}
(Theorem~F). If $\alpha$ has small tidy subgroups, then also~$\alpha_L:=\alpha|_{\lev(\alpha)}$,
by Corollary~\ref{subtidy}.
If, conversely, $\alpha_L$ has small tidy subgroups,
then
\begin{equation}\label{ifsmll}
\{e\}=\nub(\alpha_L)=\wb{\con(\alpha_L)\cap\lev(\alpha_L)}=\wb{\con(\alpha_L)},
\end{equation}
using Proposition~\ref{betterC}
for the first equality, Proposition~\ref{keynub}\,(e) for the second,
and the fact that $\lev(\alpha_L)=\lev(\alpha)$
for the last equality. Thus
\begin{equation}\label{new38}
\nub(\alpha)=\wb{\con(\alpha)\cap\lev(\alpha)}=\wb{\con(\alpha_L)}=\{e\},
\end{equation}
whence $\alpha$ has small tidy subgroups (by Proposition~\ref{betterC}).\\[2.3mm]
Assuming now that~$\alpha$ has small tidy subgroups, we have:\\[2.3mm]
(c)
Since $\ker(\alpha)\cap\parb^-(\alpha)\sub \bik(\alpha)\sub\nub(\alpha)=\{e\}$, the induced endomorphism
$\alpha|_{\parb^-(\alpha)}$ of $\parb^-(\alpha)$ is injective and hence an automorphism of the
topological group $\parb^-(\alpha)$ (by
Lemma~\ref{fi-steps}~(b)).
We know that the endomorphisms $\alpha|_{\con^-(\alpha)}$ and $\alpha|_{\lev(\alpha)}$ are surjective.
Being restrictions of the automorphism $\alpha|_{\parb^-(\alpha)}$,
also $\alpha|_{\con^-(\alpha)}$ and $\alpha|_{\lev(\alpha)}$ are automorphisms.\\[2.3mm]
(b) By (\ref{new38}), we have $\con(\alpha)\cap\lev(\alpha)=\{e\}$.
Together with Lemma~\ref{parblev} (a) and (d),
this entails that
$\parb(\alpha)=\con(\alpha)\rtimes \lev(\alpha)$ as an abstract group.
By Proposition~\ref{parblev}\,(d), the map $q\colon \lev(\alpha)\to\parb(\alpha)/\con(\alpha)$,
$g\mto g\,\con(\alpha)$ is an isomorphism of topological groups. As a consequence,
the homomorphism $\parb(\alpha)\to\lev(\alpha)$,
$xy\mto y$ (with $x\in\con(\alpha)$, $y\in\lev(\alpha)$)
is continuous and thus $\parb(\alpha)=\con(\alpha)\rtimes\lev(\alpha)$ as a topological group.\\[2.3mm]
As $\alpha$ has small tidy subgroups, so does the automorphism
$\beta:=\alpha|_{\parb^-(\alpha)}$ (see Corollary~\ref{subtidy}). Hence
\begin{eqnarray*}
\parb^-(\alpha)& =& \parb^-(\beta)=\parb(\beta^{-1})=\con(\beta^{-1})\rtimes \lev(\beta^{-1})\\
&=&
\con^-(\beta)\rtimes \lev(\beta)=
\con^-(\alpha)\rtimes \lev(\alpha).
\end{eqnarray*}
(d) Let $V\sub G$ be a compact open subgroup which is tidy for~$\alpha$. Then
\begin{eqnarray*}
V&=&V^{-1}=(V_-)^{-1}(V_+)^{-1}=V_-(V_+)^{-1}\\
&=&
(\con(\alpha)\cap V)(\lev(\alpha)\cap V)(\con^-(\alpha)\cap V)\sub \Omega,
\end{eqnarray*}
using Lemma~\ref{parblev} (b) and (c).
As each of $\con(\alpha)$ and $\con^-(\alpha)$ intersects
$\lev(\alpha)$ in~$\{e\}$ (by (b)),
we deduce from Lemma~\ref{parblev} (b) and (c) that
\begin{equation}\label{abstrgp}
V_-=(\con(\alpha)\cap V)\rtimes (\lev(\alpha)\cap V)\;\mbox{and}\;
V_+=(\con^-(\alpha)\cap V)\rtimes (\lev(\alpha)\cap V)
\end{equation}
as abstract groups. Using (b), we see that (\ref{abstrgp}) also holds in the sense
of topological groups.\\[2.3mm]
(a) By (d), $\Omega$ is an identity neighbourhood.
For $x\in G$, the translations $\lambda_x,\rho_x\colon G\to G$,
\[
\lambda_x(y):=xy,\quad\rho_x(y):=yx\quad\mbox{for $y\in G$}
\]
are homeomorphisms. If $z\in\Omega$,
write $z=xy$ with $x\in\con(\alpha)$ and $y\in\lev(\alpha)\con^-(\alpha)=\parb^-(\alpha)$ (cf.\
Lemma~\ref{parblev}\,(e)).
Since
\[
\Omega=\con(\alpha)\parb^-(\alpha),
\]
we have $(\rho_y\circ \lambda_x)(\Omega)\sub\Omega$.
Since $(\rho_y\circ \lambda_x)(\Omega)$ is a neighbourhood
of
\[
(\rho_y\circ\lambda_x)(e)=xy=z,
\]
so is~$\Omega$. Hence $\Omega$ is open.
If also $z=ab$ with $a\in \con(\alpha)$ and $y\in\parb^-(\alpha)$, then
\[
a^{-1}x=by^{-1}\in \con(\alpha)\cap \parb^-(\alpha)=\con(\alpha|_{\lev(\alpha)})=\{e\}
\]
(where we used~(\ref{ifsmll})). Hence $x=a$ and $y=b$, showing that $\pi$ is injective and
hence a bijection. It remains to show that $\pi$ is an open map.
If $V\sub G$ is a tidy subgroup for~$\alpha$ and
\[
L:=(\con(\alpha)\cap V)\times (\lev(\alpha)\cap V)\times (\con^-(\alpha)\cap V),
\]
then $\pi$ restricts to a continuous bijection
\begin{equation}\label{lhscomp}
\pi_V\colon
L \to V_-V_+=V,
\end{equation}
as a consequence of~(d). Since~$L$ is
compact, $\pi_V$ is a homeomorphism.
Hence $\pi|_L$ is an open map.
Given $a\in \con(\alpha)$, $b\in\lev(\alpha)$ and $c\in\con^-(\alpha)$,
we have
\[
\pi(x,y,z)=a\pi|_L(a^{-1}x,yb^{-1},bzc^{-1}b^{-1})bc
\]
for $(x,y,z)\in \con(\alpha)\times\lev(\alpha)\times\con^-(\alpha)$
in some open neighbourhood~$W$ of $(a,b,c)$. Since translations and the automorphism
$\con^-(\alpha)\to\con^-(\alpha)$, $v\mto bvb^{-1}$
are open maps, $\pi|_W$ is an open map and thus $\pi$ a homeomorphism.\\[2.3mm]
(e)
If $V$ is tidy, then $V$ is the product of the $\alpha$-invariant set $\con(\alpha)\cap V=\con(\alpha|_{V_-})$,
the $\alpha$-stable set $\lev(\alpha)\cap V=V_+\cap V_-$
and the set $\con^-(\alpha)\cap V=\con^-(\alpha)\cap V_+$
whose image under~$\alpha$ is contained in $\con^-(\alpha)$.
In view of (a), we now deduce from
\[
(V_+\cap V_-) (\con^-(\alpha)\cap V_+)=V_+\sub \alpha(V_+)=(V_+\cap V_-)\alpha(\con^-(\alpha)\cap V_+)
\]
that $\con^-(\alpha)\cap V_+\sub\alpha(\con^-(\alpha)\cap V_+)$.
Conversely, assume that (\ref{chart1})--(\ref{chart4}) are satisfied
by a compact open subgroup~$V\sub G$.
If $v\in V_-$, write $\pi^{-1}(v)=:(a,b,c)$.
As $\con(\alpha)$, $\lev(\alpha)$ and $\con^-(\alpha)$ are $\alpha$-invariant,
we deduce from
\[
\alpha^n(v)=\alpha^n(a)\alpha^n(b)\alpha^n(c)\in V
\]
that $\alpha^n(c)\in \con^-(\alpha)\cap V$ for all $n\in\N_0$,
whence $c\in \parb(\alpha)\cap\con^-(\alpha)=\con(\alpha)\lev(\alpha)\cap \con^-(\alpha)=\{e\}$
and thus $c=e$. We now easily deduce that
\[
V_-=(\con(\alpha)\cap V)(\lev(\alpha)\cap V).
\]
If $v\in V_+$ and $(a,b,c):=\pi^{-1}(v)$, let $(v_n)_{n\in\N_0}$
be an $\alpha$-regressive trajectory for~$v$ in~$V$.
Write $(a_n,b_n,c_n):=\pi^{-1}(v_n)$.
Then $(a_n)_{n\in\N_0}$ is an $\alpha$-regressive trajectory
for~$a$ in~$V$, whence $a\in\con(\alpha)\cap \parb^-(\alpha)=\con(\alpha)\cap\lev(\alpha)\sub\nub(\alpha)=\{e\}$
and thus $a=e$. We now easily deduce that
\[
V_+=(\lev(\alpha)\cap V)(\con^-(\alpha)\cap V).
\]
Using that $\lev(\alpha)\cap V$ is $\alpha$-stable
and $\con^-(\alpha)\sub V_{++}$ (cf.\ Lemma~\ref{modVpVm}), we find
\[
V_{++}=(\lev(\alpha)\cap V)\con^-(\alpha).
\]
As $\con^-(\alpha)$ is closed in~$G$ and $\lev(\alpha)\cap V$ is compact,
$V_{++}$ is closed in~$G$.
Since $\beta:=\alpha|_{\parb^-(\alpha)}$ is an automorphism,
the indices $[\alpha^{n+1}(V_+):\alpha^n(V_+)]=[\beta^{n+1}(V_+):\beta^n(V_+)]$
are independent of $n\in\N_0$. Hence $V$ is tidy.\\[2.3mm]
(f) As $\lev(\alpha)$ has small tidy subgroups, \hspace*{-.3mm}(f) follows from Lemma~\ref{parblev}\,(f).
\end{proof}

\begin{thebibliography}{99}\itemsep+.8pt
%
%
\bibitem{BaW}
Baumgartner, U. and G.\,A. Willis,
\emph{Contraction groups and scales of automorphisms of totally disconnected locally compact groups},
Isr.\ J. Math.\ {\bf 142} (2004), 221--248.
%
%
\bibitem{BouTop}
Bourbaki, N., ``General Topology I", Chapters 1--4, Springer, Berlin, 1988.
%
%
\bibitem{Bou}
Bourbaki, N., ``Lie Groups and Lie Algebras",
Chapters 1--3, Springer, Berlin, 1989.
%
%
%
\bibitem{DaS}
Dani, S.\,G. and R. Shah,
\emph{Contraction subgroups and semistable measures on $p$-adic Lie groups},
Math.\ Proc.\ Camb.\ Philos.\ Soc.\ {\bf 110} (1991), 299--306.
%
%

\bibitem{GBV}
Giordano Bruno, A. and S. Virili,
\emph{Topological entropy in totally disconnected locally compact groups},
Ergodic Theory Dyn. Syst. {\bf 37(7)} (2017), 2163-2186.
%
%
\bibitem{Tid}
Gl\"{o}ckner, H.,
\emph{Contraction groups for tidy automorphisms of totally disconnected groups},
Glasg.\ Math.\ J. {\bf 47} (2005), 329-333.
%
%
\bibitem{MIX}
Gl\"{o}ckner, H.,
\emph{Locally compact groups built up from $p$-adic
Lie groups, for $p$ in a given set of primes},
J. Group Theorey {\bf 9} (2006), 427--454.
%
%
\bibitem{CON}
Gl\"{o}ckner, H.,
\emph{Contractible Lie groups over local fields},
Math.\ Z. \textbf{260}
(2008), 889--904.
%
%
%
\bibitem{Lie}
Gl\"{o}ckner, H., \emph{Endomorphisms of Lie groups
over local fields}, preprint, {\tt arXiv:1701.01804v2}.
%
%
\bibitem{GaW}
Gl\"{o}ckner, H.
and G. A. Willis,
\emph{Classification of the simple factors appearing
in composition series of totally disconnected
contraction groups},
J. Reine Angew.\ Math.\
\textbf{634} (2010), 141--169.
%
%
\bibitem{HaR}
Hewitt, E. and K.\,A. Ross,
``Abstract Harmonic Analysis," Vol. 1, Springer, Berlin, ${}^2$1979.
%
%
\bibitem{Isb}
Isbell, J.\,R., ``Uniform Spaces," AMS, Providence, 1964.
%
%
\bibitem{Jaw}
Jaworski, W.,
\emph{On contraction groups of automorphisms of totally disconnected locally compact groups},
Isr.\ J. Math.\ {\bf 172} (2009), 1--8.
%
%
\bibitem{SaR}
Raja, C.\,R.\,E. and R. Shah,
\emph{Some properties of distal actions on locally compact groups},
preprint, {\tt arXiv.1201.4287v2}.
%
%
\bibitem{Rei}
Reid, C.\,D.,
\emph{Endomorphisms of profinite groups},
Groups Geom.\ Dyn.\ {\bf 8} (2014), 553--564.
%
%
\bibitem{Re2}
Reid, C.\,D., \emph{Dynamics of flat actions on totally disconnected,
locally compact groups}, New York J. Math.\ {\bf 22} (2016), 115--190.
%
%
\bibitem{Rob} Robinson, D.\,J.\,S.,
``A Course in the Theory of Groups,''
Springer, New York, $^{2}$1996.
%
%
\bibitem{Ser}
Serre, J.-P., ``Lie Algebras and Lie Groups,"
Springer, Berlin, 1992.
%
%
\bibitem{Sie}
Siebert, E., \emph{Contractive automorphisms
on locally compact groups},
 Math.\ Z.
\textbf{191} (1986), 73--90.
%
%
\bibitem{vDa}
van Dantzig, D.,
\emph{Studien over topologische Algebra},
University of Groningen (1931).
%
%
\bibitem{Wan}
Wang, J.\,S.\,P., \emph{The Mautner phenomenon for $p$-adic
Lie groups},
Math.\ Z. \textbf{185}
(1984), 403--412.
%
%
\bibitem{Wil}
Willis, G.\,A.,
\emph{The structure of totally disconnected, locally compact groups},
Math.\ Ann.\ {\bf 300} (1994), 341--363.
%
%
\bibitem{Fur}
Willis, G.\,A.,
\emph{Further properties of the scale function on a totally disconnected group},
J. Algebra {\bf 237} (2001), 142--164.
%
%
\bibitem{WJO}
Willis, G.\,A.,
\emph{Tidy subgroups for commuting
automorphisms of totally disconnected locally compact groups:
An analogue of simultaneous triangularisation of matrices},
New York J. Math.\ {\bf 10} (2004), 1--35.
%
%
\bibitem{Nub}
Willis, G.\,A.,
\emph{The nub of an automorphism of a totally disconnected, locally compact group},
Ergodic Theory Dyn.\ Syst.\ {\bf 34} (2014), 1365--1394.
%
%
\bibitem{End}
Willis, G.\,A,
\emph{The scale and tidy subgroups for endomorphisms of totally disconnected locally compact groups},
Math.\ Ann.\ {\bf 361} (2015), 403--442.\vspace{1.4mm}
%
%
\end{thebibliography}
\end{document}